\documentclass[11pt]{article}
\usepackage{amsmath,amssymb, mathtools,a4wide} 
\usepackage{amscd} 
\usepackage{amsthm} 
\usepackage{dsfont} 
\usepackage{stmaryrd} 

\usepackage{float} 

\usepackage[pdftex]{graphicx} 
\usepackage{xcolor} 
\usepackage[colorlinks,citecolor=blue,linkcolor=blue,urlcolor=blue]{hyperref} 

\usepackage{geometry} 
\usepackage[utf8]{inputenc} 
\usepackage[english]{babel} 

\usepackage{tikz} 
\usepackage{cite} 

\newtheorem{thm}{Theorem}
\newtheorem{conj}[thm]{Conjecture}
\newtheorem{lemma}[thm]{Lemma}
\newtheorem{coro}[thm]{Corollary}
\newtheorem{definition}[thm]{Definition}
\newtheorem{prop}[thm]{Proposition}
\newtheorem{example}[thm]{Example}
\newtheorem*{Remark}{Remark}
\newtheorem{exo}[thm]{Exercise}
\newtheorem{oq}[thm]{Open question}
\newtheorem*{goal}{Goal}

\numberwithin{thm}{section} 
\numberwithin{equation}{section} 
\numberwithin{figure}{section} 

\newcommand*{\R}{\mathbb{R}}
\newcommand*{\C}{\mathbb{C}}

\newcommand*{\Z}{\mathbb{Z}}
\renewcommand*{\H}{\mathbb{H}}

\renewcommand*{\S}{\Sigma}
\newcommand*{\g}{\mathfrak{g}}

\newcommand*{\e}{\varepsilon}
\renewcommand*{\l}{\lambda}
\newcommand*{\mf}{\mathfrak}
\newcommand*{\mc}{\mathcal}
\newcommand*{\bb}{\mathbb}
\newcommand*{\del}{\partial}
\newcommand*{\delbar}{\bar\partial}
\newcommand*{\cotang}{T^*\mc{T}^n}
\newcommand*{\T}{\mc{T}}

\DeclareMathOperator{\Rep}{Rep}
\DeclareMathOperator{\Hol}{Hol}
\DeclareMathOperator{\Hom}{Hom}
\DeclareMathOperator{\GL}{GL}
\DeclareMathOperator{\SL}{SL}
\DeclareMathOperator{\PSL}{PSL}
\DeclareMathOperator{\tr}{tr}
\DeclareMathOperator{\rk}{rk}
\DeclareMathOperator{\id}{id}
\DeclareMathOperator{\Fun}{Fun}
\DeclareMathOperator{\End}{End}

\begin{document}

\title{A Gentle Introduction to the Non-Abelian Hodge Correspondence}
\date{}

\maketitle

\vspace{-1.5cm}
\begin{center}
Alexander Thomas\footnote{Max-Planck Institute Bonn, athomas@mpim-bonn.mpg.de, Basic Research Community for Physics \href{https://basic-research.org/}{BRCP}}
\end{center}

\vspace{0.25cm}
\begin{abstract}
We aim at giving a pedagogical introduction to the non-abelian Hodge correspondence, a bridge between algebra, geometric structures and complex geometry. The correspondence links representations of a fundamental group, the character variety, to the theory of holomorphic bundles. 

We focus on motivations, key ideas, links between the concepts and applications. Among others we discuss the Riemann--Hilbert correspondence, Goldman's symplectic structure via the Atiyah--Bott reduction, the Narasimhan--Seshadri theorem, Higgs bundles, harmonic bundles and hyperkähler manifolds. 
\end{abstract}

\tableofcontents

\section{Introduction}

This paper is an introduction to the non-abelian Hodge correspondence, focusing on the key principles, motivations and links to other areas. Technical details and computations are mostly referred to references, while the ideas and concepts are presented. No new results are presented, but the way of presentation is original.

We base our exposition on geometric and sometimes on physical intuition, with some emphasis on symplectic geometry. The paper should be useful to get a first glimpse on the topic or to step back from technical details to see the clear conceptual picture. It should be accessible to a broad audience, in particular to master students in mathematics with interest in mathematical physics. Some knowledge of differential geometry, Riemann surfaces and Lie groups are welcome, but not a must have. 

The non-abelian Hodge correspondence is the huge achievement due to many mathematicians, above all Nigel Hitchin \cite{hitchin1987self}, Carlos Simpson \cite{simpson}, Kevin Corlette \cite{corlette} and Simon Donaldson \cite{donaldson}. It links three worlds together: the topological and algebraic world of representations of fundamental groups, the differential geometry world of connections and the complex geometry world of holomorphic bundles. It can be interpreted as a diffeomorphism between moduli spaces, which play an important role in theoretical physics. The correspondence is an incarnation of a very strong structure on these moduli spaces, a hyperkähler structure.

The paper is structured as follows:
\begin{itemize}
	\item \emph{Introduction and motivations:} Section \ref{sec:char} presents the multiple facets of the character variety. In Section \ref{bundle-formalism} we review bundles and connections and describe the Riemann--Hilbert correspondence between flat connections and the character variety.
	\item \emph{The toolbox:} an important topic is how to define quotients. One way is the hamiltonian reduction, which is presented together with a crash course in symplectic geometry in Section \ref{sec:symp-geo}. As an application, Section \ref{atiyah--bott} constructs the Goldman symplectic structure on the character variety. Another way uses stability conditions exposed in Section \ref{Sec-GIT} and applied to holomorphic bundles in Section \ref{sec:holo-bundles}.
	\item \emph{The core:} Section \ref{Sec-Higgs-bundles} introduces the notion of a Higgs bundle and states the non-abelian Hodge correspondence. The main ideas for the proof are given in Section \ref{harmonic-bundles} via harmonic map theory and the Hitchin--Simpson theorem. Section \ref{HK} gives a deeper understanding coming from hyperkähler geometry.
	\item \emph{Applications and generalizations:} In Section \ref{hit-comp} we construct Hitchin components as an application of the non-abelian Hodge correspondence. The final section \ref{sec:generalizations} exposes generalizations and research directions.
\end{itemize}

Some paragraphs are marked with an asterisk. They are more advanced and not necessary for the other (non-advanced) topics.
I used the material of this paper for a master lecture at the University of Bonn in the summer semester 2022.

\medskip
\paragraph{Acknowledgments.}
I warmly thank Vladimir Fock, from whom I learned most of the ideas and concepts. I'm also grateful to Florent Schaffhauser and Georgios Kydonakis for helpful discussions.
I gratefully acknowledge support from the Max-Planck Institute for Mathematics in Bonn.


\section{Starting point: Character varieties}\label{sec:char}

Character varieties are at the crossroad between many fields: representation theory, geometry, theoretical physics, dynamics, number theory,... They form a playground where techniques from various fields can be applied, and at the same time they have deep connections to multiple research streams.

\medskip
\paragraph{Motivation.}
Consider a group $\Gamma$ which we want to understand. The natural way to understand a group is to let it act. On the other hand, the theory we understand the best in mathematics is probably Linear Algebra. So we can try to let $\Gamma$ act on a vector space $V$, say of dimension $n$ defined over $\C$. In other words: we try to find matrices which mimic $\Gamma$. This is the basic idea of representation theory. Hence we consider $\Hom(\Gamma,\GL_n(\C))$.

As usual, we are not really interested in the set of all representations, but only in the isomorphism classes. For a linear action, any two representations into $\GL_n(\C)$ are isomorphic whenever we can obtain one from the other by a simple base change in $V$. Hence, the space of isomorphism classes is
$$\Hom(\Gamma,\GL_n(\C))/\GL_n(\C):= \Rep(\Gamma,\GL_n(\C))$$
where we quotient by the conjugation action.
Note that in order to get a nice topological space (separable), we restrict to representations $\Gamma\to \GL_n(\C)$ which are completely reducible. We come to this in Section \ref{Sec-GIT}.

Now, we can consider vector spaces with more structure, for example equipped with a hermitian product, a symplectic structure... Then we ask for representations preserving this structure. Thus, we analyze the space $\Rep(\Gamma,G)$ for some Lie group $G$, typically a subgroup of $\GL_n(\C)$.

In the specific case when $\Gamma$ is the fundamental group of a manifold $M$, we call $\Rep(\pi_1M,G)$ the \textbf{character variety} of $M$ and $G$.

The case if a fundamental group of a manifold is interesting because the character variety has several geometric meanings: it describes $(G,X)$-structures (geometric structures where the manifold $M$ is locally modeled on some space $X$ with transition functions in $G$) and flat $G$-connections. The link to flat connections is explained in Section \ref{bundle-formalism} via the Riemann--Hilbert correspondence.

For $M=\Sigma$ a surface, the character variety appears in physics, especially in string theory. It describes geometric structures on the world sheet, the surface traced out by a string in time. In mathematics, representations of surface groups into real groups ($G=\SL_n(\R)$ for example) have interesting dynamical properties. This goes under the name of \emph{higher Teichmüller theory} (see Section \ref{hit-comp} for more details).

\medskip
\paragraph{Many viewpoints.}
Since the character variety sits at the intersection of many mathematical areas, it allows many equivalent descriptions. The main goal of the paper is to understand all these incarnations.

\medskip
\textit{\textbf{Topological interpretation:}}
The character variety by its very definition is a space of representations of $\pi_1\S$ which does only depend on the topology of $\S$.

\medskip
\textit{\textbf{Smooth interpretation:}}
The character variety is also described as the space of all flat connections (on a trivial bundle over $\S$) modulo gauge equivalence. This description is called the \emph{Riemann--Hilbert correspondence}, which we will see at the end of Section \ref{bundle-formalism}.

\medskip
\textit{\textbf{Holomorphic interpretation:}}
The character variety can be described by holomorphic objects on a Riemann surface $S$ whose underlying smooth surface is $\S$. These objects are \emph{stable Higgs bundles}, which we introduce in Section \ref{Sec-Higgs-bundles}. This description is the content of the \emph{non-abelian Hodge correspondence}.

\section{Bundles and connections}\label{bundle-formalism}

To study the character variety, its link to the space of flat connections, the ``smooth interpretation'', is fundamental. We motivate and recall the basic concepts of differential geometry and show the link to character varieties.
To deepen the subject, I highly recommend the outstanding book of Baez and Muniain \cite{baez}.

\medskip
\paragraph{Dictionary.}
Most of modern physical theories, like Maxwell's theory of electromagnetism or the standard model of elementary particles, are using concepts of differential geometry such as bundles and connections. Why are these concepts so fundamental? 

To me, the reason comes from what I like to call the \emph{Global-to-Local heuristics}, the idea that our observations describe only local properties of our universe. When we try to describe natural phenomena which surround us, we get the impression that we live in a Euclidean space. Indeed, the forces applied to a same point add up in a vectorial way, and even on a larger scale like our solar system, nature seems to be well described by Newtonian mechanics in which the universe is an affine space modeled on $\mathbb{R}^3$. But our perceptions and observations are always limited in space and time. Nothing prevents Nature to behave only locally like an affine space, but to bend and twist on a global scale.

A good illustration is the surface of our Earth, which locally is well described by a flat part of the plane $\mathbb{R}^2$ (although there are mountains and valleys, on average it seems to be flat). Globally of course, the Earth is a ball, since it bends far away from the observation scales of our daily life.

The ``Global-to-Local heuristics'' can be summarized in the following dictionary:
\vspace{0.3cm}
\begin{center}
\setlength{\arrayrulewidth}{.05em}
\begin{tabular}{|c|c|}
\hline 
Linear algebra / Classical mechanics & Riemannian geometry / General relativity \\
\hline
Absolute space $\R^n$ &   Manifold \\ 
 Functions &  Sections of a bundle\\ 
Differential equations &  Connections \\
\hline
\end{tabular}
\end{center}
\vspace{0.5cm}

To the global concept of absolute space corresponds a mathematical concept which looks only locally like an open subset of $\mathbb{R}^n$: the concept of a \textbf{manifold}. 

In Newtonian mechanics, there are ingredients other than space: physical quantities are described by functions (temperature, speed, electromagnetic field, ...) whose evolution is described by differential equations. \emph{What is the local concept of a function?} I.e. what is the mathematical concept which looks locally like a function, but which might bend on large scales?

The reader who has never thought about that question should take a second to think about it. The answer is a bit tricky: a ``generalized function'' is a section of a fiber bundle! Consider a function $f:M \to F$ where $M$ is a manifold and $F$ is some target space ($F=\mathbb{R}$ for the temperature for example). An equivalent way to describe $f$ is through its graph $\mathrm{gr}(f)$ which is the subset of $M\times F$ given by all points of the form $(x,f(x))$. We can also say that a function is the choice of an inverse to the projection $M\times F \to M$. 

To generalize the notion of a function, we first need a space $E$ which locally looks like $U\times F$ where $U\subset M$ is a small open subset of $M$. This is precisely the notion of a \textbf{fiber bundle} on $M$ with fiber $F$. When $F$ is a vector space, we speak of a \textbf{vector bundle}. 
To be precise, a vector bundle with fiber $V$ is a manifold $E$ together with a surjective map $p:E\to M$ such that there exists an atlas on $M$ with charts $U_i$ such that $p^{-1}(U_i) = U_i \times V$ and transition maps which are linear and linear, i.e. in $\mathrm{GL}(V)$. You can imagine to built up $E$ by taking $U_i\times V$ and identifying them on $(U_i\cap U_j)\times V$ using the transition functions of $M$ and a fiberwise element of $\mathrm{GL}(V)$.

A ``generalized function'' is a \textbf{section} of $p$, i.e. a map $s: M\to E$ such that $p\circ s=\mathrm{id}_M$. The space of all sections is denoted by $\Gamma(E)$. For the trivial bundle $E=M\times F$, a section is nothing but a function on $M$ with values in $F$. For $E=TM$ the tangent bundle, a section is a vector field.

\begin{figure}[h!]
\centering
\begin{tikzpicture}[scale=1.2]
\draw (0.5, 0.55)--(-0.5+0.1,1.55-0.2)--(0.5,2.55)--(1.5-0.1,1.55+0.2)--cycle;
\draw (3, -0.45)--(2+0.1,0.55-0.2)--(3,1.55)--(4-0.1,0.55+0.2)--cycle;
\draw[dotted] (0.5,0.55)..controls +(1,0) and +(-1,0)..(3, -0.45);
\draw[thick] (0.5,1.55)..controls +(0.8,0.1) and +(-0.6,0.2)..(3, 0.55);
\draw (2,1.6) node {$D$};
\draw (0.2,1.6) node {$s$};
\draw (-0.3,0.1) node {$M$};
\draw[thick] (0,0)--(1,1.1);
\draw[thick] (0,0)..controls (1,0) and +(-1,0)..(2.5,-1);
\draw[thick] (2.5,-1)--+(1,1.1);
\draw[thick] (1,1.1)..controls +(1,0) and +(-1,0)..(3.5,0.1);
\end{tikzpicture}

\caption{Illustration of a bundle and a connection}
\end{figure}
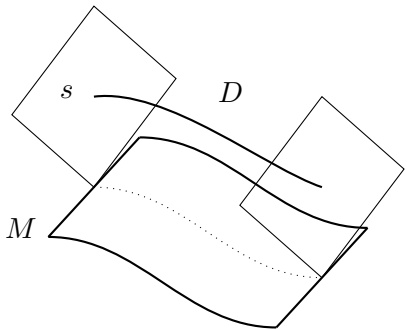

\begin{Remark}
The ``Global-to-Local heuristics'' should have a counterpart to the microscopic level: what we observe in our daily life is only at a \emph{mesoscopic} scale. Nothing prevents Nature to be different on an atomic scale. The search for a mathematical concept which gives a Euclidean structure at mesoscopic scales (and a manifold on the global scale) is still open. Candidates exist, for example discrete models or non-commutative geometry. This concept should play a paramount role in a hypothetical future theory of quantum gravity.
\end{Remark}

\medskip
\paragraph{Connections.}
What is the local concept of a linear differential equation? How to define the derivative of a section of a vector bundle? 
The answer is given by the notion of a connection.

In a first approximation, we can say that a connection is a generalization of a directional derivative, i.e. it allows to derive a section $s$ in a direction given by a vector field $X$. We write the result as $\nabla_X(s)$, which is again a section.

In a second approximation, we can say that a connection is a \emph{matrix-valued 1-form}. Locally it can be written $d+A$ where $A\in \Omega^1(M,\mathfrak{gl}_n(\C))$. To see this, fix basis vectors $e^i(x)$ in each fiber, varying smoothly in $x$. Then a section $s$ can be written $s=\sum_i s_ie^i(x)$. Locally, we can differentiate $s$ in the direction $X=\sum_j X_j\frac{\partial}{\partial x^j}$ in the usual way:
$$d_X(s) = \sum_i d_X(s_i)e^i(x) + \sum_{i,j} s_iX_j d_j(e^i(x))$$
where $d_j$ denotes the derivative in direction $x^j$. Since $(e^i(x))$ is a basis, there are matrices $A_j(x)$ defined by
$$d_j(e^i(x)) = A_j(x)e^i(x).$$
Putting $A=\sum_jA_jdx^j$, we can write $d_X(s) = (d+\sum_i A_i dx^i)_X(s)$, so the connection is $d+A$. Note that the matrix-valued 1-form $A$ appears from the fact that the basis $(e^i(x))$ depends on $x$.


To be precise, a \textbf{connection} (or \textbf{covariant derivative}) is a map $D: \Gamma(E)\times \Gamma(TM) \to \Gamma(E)$, where $E$ is a vector bundle, which satisfies $\forall s,t \in \Gamma(E), f,g \in \mathcal{C}^\infty(M)$ and $X,Y \in \Gamma(TM)$:
\begin{align*}
(i) & \text{ Linearity for sections: } & D_X(s+t) = D_X(s) + D_X(t), \\
(ii) & \text{ Linearity for vector fields: } & D_{fX+gY}(s) = fD_X(s) + gD_Y(s), \\
(iii) & \text{ Leibniz' rule: } & D_X(fs) = df(X)\, s+ f D_X(s). 
\end{align*}
The name ``connection'' comes from the fact that a connection allows to connect different fibers of $p: E\to M$. A section $s$ is said to be \textbf{flat} if $D_X(s) = 0$ for all $X\in \Gamma(TM)$.


To see the link to differential equations, consider first the case of a 1-dimensional manifold, for example $M=\R$. A linear differential equation is given by 
$$(d^n+t_1(x)d^{n-1}+t_2(x)d^{n-2}+...t_n(x))\psi(x)=0$$
where $d=\frac{d}{dx}$.

This is equivalent to a matrix-valued differential equation of order 1: $(d+A)s=0$ where
$$A = \begin{pmatrix} &-1&&\\ && \ddots &\\ &&&-1\\ t_n & t_{n-1} &\cdots & t_1\end{pmatrix} \;\text{ and }\; s=\begin{pmatrix}\psi\\ d\psi \\ \vdots \\ d^{n-1}\psi\end{pmatrix}.$$

In dimension 2, you get a system of two differential equations. You can write it as
$$\left \{ \begin{array}{cl}
D_1\psi(x,y) &=0  \\
D_2\psi(x,y) &=0  
\end{array} \right.$$

As in dimension 1, there is a standard form:
$$\left \{ \begin{array}{cl}
\partial_x \Psi(x,y) &=A_x(x,y)\Psi(x,y)  \\
\partial_y \Psi(x,y) &=A_y(x,y)\Psi(x,y)
\end{array} \right.$$
where $\Psi$ is a vector whose entries are suitable derivatives of $\psi$.

A natural question then arises: under which conditions there is a full set of solutions?

A necessary condition, which turns out to be sufficient, is that $\partial_x\partial_y\Psi = \partial_y\partial_x\Psi$ which gives
\begin{equation}\label{comp-cond-diff-sys}
\partial_xA_y-\partial_yA_x+[A_x,A_y]=0.
\end{equation}

This expression is the \emph{curvature of the connection $d+A$}, as defined below. In terms of the differential operators $D_1$ and $D_2$, if the curvature does not vanish, then you can reduce the system to a smaller one\footnote{Explicitly to $D_1\psi=0$ and $([D_1,D_2] \mod \langle D_1,D_2\rangle)\psi=0$}.

\begin{exo}
Work through an explicit example. For instance $D_1=\partial_x^2-y\partial_x$ and $D_2=\partial_y^2-x$.
\end{exo}

The \textbf{curvature} of a connection $D$ measures the failure of the covariant derivatives $D_i$ to commute. It is a 2-tensor $F_D$ given by
$$F_D(X,Y) = [D_X,D_Y]-D_{[X,Y]} \;\; \text{ for } X,Y\in \Gamma(TM).$$
The fact that it is a tensor means that $F_D(X,Y)(fs) = fF_D(X,Y)(s)$ for $f\in \mathcal{C}^\infty(M)$. If the curvature vanishes, we call the connection \textbf{flat}.

Locally we can write $D=d+A$ and take for $X$ and $Y$ coordinate vector fields. Then we get $$F_D(\partial_i,\partial_j) = [D_i, D_j]-D_{[\partial_i,\partial_j]} = [D_i,D_j] = [\partial_i+A_i, \partial_j+A_j] = \partial_i A_j-\partial_j A_i+[A_i,A_j].$$
We get Equation \eqref{comp-cond-diff-sys}. This can be written concisely as 
\begin{equation}\label{curvature-expr}
F(A) = dA+A\wedge A
\end{equation}
which is a matrix-valued 2-form.

\medskip
\paragraph{Gauge transformations and structure group.} 
The natural symmetry group acting on bundles are bundle automorphisms, also called gauge transformations.

When working with a vector bundle locally, we often fix a basis $(e_1(x), ..., e_r(x))$ in each fiber varying smoothly with $x$. Changing this basis by a matrix $A(x)$ is a \textbf{gauge transformation}, or just \textbf{gauge} in short. It is a smooth map from $M$ to $\mathrm{GL}(V)$ where $V$ is the fiber of the vector bundle $E$. Mathematically a gauge transformation is nothing but a \emph{bundle automorphism}, i.e. an invertible map $E\to E$ which preserves the fibers where it acts linearly.

How do connections and the curvature behave under gauge transformations? Since a gauge transformation acts fiberwise, we can work locally and write $D=d+A$.
\begin{prop}
The gauge action on connections is given by: $$g.A=gAg^{-1}+gd\left(g^{-1}\right).$$
On the curvature tensor we get: $$g.F(A) = gF(A)g^{-1}.$$
\end{prop}

The first point simply comes from the conjugation action of a gauge on a connection:
$$g(d+A)g^{-1} = d+gAg^{-1}+gd\left( g^{-1}\right).$$
Hence the action of $g$ on $A$ is given by $g.A = gAg^{-1}+gd\left( g^{-1}\right)$ which is an affine transformation. Indeed, the \textbf{space of all connections} $\mathcal{A}(E)$ is an affine space, whose underlying vector space is $\Omega^1(M,\mathfrak{gl}(V))$. This means that the difference of two connections $D-D_0$ is a matrix-valued 1-form $A$. Hence, while we can locally write $D=d+A$, we can always globally write $$D=D_0+A.$$
\begin{exo}
Check the gauge transformation of the curvature tensor using $F(A)=dA+A\wedge A$.
\end{exo}

In general, the transition functions of a bundle are any elements of $\mathrm{GL}(V)$, but it might happen that they all are in some subgroup $G\subset \mathrm{GL}(V)$. In that case, we say that $V$ is a bundle with \textbf{structure group} $G$. For example, if our vector space $V$ is equipped with a scalar product, we might require the transition functions to respect this structure, i.e. $G=\mathrm{O}(V)$. Then, we have a well-defined scalar product in each fiber of $E$. \textit{The structure group allows to put more structure on the bundle}, hence its name.
There is also a more abstract notion, that of a \emph{principal bundle}, which we will not treat in these lectures. 

For a more detailed treatment of bundles and connections and their links to physics, we warmly recommend the book of John Baez and Javier Munian \cite{baez}, in particular Chapter 2.

\medskip
\paragraph{Parallel transport.}
We have already said that a connection allows to connect different fibers. We make this precise.

Consider two points $x,y\in M$ and a path $\gamma:[0,1]\to M$ connecting them, i.e. $\gamma(0)=x, \gamma(1)=y$. Along $\gamma$ there is a unique flat section $s_p$ with given starting point $p\in \pi^{-1}(x)$, i.e. $s_p(x)=p$.

The map $p\mapsto s_p(y)$ is called \textbf{parallel transport} along $\gamma$. It is a linear map (element in the structure group $G$).
In general, the parallel transport depends on the path. \emph{For a flat connection, the parallel transport only depends on the homotopy class of the path} (i.e. you can change $\gamma$ by isotopies).

For flat connections, the parallel transport along loops is called the \textbf{monodromy}. The monodromy is an element of the character variety $\Rep(\pi_1M,G)$. Indeed, for loops based at $x$, we get a point in $\Hom(\pi_1(M,x),G)$ and changing $x$ conjugates the monodromy.

\medskip
\paragraph{Riemann--Hilbert correspondence.}
The first fundamental result about character varieties is that they describe the moduli space of flat connections, i.e. the space of flat connections modulo gauge equivalence. This gives a link between the topological and the smooth interpretation of character varieties.

We have just seen how to associate a point in the character variety to a flat connection, simply by considering its monodromy. This map is a diffeomorphism:
\begin{thm}[Riemann--Hilbert correspondence]
The character variety is the space of flat $G$-connections on the trivial bundle on $M$ modulo gauge:
$$\boxed{\Rep(\pi_1M,G)\cong \{\text{flat } G-\text{connections}\}/\text{gauge} .}$$
\end{thm}

The idea of the proof is to construct explicitly the inverse map: given a representation $\rho:\pi_1M \to G$, we can consider the diagonal action of $\pi_1M$ on $\widetilde{M}\times G$, where $\widetilde{M}$ denotes the universal cover of $M$ on which $\pi_1M$ acts by deck transformations. The quotient $$E_\rho=(\widetilde{M}\times G)/\pi_1M$$ is a (principal) $G$-bundle over $M$ and the trivial connection $d$ on $\widetilde{M}\times G$ descends to a flat $G$-connection. Finally, changing $\rho$ by conjugation corresponds to changing the connection by a gauge transformation.


\section{A crash course in symplectic geometry}\label{sec:symp-geo}

Grown out of modern treatments of classical mechanics in the early 19th century, symplectic geometry is a very active mathematical domain today. We give motivations and introduce basic concepts, in particular the symplectic quotient called Hamiltonian reduction. This will be used for the study of character varieties in the next section.

To deepen the subject, I warmly recommend the book of McDuff and Salamon \cite{mcduff} (especially Sections 1, 3 and 5), and the book of Kirillov \cite{kirillov} (in particular Chapter 1, Section 4 and Appendix II, Section 3). For the physical interpretation, I advocate Arnold's classical book \cite{arnold}.

\subsection{Symplectic structures}
The physical motivations for symplectic structures are sometimes a bit obscure. We will describe how they naturally arise.

Consider a physical system, for example the motion of a particle which is restricted to stay on some surface $S$. The first important idea is to consider the space of all possible states of our system, which is called the \textbf{phase space} $M$. In our example the system is uniquely given by knowing the position and the momentum of the particle, so the phase space is the cotangent bundle\footnote{It turns out that while the velocity lives in the tangent bundle, the momentum lives in the cotangent bundle.} $M=T^*S$. The picture you might have in mind is the following: we replace a complicated system by one point in a complicated space, which describes all possible states. Then the evolution of the system is nothing but a path in the phase space (see left of Figure \ref{Fig:phase-space}).

\begin{figure}[h!]
\centering
\includegraphics[height=2.8cm]{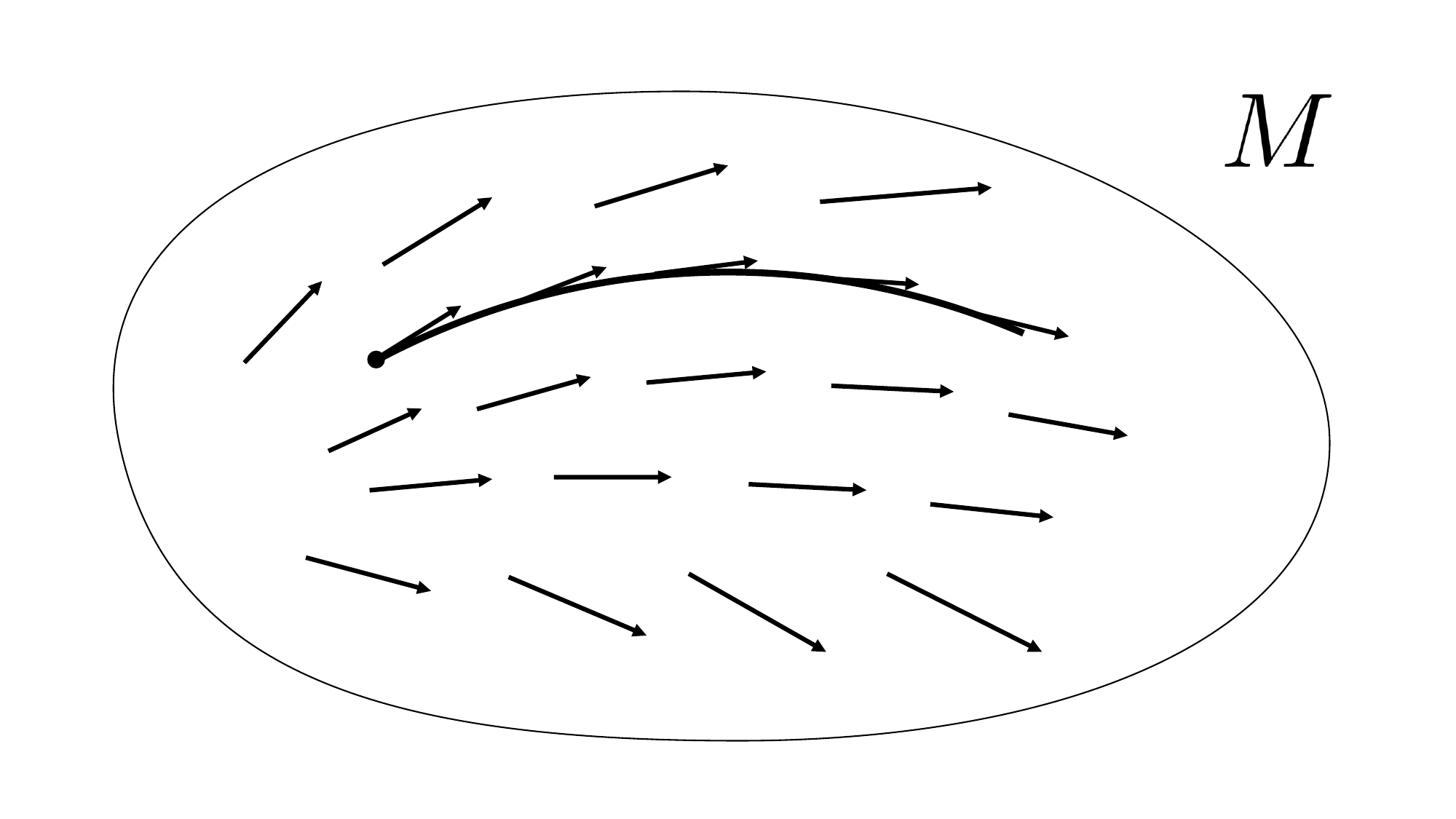} \hspace{0.3cm}
\includegraphics[height=2.8cm]{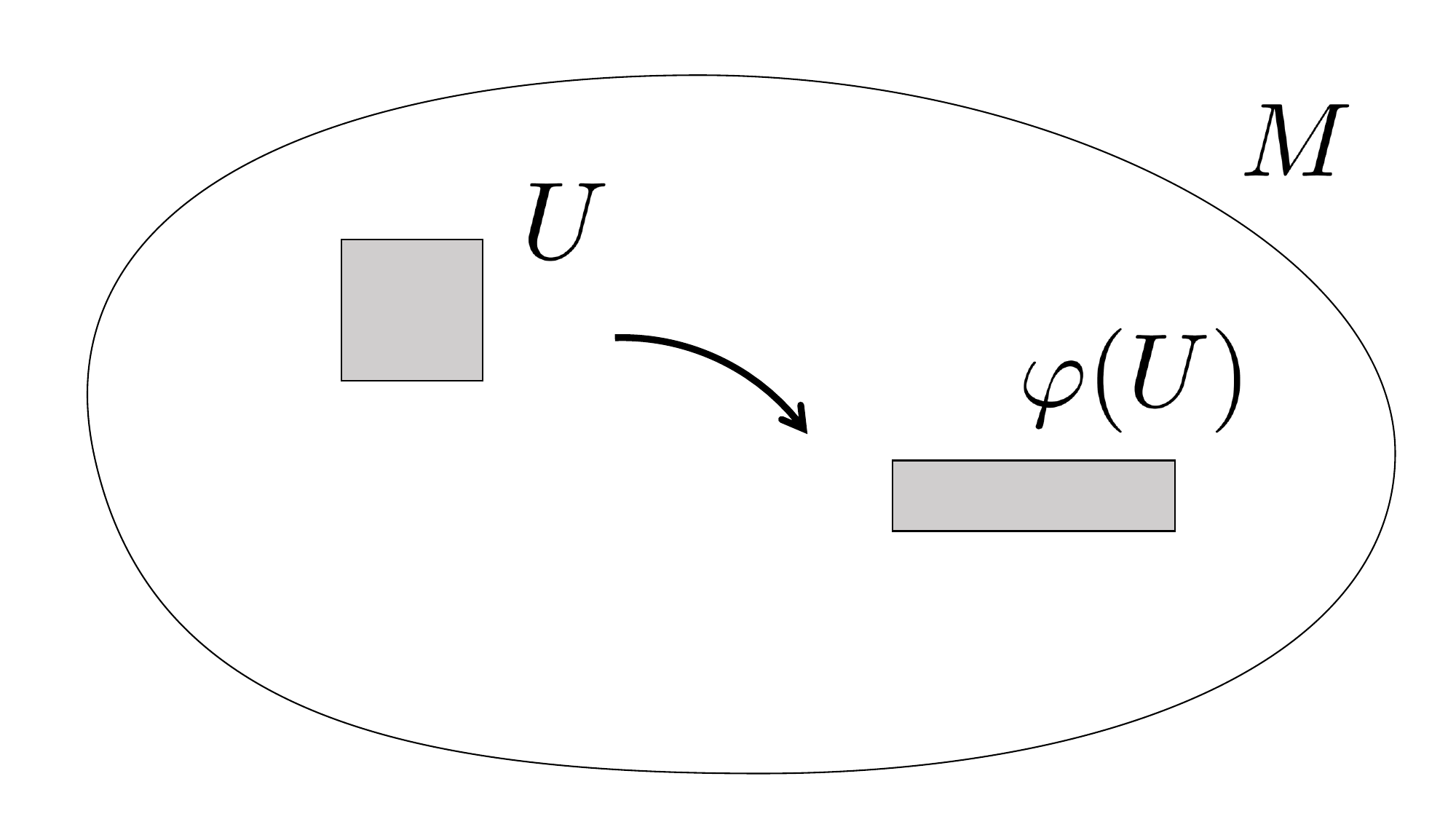}

\caption{Phase space with volume preserving evolution}\label{Fig:phase-space}
\end{figure}


Now, in practice we never precisely know at which point of $M$ our system is, since for all quantities, we have some uncertainty in the measurement. What we know is that our system is somewhere in an open set $U\subset M$. An important observation is that when the system is isolated (i.e. does not exchange energy or information with the environment), then the evolution of $U$ through the time flow is \emph{volume-preserving}. This translates the fact that \emph{we will not gain (or lose) information about the system by simply waiting some time!}

It turns out that it is even better: the physical quantities come in pairs, to each quantity we can associate a so-called \emph{conjugated} quantity. For example, the conjugate of the position in a direction is the momentum in that direction and \textit{vice versa}. The time evolution does not only preserve the total volume of $U$ in the phase space, but also the (induced) volume of all 2-dimensional slices given by a pair of conjugated quantities. This means that we may gain information on the position, but we will lose information on the momentum.

This property of ``preserving volumes of special 2-dimensional slices'' is captured by the fact that \emph{the phase space has a symplectic structure and the time evolution preserves that structure}. 
\begin{definition}
A \textbf{symplectic structure} on a manifold $M$ is a two-form $\omega$ which is closed and non-degenerate. 
\end{definition}

The notion of conjugated variables arise by some linear algebra of two-forms:
\begin{exo}
Let $V$ be a vector space equipped with a antisymmetric bilinear form $\omega$ which is non-degenerate (i.e. $\omega(x,y)=0 \;\forall \, y \Rightarrow x=0$). Show that the dimension of $V$ is even and that $V$ admits a basis $(e_1,...,e_{2n})$ such that $\omega(e_{2i-1},e_{2i}) = 1 = -\omega(e_{2i},e_{2i-1})$ for all $i=1,...,n$ and $\omega(e_i,e_j)=0$ for all other $i,j$.
\end{exo}

The standard examples arise as the simplest phase spaces of physical systems. 
\begin{itemize}
	\item A free particle moving on $\R$ gives $M=T^*\R$. The symplectic form is nothing but the area form, given by $\omega_0=dp\wedge dx$ (where $(p,x)$ are coordinates on $T^*\R\cong \R^2$).
	\item More generally, for a free particle moving in $\R^n$, its phase space is $\R^{2n}=T^*\R^n$ (position and momentum) with $\omega_0 = \sum_i dp_i\wedge dx_i$, called the \emph{standard symplectic structure}.
	\item If the particle is constraint to stay on some manifold $N$, the phase space becomes the cotangent bundle $T^*N$. Again we can write 
	\begin{equation}\label{sympl-cotangent-bundle}
	\omega=\sum_i dp_i\wedge dx_i.
	\end{equation}
	This needs some explanation. We have $$T_{(p,x)}T^*N \cong T^*_xN\oplus T_xN.$$ In a coordinate independent manner, we have $\omega((\varphi,X),(\varphi',X'))=\varphi(X')-\varphi'(X)$ which can be identified with Equation \eqref{sympl-cotangent-bundle}. Note that $\omega=d\lambda$ is exact where $\lambda=\sum_i p_i dx_i$ is called the \emph{Liouville form}.
\end{itemize}

An important result in the local theory of symplectic manifolds is the following:
\begin{thm}[Darboux theorem]
On any symplectic manifold $(M,\omega)$ there exists local coordinates $(p_i,x_i)$ such that $\omega=\sum_i dp_i\wedge dx_i$.
\end{thm}
This means that there is no local invariant in symplectic geometry. Note that this is not the case for Riemannian geometry where the curvature is a local invariant. 
Coordinates with the property of the theorem are called \emph{Darboux coordinates}.

\begin{exo}
Show that a non-degenerate 2-form on a compact manifold without boundary cannot be exact. Deduce that there is no symplectic embedding of a symplectic manifold into $(\R^{2n},\omega_0)$. Which spheres admit symplectic structures?
\end{exo}

\medskip
\paragraph{Symplectic gradient.}
The definition of a symplectic structure is very similar to that of a Riemannian structure, the main difference being that the 2-tensor is symmetric for the Riemannian structure, and anti-symmetric for a symplectic structure. One of the key notions in Riemannian geometry, the gradient of a function, is still available in symplectic geometry: the \textbf{symplectic gradient} of a function $f$ is a vector field $\mathrm{sgrad}(f)$ such that 
$$\omega(\mathrm{sgrad}(f),X) = -df(X)$$ 
for all $X\in \Gamma(TM)$. This is exactly the same definition as for the gradient (we need only non-degeneracy of $\omega$ to get a well-defined notion).

So why is the geometry of a phase space symplectic and not Riemannian? One reason is that a symplectic structure gives naturally the notion of conjugated variables. Another reason is that it allows to treat the important issue of \emph{conserved quantities}. While the usual gradient of a function $f$ points in the direction of biggest change of $f$, the symplectic gradient points in the direction where $f$ stays constant (the level set). This allows to refine our picture from above: the physical evolution of a system is described by a vector field on the phase space which is the symplectic gradient of a special function, called the \textbf{Hamiltonian} of the system, which is conserved in time. For isolated systems, \emph{the Hamiltonian is the total energy}.

\begin{example}
Consider $M=\mathbb{S}^2\subset \R^3$ the sphere with symplectic form being the area form, and the height function $h:\mathbb{S}^2\to\R$ given by the $z$-coordinate. Then the hamiltonian flow is the rotation around the $z$-axis. This of course preserves the level-sets.
\begin{figure}[h!]
\centering
\includegraphics[height=2.5cm, angle=-90,origin=c, trim=100 0 150 0, clip]{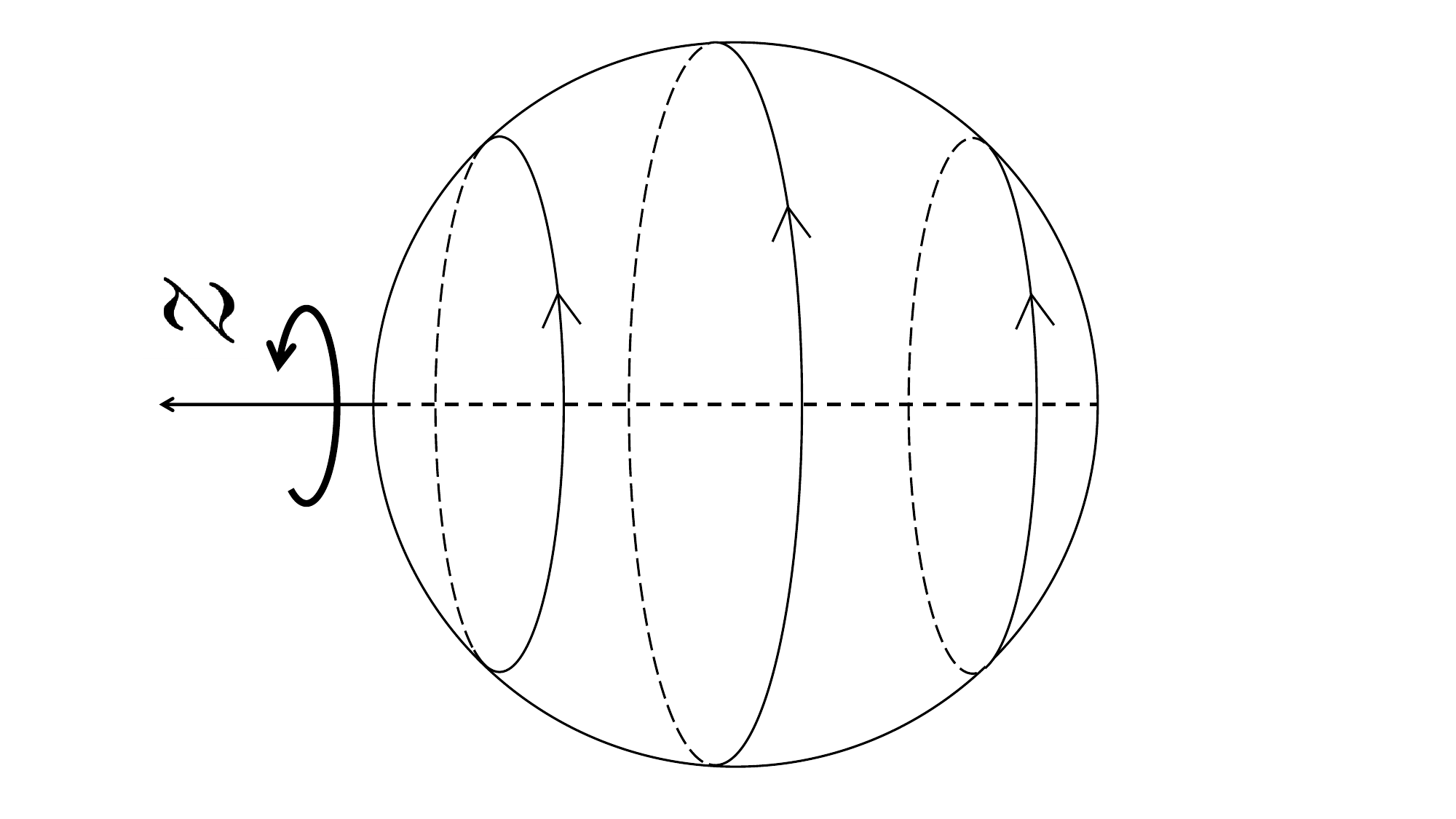}
\caption{Height function generates rotation for a sphere}
\end{figure}
\end{example}

\begin{prop}
The 1-parameter group of diffeomorphisms of $M$, obtained by integrating the symplectic gradient of a function, preserve the symplectic structure of $M$.
\end{prop}
The proof uses Cartan's magical formula and illustrates the importance of $\omega$ being closed.

\begin{proof}
We want to show that $f_t^*\omega=\omega$ for all $t\in \R_+$, where $f_t$ denotes the flow associated to a Hamiltonian $H$. For that, it is sufficient to show that the derivative of $f_t^*\omega$ is zero. We restrict attention to $t=0$, the other values are similar.

By definition of the Lie derivative $\mathcal{L}$, we have $$\frac{d}{dt}\bigg\rvert_{t=0}f_t^*\omega = \mathcal{L}_{X_H}\omega$$ where $X_H$ denotes the symplectic gradient of $H$. Using Cartan's magical formula and writing $\iota_X$ for the inner product, we get
$$\mathcal{L}_{X_H}\omega = (d\circ \iota_{X_H}+\iota_{X_H}\circ d)\omega = d\left(\omega(X_H,.)\right) + \iota_{X_H}d\omega.$$
Since $\omega$ is closed, the second term vanishes. The first term vanishes as well since by definition $\omega(X_H,.)=-dH$ is exact.
\end{proof}

A diffeomorphism $f$ of $M$ preserving its symplectic structure (i.e. $f^*\omega=\omega$) is called a \textbf{symplectomorphism}. The special case coming from the time 1 flow of a symplectic gradient is called a \textbf{Hamiltonian diffeomorphism}. Not all symplectomorphisms are Hamiltonian. This distinction can be seen on the ``infinitesimal'' level: the Lie algebra of the diffeomorphism group is the space of vector fields. A vector field $X$ is called
\begin{itemize}
	\item symplectic if $\omega(X,.)$ is a closed 1-form,
	\item Hamiltonian if $\omega(X,.)$ is exact.
\end{itemize}
A vector field associated to a symplectomorphism (resp. hamiltonian diffeomorphism) is symplectic (resp. hamiltonian).

Note that there is no constraint on the Hamiltonian of a physical system: any function in  $\mathcal{C}^\infty(M)$ can in principle be used to generate time evolution, and it is the physicists job to find out which function it is for a given system. Jürgen Jost \cite{jost} puts it in these words: 
\begin{center}
\textit{``The aim of physics is to write down the Hamiltonian of the universe. \newline The rest is mathematics.''}
\end{center}
Hence, we can give two physical meanings to a function: as \emph{observable quantity} or as \emph{infinitesimal generator of a transformation}. 
This shows the importance of the space of functions.

\medskip
\paragraph{Poisson bracket.}
The space of smooth functions $\mathcal{C}^\infty(M)$ comes with an extra structure: a \textbf{Poisson bracket}. Concretely, the symplectic form $\omega$ allows to associate to two functions $f, g\in \mathcal{C}^\infty(M)$ a function denoted by $\{f,g\}$ defined by
$$\{f,g\} = \omega (\mathrm{sgrad}(f),\mathrm{sgrad}(g)).$$
Abstractly, a Poisson bracket is a bilinear map $\mathcal{C}^\infty(M)\times \mathcal{C}^\infty(M)\to \mathcal{C}^\infty(M)$ which is symmetric, satisfies Leibniz' rule $\{fg,h\} = f\{g,h\}+\{f,h\}g$ and the \emph{Jacobi identity} $$\{f,\{g,h\}\}+\{h,\{f,g\}\}+\{g,\{h,f\}\}=0.$$
Roughly speaking, a Poisson bracket is a Lie bracket and a derivation.

\begin{exo}
Show that the Jacobi identity of the Poisson bracket is equivalent to $\omega$ being closed.
\end{exo}

In Darboux coordinates $(x_i,p_i)$, the Poisson bracket has a simple expression:
\begin{equation}\label{Poisson-bracket-expr}
\{.,.\} = \sum_i \frac{\del}{\del p_i}\wedge \frac{\del}{\del x_i}.
\end{equation}
This means that 
$$\{f,g\}=\sum_i \frac{\del f}{\del p_i}\frac{\del g}{\del x_i}-\frac{\del f}{\del x_i}\frac{\del g}{\del p_i}.$$

The important point is to know the interpretation of the Poisson bracket: \emph{$\{f,g\}(x)$ is the change of $g$ at $x$ along the flow line generated by the symplectic gradient of $f$}. In physical words: if $H$ is the Hamiltonian of a physical system and $g$ some observable, then 
\begin{equation}\label{Poisson-signification}
\boxed{\frac{dg}{dt}(x) = \{H,g\}}
\end{equation}
where $t$ is the parameter of the family of Hamiltonian diffeomorphisms generated by $H$, and $\frac{dg}{dt}=dg(X_H)$. In particular, the Hamiltonian is preserved in time since for $g=H$ we get zero in Equation \eqref{Poisson-signification}. This is the well-known fact that the total energy of an isolated system is conserved.

The strength of Equation \eqref{Poisson-signification} is that we can recover, in a uniform way, all differential equations from classical systems. All exercises from physics class in High School about determination of the equation of motion become easy (at least systematic)!

\begin{example}
Consider a point particle with mass $m$ moving on $\R$ in a potential $V$ with gravity $g$. The phase space is $M=T^*\R$ with standard symplectic structure. The total energy is $H=\frac{p^2}{2m}+mgV(x)$ (kinetic + potential energy). Therefore we get
$$\dot x = \{H,x\} = \frac{\partial H}{\partial p}=p/m,$$
which is nothing new, since $p=mv=m\dot x$. Then
$$\ddot x = \{H,\dot x\}=\{H,p/m\}=-\frac{1}{m}\frac{\partial H}{\partial x} = -gV'(x)$$
which is nothing but Newton's law.
\end{example}

\begin{example}
Consider a pendulum: a mass $m$ turning around a fixed point at distance $\ell$ (see Figure \ref{fig:pendulum}). The phase space is $M=T^*\mathbb{S}^1$ with Darboux coordinates $(p,\ell\theta)$. The total energy is $H=\frac{p^2}{2m}-mg\ell \cos\theta$. Hence
$$\dot \theta = \{H,\theta\} = \frac{1}{\ell}\{H,x\}=\frac{1}{\ell}\frac{\partial H}{\partial p}=\frac{p}{m\ell},$$
which is nothing new, since $p=mv=m\dot x=m\ell \dot\theta$. Then
$$\ddot \theta = \{H,\dot \theta\}=\frac{1}{m\ell}\{H,p\}=-\frac{1}{m\ell^2}\frac{\partial H}{\partial \theta} = -\frac{g}{\ell}\sin \theta$$
which is the usual law for a pendulum.

\begin{figure}[h!]
\centering
\includegraphics[height=3.7cm]{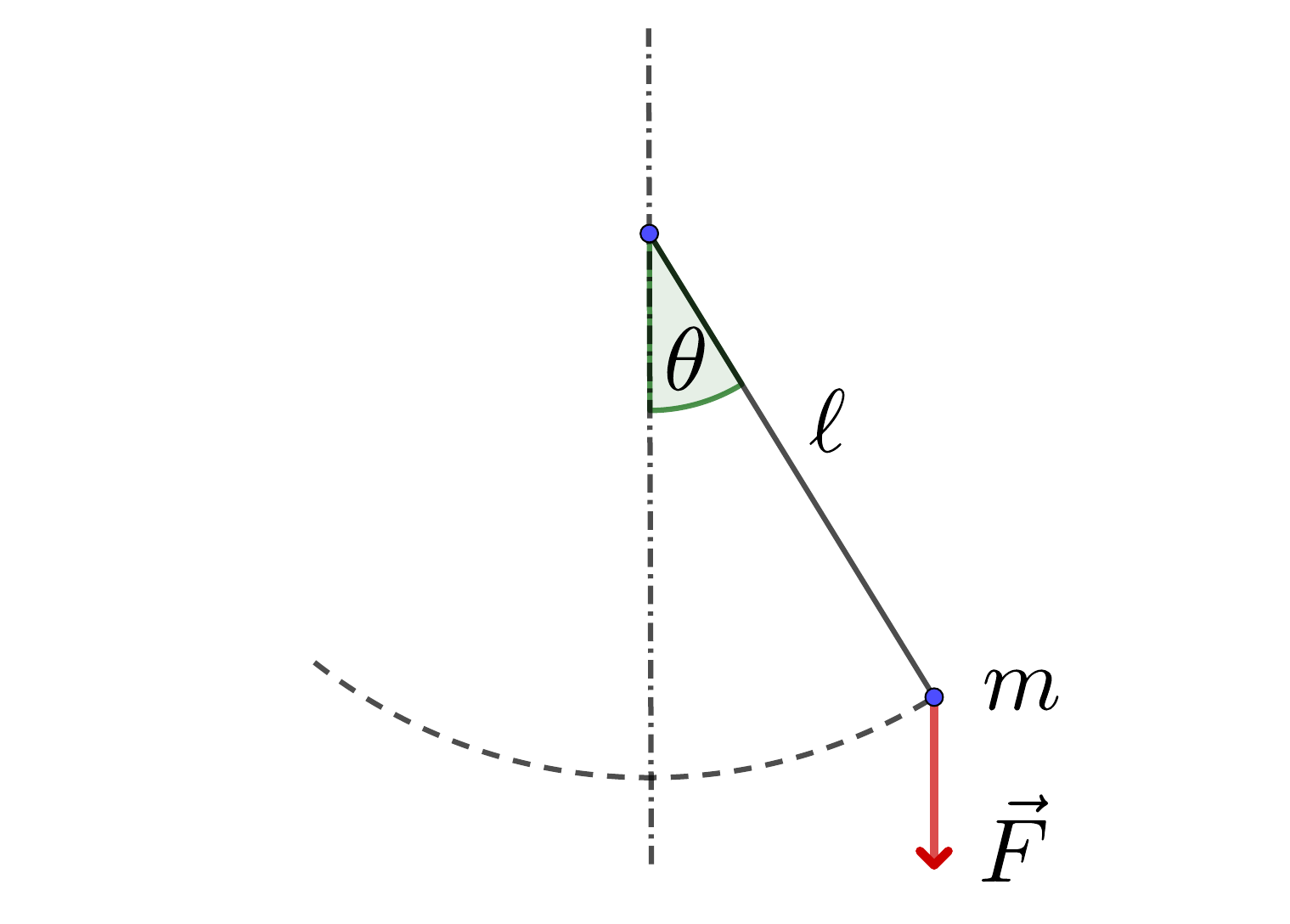}

\caption{Pendulum}\label{fig:pendulum}
\end{figure}
\end{example}

A manifold $M$ with a Poisson bracket on its function space $\mathcal{C}^\infty(M)$ is called a \textbf{Poisson manifold}. They are more general than symplectic manifolds:

\begin{example} \label{cross-product}
Using the cross product, we can construct a Poisson structure on $\R^3$, without being symplectic (since the dimension is odd). The Poisson bracket for linear functions (which can be represented by vectors) is given by the cross product. This means that $\{x,y\}=z, \{y,z\}=x$ and $\{z,x\}=y$.

Then we extend the bracket to all functions using the Leibniz rule. As a bivector, we can write
$$\{.,.\}=z\frac{\del}{\del x}\wedge \frac{\del}{\del y}+x\frac{\del}{\del y}\wedge \frac{\del}{\del z}+y\frac{\del}{\del z}\wedge \frac{\del}{\del x}.$$
\end{example}

The main thing to know about a Poisson manifold is that \emph{it is the disjoint union of symplectic manifolds} (in a unique way), called \textbf{symplectic leaves}. For $(\R^3, \times)$ from above, the symplectic leaves are all spheres with radius $r\geq 0$. Note in particular that the various symplectic leaves might have different dimensions.

The following is a very important example. In some sense, it is the universal $G$-Poisson manifold.

\begin{example}
A dual Lie algebra $\g^*$ is always Poisson\footnote{The dual is taken in the sense of linear algebra: $\g^*$ is the vector space of all linear forms on $\g$.}. For linear functions, which can be identified with $\g^{**}\cong \g$, the Poisson bracket is given by the Lie bracket in $\g$. Then we extend to all functions by the Leibniz rule. Concretely for two functions $f,g \in \mathcal{C}^\infty(\g^*)$ and $\xi \in \g^*$ we have
$$\{f,g\}(\xi) = \langle \xi, [d_\xi f, d_\xi g]\rangle$$
where $\langle .,. \rangle$ denotes the canonical pairing between $\g^*$ and $\g$. 

The symplectic leaves of $\g^*$ are the \emph{coadjoint orbits}, i.e. the orbits under the action of the Lie group $G$ acting on $\xi\in \g^*$ by 
$$\langle g.\xi, x\rangle = \langle \xi, \mathrm{Ad}_{g^{-1}}(x)\rangle \;\forall\, x\in \g$$
where $\mathrm{Ad}$ denotes the adjoint action of $G$ on $\g$.
\end{example}

\begin{exo}
For $\g=\mathfrak{so}(3)$, show that $\g^*$ is $(\R^3, \times)$ from Example \ref{cross-product} above, and deduce the symplectic leaves.
\end{exo}

\subsection{Hamiltonian reduction}\label{hamiltonian-reduction}

One of the most important operations on spaces is the quotient under a group action. For symplectic manifolds, this quotient is known as the Hamiltonian reduction (sometimes called Marsden--Weinstein quotient). From the physics perspective, given a phase space with some symmetry, this allows to define a reduced phase space.

\bigskip
Let $(M,\omega)$ be a symplectic manifold and $G$ be a Lie group acting on $M$. The action is called \textbf{symplectic} if all $g\in G$ act by symplectomorphisms, i.e. $g^*\omega=\omega$ for all $g$.

The infinitesimal action of an element $\xi\in \g$ gives a vector field $X_\xi$. Explicitly it is given by $$X_\xi(x)=\frac{d}{dt}\bigg\rvert_{t=0}\exp(t\xi).x$$ where we write $g.x$ for the action of $g\in G$ on $x\in M$. Note that for any group action, we have $[X_\xi,X_\nu]=X_{[\xi,\nu]}$ for all $\xi, \nu \in \g$. In the case of a symplectic action, the vector fields $X_\xi$ are symplectic.

In order to define the quotient of $M$ by $G$ and to ensure to obtain a symplectic manifold, the action has to satisfy some extra conditions.
\begin{definition}
The action of $G$ on $M$ is called \textbf{weakly Hamiltonian} if each $g\in G$ acts by a Hamiltonian diffeomorphism.
\end{definition}

For a weakly Hamiltonian action, $X_\xi$ is a hamiltonian vector field, i.e. the symplectic gradient of some function $H_\xi$. The function $H_\xi$ is not uniquely defined from the vector field, but only up to addition of an overall constant. 


\begin{definition}
A weakly Hamiltonian action of $G$ on $M$ is called \textbf{Hamiltonian} if there is a Lie algebra homomorphism $H: \g\to \mathcal{C}^\infty(M)$ such that $H_\xi$ generates the Hamiltonian vector field associated to $\xi\in \g$. In particular, we have
	\begin{equation}\label{Hamiltonian-map}
	H_{[\xi,\nu]} = \{H_\xi,H_\nu\}.
	\end{equation}
\end{definition}

In more abstract terms, we can say that the action is
\begin{itemize}
	\item \emph{weakly Hamiltonian} if there is a Lie algebra homomorphism $\g\to \Gamma_H(TM)$, where $\Gamma_H(TM)$ denotes the space of Hamiltonian vector fields,
	\item \emph{Hamiltonian} if there is a Lie algebra homomorphism $\g\to \mathcal{C}^\infty(M)$.
\end{itemize}

As we will see below, the difference between a weakly Hamiltonian action and a Hamiltonian action is some 2-cocycle in $\mathrm{H}^2(\g)$. 
To see a bit clearer in this zoo of symplectic, weakly hamiltonian and hamiltonian actions, there are some useful facts from topology:
\begin{itemize}
	\item If $H^1(M)=0$ or $H^1(\g)=0$, then symplectic implies weakly hamiltonian.
	\item If $H^2(\g)=0$, then weakly hamiltonian implies hamiltonian.
\end{itemize}

From a general fact in Lie algebra cohomology, the \emph{Whitehead lemma}, we know that $H^1(\g)=0=H^2(\g)$ for finite-dimensional semisimple $\g$. Therefore:
\begin{prop}
If $G$ is semisimple and finite-dimensional, then a symplectic action is automatically Hamiltonian.
\end{prop}

In the presence of a Hamiltonian action, we can define the \textbf{moment map} $\mu:M\to \g^*$ defined by
$$\mu(m).\xi = H_\xi(m).$$

The name ``moment map'' comes from some basic examples where we recover momenta:
\begin{example}
Consider $M=T^*\R^2$ with translation symmetry by $G=\R$ acting by $r.(p,x) = (p,x+r)$. One checks that the action is Hamiltonian with moment map $\mu: T^*\R^2\to \R^*\cong \R$ given by the momentum $p$, i.e. $\mu(p,x).t = pt$.
\end{example}

We recommend to be careful when dealing with moment maps since they are not very intuitive at the beginning (in particular since the target is the dual Lie algebra). Here is a general procedure to compute the moment map:

\begin{itemize}
	\item \textbf{\textit{Step 1:}} Determine the vector field $X_\xi$ for $\xi\in \g$ by computing $g.x$ for $g=1+\varepsilon \xi$ to first order in $\varepsilon$. The result is of the form $(1-\varepsilon \xi).x=x+\varepsilon X_\xi(x)$.
	\item \textbf{\textit{Step 2:}} Compute $\omega(X_\xi,\delta x)$ and put it into the form $\delta (\text{something})$, where $\delta$ denotes the variation. The expression for ``something'' is the moment map.
\end{itemize}



Some explanations for the procedure: First, the mysterious $\delta x$ is a ``variation around $x\in M$''. To be precise, given a path $x:[0,1]\to M$ with $x(0)=x$, the variation is defined to be 
$$\delta x = \frac{d}{dt}\bigg\rvert_{t=0}x(t).$$
It is a tangent vector in $T_xM$. We can generalize and define the variation of a function $f$ around $x$: $$\delta f (x) = \frac{d}{dt}\bigg\rvert_{t=0} f(x(t)).$$
The variation $\delta$ acts like a derivation, i.e. $\delta (fg) = \delta(f) g+ f\delta(g)$.

Second, let us see why the procedure works. Since $\mu: M\to \g^*$, we have $d_x\mu : T_xM\to T_{\mu(x)}\g^* \cong \g^*$. Hence $d_x\mu(\xi,X)=d_xH_\xi(X) = \omega_x(X_\xi,X)$. Let $x:[0,1]\to M$ be a path with $x(0)=x$. Then 
$$\delta \mu(x,\xi) = \frac{d}{dt}\bigg\rvert_{t=0} \mu(x(t),\xi) = d_x \mu(\xi,x'(0)) = \omega(X_\xi, \delta x).$$

This justification seems complicated, but in practice, the procedure is extremely efficient. Let us see an example:
\begin{example}
Take the phase space of a particle moving in three-space, $M=T^*\R^3\cong \R^3\times \R^3$. Consider the diagonal action of $G=\mathrm{SO}(3)$. Since the group is simple, we only have to show that the action is weakly Hamiltonian. For that, we directly compute the moment map.

Step 1: For $\xi\in \mathfrak{so}(3)$, the vector field is simply given by $X_\xi(x,p)=(\xi(x),\xi(p))$.

Step 2: We compute for $X=(x,p)$:
\begin{align*}
\omega((\xi(x),\xi(p)),(\delta x, \delta p)) &= \sum_i (\xi(x)_i \delta p_i - \xi(p)_i\delta x_i) \\
&= \sum_{i,j} (\xi_{ij} x_j\delta p_i - \xi_{ij} p_j \delta x_i) \\
&= \sum_{i,j} (\xi_{ij} x_j\delta p_i + \xi_{ji} p_j \delta x_i) \;\; \text{ using antisymmetry of } \xi \\
&= \delta \left(\textstyle\sum_{i,j}\xi_{ij}x_ip_j\right) \\
&= \delta \langle p,\xi(x)\rangle.
\end{align*}

Hence the moment map is given by $\mu(x,p).\xi=\langle p,\xi(x)\rangle$. Under some identification of $\R^3$ with $\mathfrak{so}(3)$ (see \cite[Example 5.3.1]{mcduff}), we can write $\mu(x,p)=x\times p$. Hence the moment map is nothing but the \emph{angular momentum}.
\end{example}

\begin{exo}
Show that the natural action of the linear symplectic group $\mathrm{Sp}_{2n}(\R)$ on $(\R^{2n},\omega_0)$ is hamiltonian with moment map $\mu(x).A=\frac{1}{2}\omega(x,Ax)$.
\end{exo}

In the two examples, we have seen that to a continuous symmetry, we can associate a conserved quantity, which is described by the moment map. For translations, we get the momenta, and for rotations, we get the angular momenta. 

In general, the moment map behaves nicely under the action of $G$:
\begin{equation}\label{Noether-thm}
\mu(g.m) = \mathrm{Ad}^*_g.\mu(m)
\end{equation}
It shows that for an abelian group $G$ (a torus action for example), the moment map is constant along the $G$-orbits. Thus, we can think of it as the \emph{collection of all preserved quantities associated to the symmetry}. For $G$ not abelian, this is not exactly true, but we precisely know how the moment map changes on the $G$-orbit.

Equation \eqref{Noether-thm} can be seen as one way to formulate the famous \textbf{Noether theorem}, stating that \emph{to each continuous symmetry in a physical system, there is an associated conserved quantity}.

\begin{exo}
Show Equation \eqref{Noether-thm}. Notice the importance of the requirement $H_{[\xi,\nu]} = \{H_\xi,H_\nu\}$.\footnote{Hint: For the infinitesimal action, check that $\mu((1+\xi).m).\nu-\mu(m).\nu=\frac{dH_\nu}{dt}(m)=\{H_\xi,H_\nu\}(m) = H_{[\xi,\nu]}(m) = \mathrm{ad}^*_\xi \mu(m).\nu$.}
\end{exo}


Now, we can define the Hamiltonian reduction:
\begin{definition}
For a Hamiltonian action of $G$ on $(M,\omega)$, suppose that $0\in \g^*$ is a regular value of the moment map $\mu$. Then the \textbf{Hamiltonian reduction} of $M$ by $G$ (over 0) is defined to be
$$M\sslash_0 G := \mu^{-1}(\{0\})/G.$$
More generally, for a coadjoint orbit $\mathcal{O}\subset \g^*$, we define 
$$M\sslash_\mc{O} G := \mu^{-1}(\mc{O})/G.$$
\end{definition}
Note that by Equation \eqref{Noether-thm}, the level surface $\mu^{-1}(\mc{O})$ is invariant under $G$. 

\begin{thm}
The hamiltonian reduction $M\sslash_\mc{O} G$ inherits a symplectic structure from $M$. To be precise: if $M\sslash_\mc{O} G$ is a manifold, i.e. if the action of $G$ on the level set $\mu^{-1}(\{\mc{O}\})$ is free and proper, then it symplectic.
\end{thm}

The situation is even more exciting: if we take the simple quotient $M/G$, we get a Poisson manifold (which is not symplectic). Indeed, a function on $M/G$ is nothing but a $G$-invariant function on $M$ and the Poisson bracket of $G$-invariant functions stays $G$-invariant. The natural question you should always ask when you have a Poisson manifold is that of its symplectic leaves. Well:

\begin{thm}\label{thm:Poisson-manifold}
For a Hamiltonian action of $G$ on $M$, the quotient $M/G$ is a Poisson manifold (potentially singular) whose symplectic leaves are given by $M\sslash_\mc{O} G $ where $\mc{O}$ describes all coadjoint orbits of $\g^*$.
\end{thm}

\begin{example}
Consider $M=T^*G$ with the natural $G$-action. Since $T^*G \cong \g^*\times G$, we have $M/G = \g^*$ which is Poisson. Further, $\mu$ is simply the projection map. Hence the reduced space $M\sslash_\mc{O} G$ is nothing but $\mc{O}$ itself. We recover that the symplectic leaves of $\g^*$ are its coadjoint orbits. 
\end{example}

\begin{example}\label{cp-n-quotient}
Consider the action of $\mathbb{S}^1=U(1)$ acting on $M=\C^{n}$ via $\lambda.(z_j) = (\lambda z_j)$. Since $\C^n\cong T^*\R^n$, we have the standard symplectic structure which in complex coordinates is given by $\omega_0=\frac{i}{2}\sum_j dz_j\wedge d\bar{z}_j$. Note that the action is symplectic since $\lambda \bar{\lambda}=1$.

Identifying $\mathfrak{u}(1)\cong i\R$, step 1 gives the vector field $Z_{ir}(z_j)=(irz_j)$ for $r\in \R$. Then step 2 gives
$$\omega(Z_{ir},\delta z)=\frac{i}{2}\sum_j r(iz_j\delta \bar{z}_j+i\bar{z}_j\delta z_j) = \delta\left(-\frac{1}{2}r\left\|z\right\|^2\right).$$
Since $\mathbb{S}^1$ is abelian, every point of its dual Lie algebra is a coadjoint orbit. The hamiltonian reduction over -2 gives
$$\C^n\sslash_{\{-2\}} \mathbb{S}^1=\mu^{-1}(\{-2\})/\mathbb{S}^1 \cong \C P^{n-1}.$$
Hence, we have shown that \emph{complex projective spaces carry a natural symplectic structure} (called the Fubini--Study structure). Note that the coadjoint orbit of 0 is not a regular value for $\mu$ in this example.
\end{example}

An important fact which we will use quite often when computing the cotangent bundle of moduli spaces is left as an exercise:
\begin{exo}\label{cotang-group-quotient}
For $G$ acting on a manifold $X$, show that (under mild conditions) 
$$\boxed{T^*(X/G) \cong T^*X\sslash G.}$$
\end{exo}

\begin{exo}
Consider the diagonal action of $G$ on $T^*\g$. Show that the action is hamiltonian and compute the moment map\footnote{Using the Killing form to identify $\g^*$ with $\g$, you should find $\mu(X,Y)=[X,Y]$.}.
\end{exo}

\medskip
\paragraph{Weakly hamiltonian reduction*.}
Finally, without entering into too much detail, let us analyze the case of a weakly Hamiltonian action. For more details, we refer to \cite[Chapter 1, Section 4]{kirillov}. The slogan is: 
\begin{center}
\emph{``You can always transform a weakly Hamiltonian action into a Hamiltonian one by replacing the group $G$ by a central extension.''}
\end{center}

Consider a weakly Hamiltonian action of $G$ on $(M,\omega)$. Since every $g\in G$ acts by Hamiltonian diffeomorphisms, each vector field associated to $\xi\in \g$ is the symplectic gradient of some function $H_\xi$. Define constants $c(\xi,\nu)$ by
$$\{H_\xi,H_\nu\}=H_{[\xi,\nu]}+c(\xi,\nu).$$
\begin{exo}
Prove that $c(\xi,\nu)$ is a constant function\footnote{Hint: Show that the symplectic gradient of $\{H_\xi,H_\nu\}$ and $H_{[\xi,\nu]}$ coincide.}.
\end{exo}

One checks that the Jacobi identity implies that $c$ form a 2-cocycle in the cohomology of $\g$. This means that $$c([X,Y],Z)+c([Y,Z],X)+c([Z,X],Y)=0 \;\forall\, X,Y,Z\in \g.$$ Changing the functions $H_\xi$ by constants changes $c$ by a coboundary term. Hence, the cohomology class $[c]\in H^2(\g)$ is well-defined.

Since elements in $H^2(\g)$ describe central extensions, there is an associated central extension $\widehat{\g}=\g\oplus \mathbb{C}c$. The bracket is defined by 
$$[(\xi_1,a_1),(\xi_2,a_2)] = ([\xi_1,\xi_2],c(\xi_1,\xi_2)).$$

Consider the group $\widehat{G}$ associated to $\widehat{\g}$. It is an extension of $G$ by a one-dimensional subgroup $Z$ consisting in central elements:
$$1\to Z \to \widehat{G}\to G\to 1.$$
Define the action of $\widehat{G}$ on $M$ to be the one of $G$, with the subgroup $Z$ acting trivially. By construction, this action is Hamiltonian. Indeed we can put $H_{(\xi,a)}=H_\xi+a$ and then we get
$$\{H_{(\xi_1,a_1)},H_{(\xi_2,a_2)}\} = \{H_{\xi_1},H_{\xi_2}\} = H_{[\xi_1,\xi_2]}+c(\xi_1,\xi_2)=H_{([\xi_1,\xi_2],c(\xi_1,\xi_2))}=H_{[(\xi_1,a_1),(\xi_2,a_2)]}.$$
The only thing which changes is the structure of the dual Lie algebra, which is now $\widehat{\g}^*$. So the moment map remembers one more extra information.

We will see a nice application of this construction in the Atiyah--Bott reduction for surfaces with boundary in the next section.

\section{Atiyah--Bott reduction}\label{atiyah--bott}

We have seen in the Riemann--Hilbert correspondence that the character variety is the moduli space of flat connections. In the special case where the manifold $M=\Sigma$ is a closed surface, the character variety gets a symplectic structure, called the \textbf{Goldman symplectic structure}. More precisely, the character variety is the Hamiltonian reduction of a very simple, but infinite-dimensional space, the space of all connections. For a surface with boundary, the character variety has a Poisson structure.

\medskip
Let $\Sigma$ be a surface and $G$ be a subgroup of $\GL_n(\C)$ with Lie algebra $\g$ (it can easily be adapted to general semisimple Lie groups). Let $E$ be a trivial $G$-bundle over $\Sigma$. Denote by $\mathcal{A}$ the space of all $\g$-connections on $E$. We have seen in Section \ref{bundle-formalism} that $\mathcal{A}$ is an affine space modeled over the vector space of $\g$-valued 1-forms $\Omega^1(\Sigma, \g)$. Further, denote by $\mathcal{G}$ the space of all gauge transforms, i.e. bundle automorphisms. We can identify the gauge group with $G$-valued functions: $\mathcal{G}=\Omega^0(\Sigma,G)$.

On the space of all connections $\mathcal{A}$, there is a natural symplectic structure given by 
\begin{equation}\label{sympl-str-on-conn}
\hat{\omega} = \int_{\Sigma} \tr \;\delta A \wedge \delta A
\end{equation}
where $\tr$ denotes an $\mathrm{Ad}$-invariant non-degenerate form on the Lie algebra, which we call the \emph{trace form}. For matrix Lie algebras $\g\subset \mathfrak{gl}_n$, we can simply choose the trace (for semisimple Lie algebras, we can take the Killing form).

Let us explain this simple looking expression for $\hat{\omega}$. 
Since $\mathcal{A}$ is an affine space, its tangent space at every point is canonically isomorphic to $\Omega^1(\Sigma, \g)$. So given $A \in \mathcal{A}$ and $B, C \in T_A\mathcal{A} \cong \Omega^1(\Sigma, \g)$, we can write $B=B_xdx+B_ydy$ with $B_x,B_y\in \g$ and similar for $C$. Then we have $$\hat{\omega}_A(B, C) = \int_{\Sigma} \tr \; B\wedge C = \int_\S \tr(B_xC_y-B_yC_x)dx\wedge dy.$$ Note that $\hat{\omega}$ is independent of $A$ so $d\hat{\omega} = 0$. Further, the 2-form $\hat{\omega}$ is clearly antisymmetric and non-degenerate (since the trace form is). Note that this construction only works on a surface. Finally, if $G$ is a complex Lie group, $\hat\omega$ is a complex symplectic structure. 

\begin{Remark}
When using the wedge product and commutator of $\g$-valued 1-forms, some intuitions from usual exterior calculus are not valid anymore. In particular we have 
\begin{equation}\label{a-wedge-a}
[A,B] = [B,A] \;\;\; \text{ and } \;\;\; A\wedge A =\frac{1}{2}[A,A]
\end{equation} 
which is non-zero in general. Check these equalities by plugging in $A=A_x\, dx+A_y \, dy$ and similar for $B$.
\end{Remark}

On the space of connections, we have the natural action by the gauge transforms.
The surprising observation of Atiyah and Bott (see end of chapter 9 in \cite{atiyah-bott} for unitary case, see section 1.8 in Goldman's paper \cite{goldman1984symplectic} for the general case) is the following:
\begin{thm}[Atiyah, Bott 1983]\label{atiyah-bott-thm}
The action of gauge transforms on the space of connections is Hamiltonian with moment map the curvature.
\end{thm}

Let us explain the moment map with more detail: it is a map $m:\mathcal{A} \to \mathrm{Lie}(\mathcal{G})^*$. The Lie algebra $\mathrm{Lie}(\mathcal{G})$ is equal to $\Omega^0(\Sigma, \g)$, so its dual is isomorphic to $\Omega^2(\Sigma, \g)$ via the pairing $\langle X,f \rangle = \int_{\Sigma} \tr fX$ for $f\in\Omega^0(\Sigma, \g)$ and $X \in \Omega^2(\Sigma, \g)$.
On the other hand, given a connection $A$, its curvature $F(A)$ is a $\g$-valued 2-form, i.e. an element of $\Omega^2(\Sigma,\g)$. Hence, the theorem asserts that the moment map $\mu$ evaluated at $A\in \mathcal{A}$ and $f\in \Omega^0(\S,\g)$ is given by
\begin{equation}\label{id-AB-reduction}
\mu(A).f = \int_\S \tr f F(A).
\end{equation}

We give a sketch of the proof, a computation of the moment map which is the curvature, neglecting all issues about infinite-dimensional spaces.

\begin{proof}[Sketch of proof]
The action of a gauge transform $g$ on a connection $A$ is given by $g.A = gAg^{-1}+gdg^{-1}$ coming from expanding $g(d+A)g^{-1}$.
So the action on a tangent vector $\delta A$ is given by $g.\delta A = g\delta A g^{-1}$. 

\underline{Step 1:} Let us compute the infinitesimal action by an element $g=1+\varepsilon \xi$ to show that the action is weakly Hamiltonian. We get $(1+\varepsilon\xi).A = A +\varepsilon([\xi,A]-d\xi)$. So we have a vector field $A_\xi=[\xi,A]-d\xi$ on $\mathcal{A}$.

\underline{Step 2:} Now we compute
\begin{align*}
\hat{\omega}(A_\xi,\delta A) &= \int_{\Sigma} \tr \; A_{\xi} \wedge \delta A = \int_{\Sigma} \tr \;([\xi,A]-d\xi) \wedge \delta A \\
&= \int_{\Sigma}\tr [\delta A,A]\xi +\int_{\Sigma} \tr \; \xi\; d\delta A \\
&= \delta \left( \int_{\Sigma} \tr (\xi \; (dA+A\wedge A)) \right)
\end{align*}
where we used integration by parts and the following facts:
\begin{itemize}
	\item the so-called cyclicity property of the trace: $\tr [A,B]C = \tr [B,C]A$,
	\item by Equation \eqref{a-wedge-a}, we have $\delta(A\wedge A) = \frac{1}{2}\delta([A,A]) = \frac{1}{2}([\delta A,A]+[A,\delta A])= [\delta A,A].$
\end{itemize}
Therefore using the identification by Equation \eqref{id-AB-reduction}, we get $\mu(A)=dA+A\wedge A=F(A)$ which is the curvature.

\underline{Step 3:} to show that the action is Hamiltonian, we compute 
\begin{align*}
\{H_{\xi_1},H_{\xi_2}\} &= \hat\omega(d\xi_1+[A,\xi_1],d\xi_2+[A,\xi_2]) \\
&= \int_\S \tr \left(d\xi_1\wedge d\xi_2+d\xi_1\wedge [A,\xi_2]+[A,\xi_1]\wedge d\xi_2+[A,\xi_1]\wedge [A,\xi_2] \right) \\
&= \int_\S \tr \left(d(\xi_1d\xi_2)+[\xi_1,\xi_2](dA+A\wedge A) \right) \\
&= H_{[\xi_1,\xi_2]}
\end{align*}
where we used integration by parts and cyclic properties of the trace.
\end{proof}

The important consequence of the Atiyah--Bott theorem is:
\begin{coro}
We have $$\boxed{\Rep(\pi_1\S,G)\cong \{\text{flat connections}\}/\mathcal{G}= \mathcal{A}\sslash \mathcal{G}.}$$
In particular for a closed surface $\S$, the character variety $\Rep(\pi_1\S,G)$ is a symplectic manifold.
\end{coro}
\begin{proof}
Carrying out the hamiltonian reduction of the gauge action on the space of connections $\mathcal{A}$, we get $$\mathcal{A}\sslash \mathcal{G}=\{\text{flat connections}\}/\mathcal{G}.$$
By the Riemann--Hilbert correspondence, the moduli space of flat connections is diffeomorphic to the character variety.
\end{proof}

\begin{Remark}
Goldman \cite{goldman1984symplectic} gives an explicit formula for this symplectic structure. This is why it is called the \emph{Goldman symplectic structure}. He computes the tangent space to $\Rep(\pi_1 \S,G)$ at a point $\varphi\in \Rep(\pi_1 \S,G)$ in terms of cohomology (group cohomology with coefficients in a twisted module): $T_\varphi \Rep(\pi_1 \S,G)=\mathrm{H}^1(\pi_1 \S, \mathfrak{g}_{\mathrm{Ad}(\varphi)}).$
Combining the cup-product with the trace form, he gets a map
$$\mathrm{H}^1(\pi_1 \S, \mathfrak{g}_{\mathrm{Ad}(\varphi)})\otimes \mathrm{H}^1(\pi_1 \S, \mathfrak{g}_{\mathrm{Ad}(\varphi)}) \to H^2(\pi_1 S,\R) \cong \R$$
which is nothing but the symplectic form.
\end{Remark}

\medskip
\paragraph{Surface with boundary*.}
Consider now the case of a surface with boundary. In step 2 above, we get an extra boundary term from the integration by parts: 
$$H_\xi(A) = \int_\S \tr \xi(dA+A\wedge A) + \int_{\partial \S}\tr \xi A.$$

Since we can find these functions, the action is weakly Hamiltonian. In step 3, we also get an extra term due to integration by parts:
$$\{H_{\xi_1},H_{\xi_2}\} = H_{[\xi_1,\xi_2]}+\int_{\del\S} \xi_1d\xi_2.$$
Thus we have a non-trivial cocycle $c(\xi_1,\xi_2)=\int_{\del\S} \xi_1d\xi_2$.

\begin{thm}
For a surface with boundary, the character variety $\Rep(\pi_1\S,G)$ is a Poisson manifold whose symplectic leaves are flat connections with prescribed conjugacy class for the monodromy around each boundary component.
\end{thm}

The idea of the proof is to reduce the problem to the boundary $\partial \S$ which is a disjoint union of circles. The remaining gauge group on the circle is the \emph{loop group}. Since the action is only weakly Hamiltonian, we have to consider a central extension of the loop group, the famous \emph{affine Lie group}. A good reference for loop groups is the Pressley--Segal book \cite{segal}, especially Chapter 4.

\begin{proof}
Consider $\mathcal{G}_0\subset \mathcal{G}$ the subgroup of those gauge transformations which are the identity on $\del\S$. For $\mathcal{G}_0$, the cocycle vanishes, so we can define the symplectic quotient $\mathcal{A}\sslash\mathcal{G}_0 = \{\text{flat connections}\}/\mathcal{G}_0$.
By some gymnastics, we get 
$$\Rep(\pi_1\S,G) = \{\text{flat connections}\}/\mathcal{G} = (\{\text{flat connections}\}/\mathcal{G}_0) \big/ (\mathcal{G}/\mathcal{G}_0).$$
Hence we obtain the character variety as the quotient of a symplectic manifold. By theorem \ref{thm:Poisson-manifold}, it is a Poisson manifold whose symplectic leaves are given by all possible Hamiltonian reductions.

To determine the symplectic leaves, note that the remaining gauge group is $\mathcal{G}/\mathcal{G}_0 \cong \mathcal{L}G^k$ where $\mc{L}G=\mathcal{C}^\infty(\mathbb{S}^1,G)$ denotes the \emph{loop group} and $k$ the number of boundary components of $\S$. Without loss of generality we consider $k=1$ in the sequel.

The action of $\mc{L}G$ on $\mathcal{A}\sslash\mathcal{G}_0$ is only weakly Hamiltonian, since we have the cocycle $c$. Hence we have to consider the central extension $\widehat{\g}=\mc{L}\g\oplus \C$ with Lie bracket given by $$[(A(z),a),(B(z),b)] = ([A(z),B(z)],\textstyle\oint \tr AdB)$$
where $z\in \bb{S}^1\subset \C$ and $\oint \tr AdB = \int_{\bb{S}^1} \tr AdB$ represents the cocycle $c$. This is the famous \emph{affine Kac--Moody algebra} of type $\g$.

Since the action of $\widehat{G}$ is only on the boundary circle, we can restrict our connection to $\del \S$. Surprisingly, the moment map for the action of $\widehat{G}$ is nearly the identity (it is an inclusion):
\begin{lemma}
The dual affine Lie algebra $\widehat{g}^*$ can be identified with the space of all $k$-connections of type $G$ on the circle. The coadjoint action of $\widehat{G}$ is the gauge action.
\end{lemma}
A \emph{$k$-connection} is a generalization of a connection where the Leibniz rule is replaces by $D(fs)=k \, df s+fD(s)$. Locally, a $k$-connection is of the form $kd+A$. For $k=1$ we get usual connections and for $k=0$ we get $\g$-valued 1-forms. In particular, the moment map of the action of $\widehat{G}$ on the space of connections on $\mathbb{S}^1$ is the inclusion (with $k=1$).

\begin{proof}[Proof of lemma]
Elements of $\widehat{\g}^*$ are pairs $(X(z),k)$ where $X$ is a $\g$-valued 1-form on $\bb{S}^1$ and $k\in \C$. The pairing with $\widehat{\g}$ is given by $$\langle (X(z),k),(A(z),a)\rangle = \textstyle\oint \tr A(z)X(z) + ka.$$
We first show the infinitesimal version of the lemma by computing the $\widehat{\g}$-coadjoint action:
\begin{align*}
\langle \mathrm{ad}^*_{(A,a)}(X,k),(B,b)\rangle &= \langle (X,k), [(-A,-a),(B,b)]\rangle \\
&= -\textstyle\oint \tr X[A,B] - k\textstyle\oint \tr A dB \\
&= \textstyle\oint \tr B ([A,X]+k\, dA)\\
&= \langle ([A,X]+k \, dA,0), (B,b)\rangle.
\end{align*}
Hence $\mathrm{ad}^*_{(A,a)}(X,k)=[A,X]+k\, dA$ which can be identified with the action of an infinitesimal gauge $A$ on a $k$-connection $-kd+X$. After integration, the $\widehat{G}$-coadjoint action is given by the gauge action.
\end{proof}

Finally, we can conclude on the symplectic leaves. The coadjoint orbits are the gauge-equivalence classes of connections on the circle. Now cut of the circle at a point. Using the gauge, we can trivialize the connection. The only information we can not change is how to glue the two ends together. This is precisely the conjugacy class of the monodromy. Therefore, the condition from the second moment map prescribes the conjugacy class of the monodromy around $\del\S$.
\end{proof}

\section{GIT quotients and stability conditions}\label{Sec-GIT}

The geometric invariant theory (GIT) allows to define in a quite general setting a quotient of a manifold $M$ by some group action $G$. It introduces the notion of stable and unstable points which are treated differently to define a well-behaved quotient $M/G$. These stability conditions play a crucial role for defining moduli spaces, in particular for flat connections. We recommend \cite{thomas2005notes} for more details.

\medskip
\paragraph{Introduction.}
Let $G$ be a group acting on some manifold $M$. We wish to define a nice space ``$M/G$''. It should satisfy a universal property: whenever there is a $G$-invariant map $M\to X$, it should factor through $M/G$.

If we take the set-theoretic quotient, i.e. the space of orbits, we often get a space which is not Hausdorff. This happens for example whenever two orbits $\mathcal{O}_1$ and $\mathcal{O}_2$ get arbitrarily close, i.e. if $\bar{\mc{O}}_1\cap\bar{\mc{O}}_2 \neq \emptyset$. Indeed any open sets around the points representing $\mc{O}_1$ and $\mc{O}_2$ in $M/G$ intersect.

\begin{example}\label{ex-aff-git}
Consider the action of $G=\C^*$ on $M=\C^2$ given by $\lambda.(x,y) = (\lambda x, \lambda^{-1}y)$. The orbits are shown in Figure \ref{non-hausdorff}. Note that the only non-closed orbits are $(x=0,y\neq 0)$ and $(x\neq 0,y=0)$ whose closures intersect. The set-theoretic quotient gives $\C$ with three points at the origin.
\begin{figure}[h!]
\centering
\includegraphics[height=3.5cm]{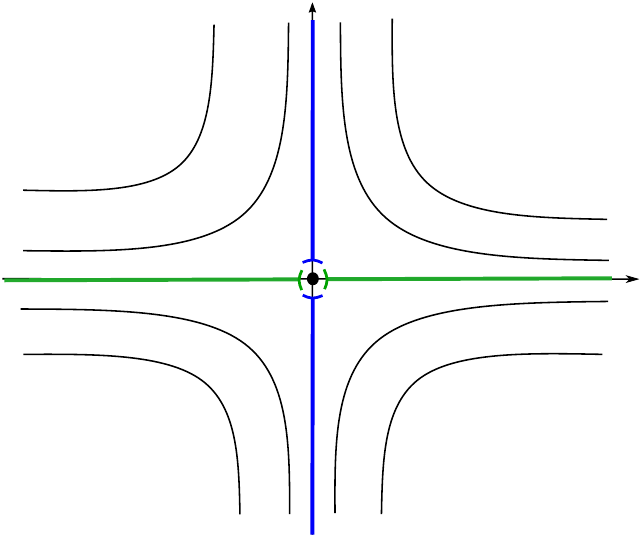}

\caption{Example of non-Hausdorff quotient}\label{non-hausdorff}
\end{figure}
\end{example}

We would like to define a quotient such that in the example, the result is $\C$. For that, two options seem to be reasonable:
\begin{enumerate}
	\item Only keep the close orbits.
	\item Identify orbits whenever their closure intersect.
\end{enumerate}
In the first case, we would ``throw away'' the two non-closed orbits, and in the second case we would identify all three points above the origin to a single point. The GIT quotient construction takes the second case, but as we will see is equivalent to the first one in the affine case.

\medskip
\paragraph{Affine GIT quotient.}
We will use the very basic idea of algebraic geometry: the dictionary between geometry and algebra. 
Consider a smooth affine variety $M$, i.e. the zero-set of a set of polynomials. Replacing these polynomials by the ideal $I$ they generate (which does not change $M$, but gives more structure) we can write $$M=\{x\in \C^n\mid P(x) = 0\forall x\in I\}.$$
To $M$, we associate the space of all polynomial functions on $M$. Such a function is a restriction of any element of $\C[x_1,...,x_n]$ to $M$. Since any element of $I$ is zero on $M$ (by definition), we get $$\Fun(M) = \C[x_1,...,x_n]/I.$$
Finally from the quotient $\Fun(M)=\C[x_1,...,x_n]/I$, we can directly get $M$: every point of $M$ corresponds to a unique maximal ideal in $\Fun(M)$. The algebraic procedure of taking all maximal ideals of a ring $A$ is called the spectrum $\mathrm{Spec}(A)$.

The idea of the GIT quotient is now simple to explain: the functions on $M/G$ have to be the $G$-invariant functions on $M$. By the correspondence between geometry and algebra, we can define the quotient this way:
\begin{equation}\label{affine-GIT}
M/_{GIT}\, G := \mathrm{Spec}(\Fun(M)^G).
\end{equation}

In Example \ref{ex-aff-git} above, where $\C^*$ acts on $\C^2$, we get $$\mathrm{Spec}(\Fun(\C^2)^{\C^*}) = \mathrm{Spec}(\C[xy])=\C.$$
Note that the function $xy$ does not distinguish between the three orbits $(0,0), (x=0, y\neq 0)$ and $(x\neq 0,y=0)$.

We see that two orbits $\mc{O}_1$ and $\mc{O}_2$ get identified if $\overline{\mc{O}}_1\cap\overline{\mc{O}}_2\neq \emptyset$: indeed any smooth $G$-invariant function is constant on $\overline{\mc{O}}_1$ and $\overline{\mc{O}}_2$. Since their closures intersect, the function takes the same value on them. Since no $G$-invariant function can distinguish the two orbits, they get identified in the GIT quotient.

Finally, notice that in the closure of any orbit $\overline{\mc{O}}$, there is a unique closed orbit. That's why we can also say that the affine GIT quotient keeps only the closed orbits.

\begin{exo}
a) Consider the action of $\mathrm{GL}_n(\C)$ on $\mathfrak{gl}_n(\C)$ by conjugation. Show that the orbit of a matrix $M$ is closed iff $M$ is diagonalizable and that its orbit is of maximal dimension if the eigenvalues are pairwise distinct.

b) Compute the GIT quotient\footnote{Hint: Use that the diagonalization is unique up to permutation, and that the ring of symmetric polynomial functions is free.}.

c) Show that the invariant functions are generated by the coefficients of the characteristic polynomial\footnote{Hint: Use the Frobenius form of the matrix.}.
\end{exo}

\begin{Remark}
The exercise generalizes to any semisimple complex Lie group $G$, the notion of diagonalizable being replaces by regular elements. The Chevalley restriction theorem gives $\mathfrak{g}^{reg}/G = \mathfrak{h}/W$ where $\mathfrak{h}$ denotes the Cartan subalgebra (the diagonal matrices above) and $W$ the Weyl group (the permutation group above). In addition $\Fun(\mathfrak{h}/W)=\C[p_1,...,p_r]$ is a free algebra generated by invariant polynomials, which play an important role in Lie theory.
\end{Remark}

Let us now turn to the character variety. The space $\Hom(\pi_1 S,G)$ is an affine variety if $G$ is an affine algebraic group. This can be seen from an explicit presentation of $\pi_1S$:
$$\pi_1 S=\langle (a_i,b_i)_{1\leq i\leq g}\mid \textstyle\prod_i [a_i,b_i]=1 \rangle.$$
This presentation comes from the fact that a surface of genus $g$ can be obtained by a special gluing of a $4g$-gon. Since the only relation is algebraic, the space $\Hom(\pi_1 S,G)$ is an algebraic subset of $G^{2g}$.

To go to the character variety, we have to quotient by the conjugation action by $G$. To get a nice quotient, we have to take the GIT quotient. In the case $G=\GL_n(\C)$, we have
\begin{thm}
A point $\rho\in\Hom(\pi_1 S,\GL_n(\C))$ is polystable iff $\rho$ is completely reducible, i.e. all $\rho$-invariant subspaces of $\C^n$ admit an invariant supplement. If $\rho$ is irreducible (there are no non-trivial $\rho$-invariant subspaces of $\C^n$), then it is a stable point.
\end{thm}
For the proof, which is quite delicate, we refer to \cite[Section 7]{sikora}. From the algebraic geometric perspective, a point $\rho\in\Hom(\pi_1S,G)$ is regular iff $\rho$ is irreducible, see \cite[Theorem 26]{gunning}.

\medskip
\paragraph{Projective GIT quotient.}
The construction of an affine GIT quotient does not always give a satisfactory answer. Consider the following example:
\begin{example}\label{ex-proj-git} Let $\C^*$ act on $\C^2$ via $\lambda.(x,y) = (\lambda x,\lambda y)$. Then all orbit closures intersect and the GIT quotient would give just one point. Of course we would like to get $\C P^1$ as quotient.
\end{example}

The idea is to slightly modify the functors $\Fun$ and $\mathrm{Spec}$ in the defining Equation \eqref{affine-GIT} above. For that, we consider a \emph{projective} variety $X\subset \C P^n$ and we suppose that $G$ acts via $G\to \SL_{n+1}(\C)$. Then the action lifts to the cone $\tilde{X} \subset \C^{n+1}-\{0\}$.
\begin{figure}[h!]
\centering
\includegraphics[height=3.8cm]{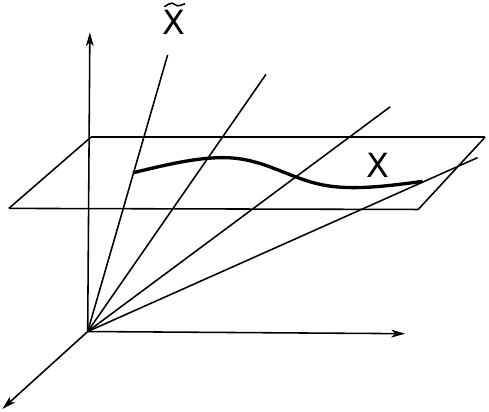}

\caption{Projective cone over $X$}\label{proj-cone}
\end{figure}

The main difference with the affine case is that on $\tilde{X}$ we have a $\C^*$-action, which induces a grading of its function space into homogeneous parts. We define 
\begin{equation}\label{proj-GIT}
M/G=\mathrm{Proj}\left(\mathrm{Projfun}(\tilde{X})^G\right)
\end{equation}
where $\mathrm{Projfun}$ is the direct sum of all homogeneous polynomial function of degree $k\geq 1$ on $\tilde{X}$, and $\mathrm{Proj}$ associates a projective variety to a ring (by taking its maximal essential ideals).

In Example \ref{ex-proj-git} above, the action extends to $\C^3$ via $\lambda.(x,y,z) = (\lambda x,\lambda y, \lambda^{-2}z)$ (since the action takes values in $\SL_3(\C)$). The invariant homogeneous polynomials are generated by monomials $x^ay^bz^{(a+b)/2}$ with $a+b$ even. Taking $z=1$ gives $\mathrm{Projfun}(\tilde{X})^{\C^*} = \oplus_{k\geq 1}\C_{2k}[x,y]$, whose associated projective variety is $\C P^1$ (with structure bundle $\mathcal{O}(2)$).

The good news is that you don't have to know much about the functors $\mathrm{Projfun}$ and $\mathrm{Proj}$ since there is a \emph{simple geometric recipe to compute the projective GIT quotient}. For that, we define types of points and treat each type differently: a point $x\in X$ is
\begin{itemize}
	\item \textbf{Unstable}, if for all non-constant homogeneous $G$-invariant polynomial $f$, we have $f(x)=0$.
	\item \textbf{Semistable}, if it is not unstable.
	\item \textbf{Polystable}, if it is semistable and has a closed orbit in $\tilde{X}$.
	\item \textbf{Stable}, if it is semistable and its orbit in $\tilde{X}$ is closed and of maximal dimension.
\end{itemize}

The recipe of the projective GIT quotient can then be given as follows: \emph{Throw away all unstable points and identify semistable points if their orbit closures in $\tilde{X}$ intersect.}

The explanation is that unstable points are not seen by $G$-invariant homogeneous functions. The same argument as for the affine GIT quotient explains why to identify orbits whose closure intersect.

In addition to the simple recipe, there is a nice geometric characterization of unstable points:
\begin{prop}
A point $x$ is unstable iff 0 is in the orbit closure of $\tilde{x}$ (any preimage of $x$ of $\tilde{X}\to X$).
\end{prop}
The so-called Hilbert--Mumford criterion states that it is sufficient to check that property for all 1-parameter subgroups of $G$.

In Example \ref{ex-proj-git}, the point $(0,0)$ is unstable since $\lambda.(0,0,1)=(0,0,\lambda^{-2}) \to 0$ for $\lambda\to \infty$. Any other point $(x,y)$ is semistable since $x^2z$ or $y^2z$ is an invariant homogenous function not vanishing on the point. They are even stable since the orbit in $\C^3$ is closed and the stabilizer is trivial. Hence, we throw away the origin and take the set-theoretic quotient of the rest: we get precisely $\C P^1$.

The projective GIT quotient is summarized in the following table:
\vspace{0.3cm}
\begin{center}
\begin{tabular}{c|c|c|c}
Type & Algebraic  & Geometric & GIT quotient \\
\hline 
Unstable & $f(x)=0 \,\forall f\in \mathrm{Fun}(\tilde{X})^G$ &   $0\in \overline{G.\tilde{x}}$ & throw away \\ 
Semistable &  not unstable & $0\notin \overline{G.\tilde{x}}$ & $\mc{O}_1\sim\mc{O}_2$ if $\overline{\mc{O}}_1\cap\overline{\mc{O}}_2\neq \emptyset$\\ 
Polystable &  semistable, closed orbit & $G.\tilde{x}$ is closed & keep \\
Stable &  maximal polystable &  $+ \;\mathrm{Stab}(\tilde{x})$ finite & keep
\end{tabular}
\end{center}
\vspace{0.5cm}

Note that in the affine case, all orbits are semistable (since we do not exclude the constant functions) and in every orbit closure there is a unique polystable orbit. This is not always true in the projective case.

\medskip
\paragraph{Link to symplectic quotient.}
The link between the GIT quotient and the symplectic quotient is given by the Kempf--Ness theorem. Roughly, it tells that
\begin{equation}
\boxed{X\sslash K \cong X/_{GIT}\, K^\C = X^{ps}/K^\C.}
\end{equation}

To be more precise:
\begin{thm}[Kempf--Ness]
Let $V$ be a complex vector space with hermitian inner product, $K\subset \mathrm{U}(V)$ a closed subgroup and put $G=K^\C$. Let $X\subset V$ be a $G$-invariant affine variety. Then the action of $K$ on $X$ is Hamiltonian and the Hamiltonian reduction $X\sslash K$ equals the GIT quotient $X/G$.
\end{thm}
Denoting by $\mu$ the moment map, one can show that $X^{ps}=G.\mu^{-1}(0)$, i.e. that a $G$-orbit intersects $\mu^{-1}(0)$ iff it is closed.
We refer to the original article \cite{kempf-ness} for the proof.

To give an example, let us reconsider the case of projective space, already seen in Example \ref{cp-n-quotient}:
\begin{example}
Consider $G=\mathrm{U}(1)\subset \C$ acting on $\C^n$ by scaling all coordinates by some factor (of module 1). We have seen in Example \ref{cp-n-quotient} that the symplectic reduction over the coadjoint orbit $\{-2\}$ gives $\C P^{n-1}$.
This coincides with the projective GIT quotient (see Example \ref{ex-proj-git}).
\end{example}

\section{Stable bundles and the Narasimhan--Seshadri theorem}\label{sec:holo-bundles}

The goal of this section is to characterize flat bundles with unitary monodromy, i.e. understand $\Rep(\pi_1\S,\mathrm{U}(n))$. We will see that the corresponding bundles carry a holomorphic structure and satisfy some stability condition (in the GIT sense). Before, we will see some basic theory about bundles.

\medskip
\paragraph{Different types of bundles.}
Up to now, we were a bit sloppy when speaking about bundles. In fact, there are three different types of bundles: real (or topological), complex and holomorphic bundles.

Before, we always treated \emph{real bundles}, where the transition functions are smooth. If the fiber is a complex vector space, we speak about a \emph{complex vector bundle}. If in addition the base manifold $X$ is complex, we can define a \emph{holomorphic bundle} to be a complex bundle with holomorphic transition functions. To summarize:
\vspace{0.3cm}
\begin{center}
\begin{tabular}{c|c|c}
Real  & Complex & Holomorphic \\
\hline
smooth transitions & complex structure on fiber & holomorphic transitions\\ 
$\check{\mathrm{H}}^1(X,\mathcal{C}^\infty(\GL_n(\R)))$ & $\check{\mathrm{H}}^1(X,\mathcal{C}^\infty(\GL_n(\C)))$ & $\check{\mathrm{H}}^1(X,\mathrm{Hol}(\GL_n(\C)))$ \\ 
\end{tabular}
\end{center}
\vspace{0.5cm}

If you know about sheaf cohomology, the last line indicates how to characterize the different types of bundles in topological terms. The set of all transition functions form a 1-cocycle in sheaf cohomology, and the isomorphism class of the bundle is described by the associated cohomology class.

The case of line bundles is particularly interesting since they form a group under tensor product. The transition functions simply get multiplied when you tensor line bundles. Since $\GL_1(\C)\cong \C^*$, the transition function vanish nowhere, so are invertible.

\begin{Remark}
Again for the reader familiar with algebraic topology, some remarks on the classification of line bundles:
\begin{itemize}
	\item Using $0\to \Z/2\Z \to \mathcal{C}^\infty(\R^*) \to \mathcal{C}^\infty(\R_{>0}) \to 0$ given by $f\mapsto f^2$, one can show that $\check{\mathrm{H}}^1(X,\mathcal{C}^\infty(\R^*))\cong \mathrm{H}^1(X,\Z/2\Z)$, so real line bundles are classified by the \emph{first Stiefel--Whitney class}.
	\item Using $0\to \Z \to \mathcal{C}^\infty(\C) \to \mathcal{C}^\infty(\C^*) \to 0$ given by $f\mapsto \exp(f)$, one can show that $\check{\mathrm{H}}^1(X,\mathcal{C}^\infty(\C^*))\cong \mathrm{H}^2(X,\Z)$, so complex line bundles are classified by the \textbf{first Chern class}. For $X=\S$ a surface, we have $H^2(X,\Z) \cong \Z$, so the first Chern class equals the \textbf{degree} of the bundle.
	\item Holomorphic line bundles are much more abundant. On a Riemann surface $S$ of genus at least 2, there are smooth families of them. The space of all line bundles is called the \emph{Picard variety} $\mathrm{Pic}(S)$ and is given by $\mathrm{Pic}(S) \cong \Z\times \mathrm{Jac}(S)$ where the first factor gives the degree and the second factor is the \emph{Jacobian variety}.
\end{itemize}
\end{Remark}

Let us give some examples of line bundles over $X=\C P^1$.
\begin{example}
The Riemann sphere $\C P^1$ is given by two charts $U_0 = \C$ and $U_1=\C$ with transition map $U_0\cap U_1 = \C^* \to \C^*$ given by $z\mapsto 1/z$. Define the line bundle $\mathcal{O}(k)$ in the following way: it is made of two pieces, $U_0\times \C$ and $U_1\times \C$ which are glued together via
$$\left \{ \begin{array}{cl}
 (U_0\cap U_1)\times \C &\to (U_0\cap U_1)\times \C\\
(z,v) &\mapsto (1/z, z^kv)
\end{array} \right.$$
It turns out that the $\mathcal{O}(k)$ describe all holomorphic line bundles: $\mathrm{Pic}(\C P^1) \cong \Z$.
\end{example}

To any vector bundle $E\to M$, we can associate a line bundle in the following way. There is an open covering of $M$ by open sets $U_\alpha$ where $E$ is trivial and with transition maps $f_{\alpha \beta}$. Define the \textbf{determinant line bundle} $\det(E)$ to be the line bundle which is trivial over the $U_\alpha$ with transition functions $\det(f_{\alpha \beta})$. The first Chern class is preserved: $c_1(E)=c_1(\det E)$.

To understand a bundle, we can try to decompose it into simpler pieces, to write it as a direct sum. On real or complex line bundles, we can put a scalar or hermitian product on the fibers (varying smoothly) which allows to construct canonically a complement to any subbundle. This means that whenever we have a subbundle $F\subset E$, we can write $E=F\oplus F'$. We say that a bundle is \textbf{irreducible} if it has no non-trivial subbundles. 

This is not true for holomorphic bundles since there might be no hermitian structure varying holomorphically with the point. We speak about an \textbf{indecomposable} holomorphic bundle whenever it cannot be written as a direct sum of two other bundles.

Since line subbundles are nothing but non-vanishing sections, we can look for these to decompose the bundle:
\begin{prop}
Let $E\to M$ be a vector bundle with $M$ a real $m$-dimensional manifold. 
\begin{itemize}
	\item If $E$ is real of rank more than $m$, then there is a non-vanishing section.
	\item If $E$ is complex of (complex) rank more than $m/2$, then there is a non-vanishing section.
\end{itemize}
\end{prop}
The proof is simple: a generic perturbation of any section works.

\begin{coro}
A complex bundle $E\to \S$ of rank $k$ over a surface is isomorphic to $\det(V)\oplus \C^{k-1}$ (as complex bundle).
\end{coro}
\begin{proof}
By the previous proposition we find $k-1$ non-vanishing independent sections. So $V\cong L\oplus \C^{k-1}$ where $L$ is some line bundle. Then $\det(V) = \det(L) = L$.
\end{proof}

Now we can give an example illustrating the difference between the three kinds of bundles:
\begin{example}
On $\C P^1$, consider the holomorphic bundle $\mc{O}(1)\oplus \mc{O}(1)$. Its underlying complex vector bundle is $\C\oplus T\C P^1$, since it can be easily checked that $\det(\oplus \mc{O}(a_i))=\mc{O}(\sum a_i)$ and $\mc{O}(2)$ is the tangent bundle of $\C P^1$. Note that $\C\oplus T\C P^1$ is not trivial, since there is no non-vanishing section of $T\C P^1$ (the ``hairy ball theorem'').

Finally, its underlying real bundle is \emph{trivial}. You can check that $T(\bb{S}^2) \oplus \R$ is a trivial bundle. You just take the unit sphere $\bb{S}^2\subset \R^3$ and consider its normal bundle, which is trivial (there is a constant section). Hence its direct sum with the tangent bundle gives the tangent bundle of $\R^3$, which is trivial, restricted to the sphere. Therefore $T(\bb{S}^2)\oplus \R^2 \cong \R^4$.
\end{example}

\begin{Remark}
A theorem of Grothendieck asserts that the set of all holomorphic bundles on $\C P^1$ are the $\oplus \mc{O}(a_i)$. This can be proven using methods from loop groups.
\end{Remark}

\medskip
\paragraph{Flat bundles.}
Now that we have seen the different kind of bundles, we may ask which bundles arise in the Riemann--Hilbert correspondence.

Recall that to a representation $\rho:\pi_1M\to \GL_n(\C)$, we associate the bundle $E_\rho=(\widetilde{M}\times \C^n)/\pi_1M$. This is a complex vector bundle which is trivial, since $\pi_1M$ acts faithfully on $\widetilde{M}$. On a surface, complex bundles are classified by their degree. We have:
\begin{prop}
A complex vector bundle $E$ over a surface $\S$ admits a flat connection iff $\deg(E) =0$ (iff $E$ is trivial).
\end{prop}
Over a general manifold, a complex vector bundle is trivial iff all its Chern classes vanish.

\begin{Remark}
More generally, there is a beautiful link between characteristic classes and the curvature of a connection, described by \emph{Chern--Weil theory}. For the degree of a bundle $E$ over $M$, the link is $$\deg(E) = c_1(E) = \left[\frac{i}{2\pi}\tr F(A) \right] \in H^2(M,\Z).$$
\end{Remark}

Since trivial complex bundles can have non-trivial holomorphic structures, we can ask which \emph{holomorphic} bundles can arise through the Riemann--Hilbert correspondence? For this to make sense, we have to equip the smooth surface $\S$ with a complex structure. The resulting Riemann surface is denoted by $S$.

We have seen the notion of indecomposable holomorphic bundles (not the direct sum of others). By definition, they form the building blocks for all holomorphic bundles:
\begin{prop}
Every holomorphic bundle over a compact Riemann surface is the direct sum of indecomposable bundle in a unique way (up to permutation of the factors).
\end{prop}

For holomorphic bundles, we have the following theorem due to Weyl:
\begin{thm}[Weyl]
A holomorphic bundle $E=\oplus E_i$ which is a direct sum of indecomposable bundles $E_i$ admits a flat connection if and only if $c_1(E_i)=0 \,\forall\, i$.
\end{thm}
A proof can be found in \cite[Theorem 16]{gunning}.

\medskip
\paragraph{Cauchy--Riemann operator.}
On a holomorphic bundle $E$, there is a natural connection $\delbar_E$ generalizing the Cauchy--Riemann operator (which defines what a holomorphic function is). It turns out that holomorphic structures are in bijection with these operators.

A \textbf{Cauchy--Riemann operator} $\delbar_E$ (also called \emph{Dolbeault operator}) on a complex bundle $E$ over a complex manifold $M$ is a connection, i.e. a map $\Gamma(E)\otimes \Gamma(TM) \to \Gamma(E)$, such that $\delbar_E^2=0$ and
\begin{equation}\label{cauchy-riemann}
\delbar_E(fs) = (\delbar f) s+f\delbar_E(s)
\end{equation}
for all sections $s\in\Gamma(E)$ and functions $f\in \mathcal{C}^\infty(M)$. Note that $\delbar$ is the usual Cauchy--Riemann operator, which is well-defined since $M$ is a complex manifold.

\begin{prop}
A holomorphic structure on a complex vector bundle $E$ is equivalent to the existence of a Cauchy--Riemann operator $\delbar_E$.
\end{prop}
We refer to the thesis of McCarthy \cite[Theorem 3.3.3]{mccarthy} and \cite[p.555]{atiyah-bott} for a proof. The delicate point is to construct holomorphic transition functions from $\delbar_E$ which leads to an elliptic system to solve.

On a Riemann surface $S$, a complex bundle admits plenty of holomorphic structures: any connection gives one. Indeed we can decompose a connection $\nabla=\nabla^{1,0}+\nabla^{0,1}$ into the $dz$ and $d\bar{z}$-part. Then $\nabla^{0,1}$ is a Cauchy--Riemann operator (since $d=\del+\delbar$). 

Consider now a trivial complex bundle $E$ over $S$. Using the differential $d$ as canonical base point, we can identify the space of $G$-connections with $\Omega^1(S,\g)$. Hence, we can describe holomorphic structures by $\Omega^{0,1}(S,\g)$.
Then, two holomorphic structures are equivalent (under some bundle automorphism) iff the corresponding operators are gauge-equivalent. Note that this action is given by 
$$g.B = gBg^{-1}+g\delbar\left(g^{-1}\right).$$ 
Denote by $\Hol(S,\GL_n(\C))$ the moduli space of holomorphic structures on $E=S\times \C^n$ and by $\mc{G}^\C = \mc{G}(\GL_n(\C))$ the gauge group. Since the trivial complex bundle has degree 0, we have 
\begin{equation}\label{hol-quotient}
\Hol_{\deg=0}(S,\GL_n(\C)) \cong \Omega^{0,1}(S,\mf{gl}_n(\C))/\mc{G}^\C.
\end{equation}

\medskip
\paragraph{Stability for bundles.}
The space of holomorphic structures is given by the quotient \eqref{hol-quotient}. To get a nice space, we have to interpret this quotient in the GIT sense. The associated stability condition is described now.

For a holomorphic bundle $E\to M$, we define its \textbf{slope} $\mu(E)$ by 
\begin{equation}
\mu(E) = \frac{\deg (E)}{\rk(E)}.
\end{equation}
To memorize: the degree can be zero, so cannot be in the denominator. 

\begin{definition}[Mumford]
A holomorphic bundle $E$ is \textbf{stable} if for all holomorphic subbundles $F\subset E$ we have $\mu(F)<\mu(E)$. The bundle is semistable if the inequality is not strict.
\end{definition}

Let us see two properties of stable bundles:
\begin{prop}
Let $E$ be a stable holomorphic bundle. Then
\begin{enumerate}
	\item $E$ cannot be a direct sum $E_1\oplus E_2$.
	\item $E$ has only trivial holomorphic automorphisms (of the form $\lambda \id$ for some constant $\lambda$).
\end{enumerate}
\end{prop}
\begin{proof}
Part (1) is simply proved by contradiction: if $E=E_1\oplus E_2$, then $\deg(E) = \deg(E_1)+\deg(E_2)$ and $\rk(E)=\rk(E_1)+\rk(E_2)$. Since $E$ is stable, we have $\mu(E_1)<\mu(E)$ and $\mu(E_2)<\mu(E)$ which leads to a contradiction: $$\deg(E) =\deg(E_1)+\deg(E_2) < \deg(E)\left(\frac{\rk(E_1)}{\rk(E)}+\frac{\rk(E_2)}{\rk(E)}\right) = \deg(E).$$

For part (2), consider a holomorphic automorphism $\varphi: E\to E$. Then its characteristic polynomial is constant, since its coefficients are holomorphic functions on a compact manifold. If there are at least two different eigenvalues, we can decompose $E$ into a direct sum which is impossible by (1). So there is a constant $\lambda\in \C^*$ such that $\varphi =\lambda \id+\psi$ where $\psi$ is nilpotent. Suppose $\psi\neq 0$. Since $\psi$ is still a holomorphic automorphism of $E$, we have $\mu(E) > \mu(\mathrm{Im}(\psi))$ and $\mu(E)>\mu(\mathrm{ker}(\psi))$ by stability. Since $\mathrm{Im}(\psi)=E/\mathrm{ker}(\psi)$, we also get $\mu(E) < \mu(\mathrm{Im}(\psi))$, a contradiction. Hence $\varphi=\lambda \id$.
\end{proof}

The second property gives a hint why Mumford's stability is the appropriate notion, since stable objects have usually a small automorphism group.

\medskip
\paragraph{Unitary character varieties.}
We are now ready to characterize the holomorphic bundles arising through unitary representations of the fundamental group:
\begin{thm}[Narasimhan--Seshadri \cite{ns}]\label{ns-thm}
The character variety for the unitary group is in bijection with semistable holomorphic bundles of degree 0:
$$\boxed{\Rep(\pi_1\S,\mathrm{U}(n)) \cong \Hol^{ss}_{\deg=0}(S,\GL_n(\C)).}$$
Moreover, the stable points correspond to irreducible representations.
\end{thm}

In other words: a holomorphic bundle $E=\oplus E_i$ with $E_i$ indecomposable, comes from a unitary representation iff all $E_i$ are stable and of degree 0.
Note that we can generalize Narasimhan--Seshadri theorem to any group $K\subset \mathrm{U}(n)$: 
$$\Rep(\pi_1\S,K) \cong \Hol^{ss}_{\deg=0}(S,K^\C)$$
where the holomorphic bundles have structure group $K^\C$.

Let us try to understand the profound meaning of the theorem. One direction is less surprising: given a representation coming from a flat connection $\nabla$, its $(0,1)$-part defines a holomorphic structure. The theorem asserts that this holomorphic bundle is semistable.

The other direction is much more surprising: given a stable holomorphic bundle, there is preferred \emph{flat} unitary connection on it! This is surprising since a holomorphic structure is given by a $(0,1)$-part of a connection. Looking for unitary connections, we can complete this $(0,1)$-part into a full connection, but there is no reason for that connection to be flat! The point is that we work on the level of moduli spaces, so within gauge equivalence classes. The theorem asserts that in the complex gauge-orbit of a stable holomorphic structure, there is a representative, given by a Cauchy--Riemann operator $\delbar_E$, whose associated unitary connection is flat.

We present the idea of the proof in layers, like peeling an onion. The proof strategy presented here is due to Donaldson \cite{donaldson-NS} which uses the ideas of the Kempf--Ness theorem.

\begin{proof}[Idea of proof]
We start by the Atiyah--Bott reduction which gives
$$\Rep(\pi_1\S,\mathrm{U}(n)) \cong \mc{A}(\mathfrak{u}(n))\sslash\mc{G}$$
where $\mc{A}(\mathfrak{u}(n))$ denotes the space of all unitary connections on a trivial complex bundle $V$ over $\S$ equipped with a hermitian structure $h$. $\mc{G}=\mc{G}(\mathrm{U}(n))$ denotes the unitary gauge group.

By the principle of the Kempf--Ness theorem, we have
$$\mc{A}(\mathfrak{u}(n))\sslash\mc{G} \cong \mc{A}^{ss}(\mathfrak{u}(n))/ \mc{G}^\C$$
where $\mc{G}^\C$ denotes the complex gauge group. An element $g\in \mc{G}^\C$ acts on a unitary connection $d+A$ in two steps: first act only on $A^{(0,1)}$ by a gauge transformation, i.e. $g.A^{(0,1)}=gA^{(0,1)}g^{-1}+g\delbar(g^{-1})$, and second complete the result to a unitary connection. Note that the total action is not a gauge action (unless $g\in \mc{G}$) which allows to modify the curvature.

Since a unitary connection is uniquely determined by its $(0,1)$-part:
$$\mc{A}^{ss}(\mathfrak{u}(n))/ \mc{G}^\C \cong \Omega^{ss, (0,1)}(\mathfrak{gl}_n(\C))/\mc{G}^\C.$$

Finally, since a Cauchy--Riemann operator determines a holomorphic structure, we get from Equation \eqref{hol-quotient}:
$$ \Omega^{ss, (0,1)}(\mathfrak{gl}_n(\C))/\mc{G}^\C \cong \Hol^{ss}_{\deg=0}(S,\GL_n(\C)).$$
\end{proof}

This elegant proof, combining nicely all the material we have seen before, is only the first layer, lacking lots of important details. In a second layer, one has to prove two things: the stability condition appearing in the proof idea is identical with the Mumford slope-stability, and the Kempf--Ness theorem can be adapted to the infinite-dimensional setting.

To carry out the latter, Donaldson imitates the proof strategy of the Kempf--Ness theorem. The rough idea is to show that the complex gauge orbit of a unitary connection intersects the zero-set of the moment map (the flatness condition) iff the connection is semistable. To achieve this, one uses a gradient descent method. The function we consider for that is simply the norm of the moment map (for some adapted $L^2$-norm): $$A\mapsto \left\|F(A)\right\|^2.$$
This is called the \emph{Yang--Mills functional}. The absolute minima of this function are obviously given by flat connections. 

So you start with a point, apply the gradient flow to get a sequence of connections in the same complex gauge orbit. What you have to show is that you converge to an absolute minimum iff your starting point is semistable.

This can be done in a third layer using the Uhlenbeck--Yau compactness theorem. See Donaldson's paper \cite{donaldson-NS} for details.

\medskip
We might ask, what is so special about bundles coming from unitary representations? One aspect is the following: since the transition functions are unitary, they are in particular bounded. Hence any holomorphic section of $E$ is constant, by the maximum principle ($S$ is compact and transitions are bounded).

To understand character varieties for non-unitary groups, especially non-compact groups, we need the notion of Higgs bundles.

\section{Higgs bundles and the non-abelian Hodge correspondence}\label{Sec-Higgs-bundles}

In this section we will see how the notion of a Higgs bundle naturally arises. We then state the main theorem of our lecture, the non-abelian Hodge correspondence.

\medskip
\paragraph{Cotangent space to $\Hol(S,\GL_n(\C))$.}

The Narasimhan--Seshadri theorem has told us that stable holomorphic bundles correspond to irreducible unitary representations of $\pi_1\S$. In some sense, the holomorphic structure of a bundle is encoded in a $(0,1)$-form which can be uniquely completed to a unitary connection. 

To describe representations $\pi_1\S\to \GL_n(\C)$, we need more than holomorphic bundles: something which is encoded both in a $(0,1)$-form and a $(1,0)$-form. Put $\g=\mathfrak{gl}_n(\C)$. We remark that 
\begin{equation}\label{cotangent01}
T^*\Omega^{0,1}(S,\g) \cong \Omega^{0,1}(S,\g)\oplus \Omega^{1,0}(S,\g).
\end{equation}
Indeed, the tangent space at any point is given by $\Omega^{0,1}(S,\g)$ itself (since it is a vector space). For $\alpha\in\Omega^{1,0}(S,\g)$ and $\beta\in\Omega^{0,1}(S,\g)$ the map
$$\langle \alpha,\beta \rangle = \int_S \tr \alpha\wedge \beta$$
is a non-degenerate pairing, which explains Equation \eqref{cotangent01}.

This leads to the idea to consider the cotangent bundle of $\Hol(S,\GL_n(\C))$.
\begin{prop}\label{Prop:higgs-bundle-cotangent-vector}
$$T^*\Hol(S,\GL_n(\C)) \cong \{(\bar{A},\Phi)\in \Omega^{0,1}\times \Omega^{1,0}\; \big| \; \delbar\Phi+[\bar{A},\Phi]=0\}/\mathcal{G}$$
where the gauge group $\mathcal{G}$ acts by $g.\bar{A} = gAg^{-1}+gd\left(g^{-1}\right)$ and $g.\Phi = g\Phi g^{-1}$.
\end{prop}

The gauge action is explained by the idea that $\Phi$ is a cotangent vector, so gives a small deformation $\delbar+\varepsilon \Phi+\bar{A}$. Since the gauge parameter does not involve $\varepsilon$, it acts on $\Phi$ simply by conjugation.

The main ingredient to proof the proposition is the formula $T^*(X/G) = T^*X\sslash G$ seen in Exercise \ref{cotang-group-quotient}.
\begin{proof}
Since $\Hol(S,\GL_n(\C))=\Omega^{0,1}(S,\g)/\mathcal{G}$, we have
$$T^*\Hol(S,\GL_n(\C)) = T^*\Omega^{0,1}(S,\g) \sslash \mathcal{G} = (\Omega^{0,1}(S,\g)\oplus \Omega^{1,0}(S,\g)) \sslash \mathcal{G}$$
by Equation \eqref{cotangent01}. The symplectic form is given by
$$\omega = \int_S \tr \delta \bar{A}\wedge \delta \Phi.$$

To compute the moment map, we first compute the infinitesimal gauge action by $g=1+\e$. We easily get the vector fields representing an infinitesimal change
$$\bar{A}_\e = -\delbar \e+[\e,\bar{A}]$$
and 
$$\Phi_\e= [\e,\Phi].$$

By the usual procedure, we compute
\begin{align*}
\iota_{(\bar{A}_\e,\Phi_\e)}\omega(\delta \bar{A},\delta \Phi) &= \int_S \tr (\bar{A}_\e \delta \Phi-\delta\bar{A}\Phi_\e) \\
&= \int_S \tr((-\delbar\e+[\e,\bar{A}])\delta\Phi-\delta\bar{A}[\e,\Phi]) \\
&= \int_S \tr \e(\delbar\delta \Phi+[\bar{A},\delta \Phi]+[\delta \bar{A},\Phi]) \\
&= \delta\left(\int_S \tr \e(\delbar\Phi+[\bar{A},\Phi])\right).
\end{align*}

Hence, the moment map is given by $\mu(\bar{A},\Phi) = \delbar\Phi+[\bar{A},\Phi]$, which gives the Proposition.
\end{proof}

We can understand the Proposition in more conceptual terms: $\bar{A}$ defines a Cauchy--Riemann operator on $E$, which induces a holomorphic structure on $\End(E)$ given by a Cauchy--Riemann operator $\delbar_{\End(E)}$. On a section $\Phi$ of $\End(E)$ it acts via 
$$\delbar_{\End(E)}\Phi = \delbar{\Phi}+[\bar{A},\Phi].$$

If $\Phi$ is the $\End(E)$-valued $(1,0)$-form from above, the Proposition tells us that $\delbar_{\End(E)}\Phi=0$, hence $\Phi$ is holomorphic. Such an object is called a \textbf{Higgs field}. In technical terms, we have
$$\Phi \in \mathrm{H}^0(S,\End(E)\otimes K)$$
where $K$ denotes the canonical bundle (holomorphic $(1,0)$-forms), but you really should think of a Higgs field as a cotangent vector to the moduli space of holomophic structures.

\begin{definition}
A \textbf{Higgs bundle} is a holomorphic bundle $E$ equipped with a Higgs field $\Phi\in \mathrm{H}^0(S,\End(E)\otimes K)$.
\end{definition}

\medskip
\paragraph{Moduli space.}
We want to define the moduli space of Higgs bundles by $$\mc{M}_H(S,\GL_n(\C)) = \{\text{Higgs bundles}\}/\mc{G}.$$
As usual, we have to take the GIT quotient to get a nice topological space (in fact we get a manifold).

The appropriate stability condition is the following:
\begin{definition}
A Higgs bundle $(E,\Phi)$ is \textbf{stable}, if for all $\Phi$-invariant holomorphic subbundles $F\subset E$ we have $\mu(F)<\mu(E)$. It is called \textbf{semistable} if the inequality is not necessarily strict.
\end{definition}

Note that $\Phi$-invariant means that for all vector fields $X\in \Gamma(TS)$, we have $\Phi(X).F\subset F$.

\begin{example}
If $E$ is a stable holomorphic bundle, then $(E,\Phi)$ is a stable Higgs bundle for all Higgs fields $\Phi$.
\end{example}

\begin{example}\label{higgs-basic-ex}
Fix a so-called spin structure on $S$, i.e. a line bundle denoted by $K^{1/2}$ whose square is the canonical bundle $K$. Then, consider $$\left(E=K^{1/2}\oplus K^{-1/2}, \Phi=\begin{pmatrix}0 & 0 \\ 1 & 0\end{pmatrix}\right).$$
Note that the non-zero entry in $\Phi$ makes sense since it is an element of $$\Hom(K^{1/2},K^{-1/2})\otimes K \cong K^{-1}\otimes K \cong \mc{O}$$ which is the trivial line bundle. It is easy to check that the only non-trivial $\Phi$-invariant subbundle is $K^{-1/2}$. Since $\deg(K^{-1/2}) < 0 = \deg(E)$, we get that $(E,\Phi)$ is a stable Higgs bundle.
\end{example}
Note that in the last example, $E$ is not stable as holomorphic bundle since it has $K^{1/2}$ as holomorphic subbundle.

From the two examples, we see that
\begin{equation}\label{higgs-moduli-cotangent-bundle}
T^*\mathrm{Hol}^{s}(S,\mathrm{SL}_n(\C)) \subset \mathcal{M}_H \subset T^*\mathrm{Hol}(S,\mathrm{SL}_n(\C)),
\end{equation}
i.e. the moduli space of Higgs bundle sits between two cotangent bundles of holomorphic structures.

\begin{prop}
Stability is an open condition, i.e. if $(E,\Phi)$ is a stable Higgs bundle, then any Higgs bundle sufficiently close to it is also stable.
\end{prop}

\medskip
\paragraph{Non-abelian Hodge correspondence}
We are now ready to state the main theorem of our lectures:
\begin{thm}[non-abelian Hodge correspondence]\label{non-ab-Hodge-thm}
The moduli space of polystable Higgs bundles of degree 0 is diffeomorphic to the $\GL_n(\C)$-character variety: $$\boxed{\mc{M}^{ps}_{H, \deg=0}(S,\GL_n(\C)) \cong \Rep^{c.r.}(\pi_1\S,\GL_n(\C)).}$$
\end{thm}

The non-abelian Hodge correspondence gives a remarkable link between purely topological objects, completely reducible representations of the fundamental group, and holomorphic objects, polystable Higgs bundles. Irreducible representations correspond to stable Higgs bundles.

It generalizes the Narasimhan--Seshadri theorem: we will see that an irreducible unitary representation corresponds to a stable Higgs bundle $(E,\Phi)$ with vanishing Higgs field $\Phi=0$. So necessarily $E$ is stable.

This deep theorem is due to many people, above all Hitchin \cite{hitchin1987self}, Simpson \cite{simpson}, Corlette \cite{corlette} and Donaldson \cite{donaldson}. A nice and concise account for the main proof ideas together with crucial steps are given in \cite{wentworth2016}.

\medskip
Why is it called ``non-abelian Hodge correspondence''? The answer here is inspired by the introduction of Simpson's paper \cite{simpson}.

A basic motto in algebraic topology is that \emph{all cohomology theories are (more or less) equivalent}. This is why there is an axiomatization of cohomology (by Eilenberg--Steenrod). One instance of this motto is the \emph{Hodge correspondence}, which gives a link between de Rham cohomology (defined using differential geometry) and Dolbeault cohomology (defined using holomorphic objects):
$$H^k_{dR}(X,\C) \cong \bigoplus_{p+q=k}H^{p,q}_{Dol}(X).$$

For $k=1$ we get $$H^1(S,\C) \cong H^{0,1}(S) \oplus H^{1,0}(S) \cong H^1(S,\mc{O}_S) \oplus H^0(S,\Omega^1(S)),$$
where we used sheaf cohomology in the last term.

Replacing $\C$ by $\GL_n(\C)$, which is non-abelian for $n>1$, we can interpret $\Rep(\pi_1 S,\GL_n(\C))$ as $H^1(S,\GL_n(\C))$ since for $n=1$:
$$H^1(S,\GL_1(\C)) = H^1(S,\C^*) \cong \Hom(\pi_1 S,\C^*) = \Rep(\pi_1 S,\C^*)$$
where we used in the last equality that $\C^*$ is abelian.

The generalization of $H^1(S,\mc{O}_S)$ is $\check{H}^1(S,\Hol(\GL_n(\C)))$ which describes holomorphic bundles of rank $n$. Finally $H^0(S,\Omega^1(S))$ becomes
$H^0(S,\GL_n(\C)\otimes K)$ giving the Higgs field.

The analogy can even be enlarged: the most basic cohomology theories (simplicial, singular or cellular) are purely topological. They are sometimes calles Betti cohomology. By the basic motto, they all coincide with the de Rham or Dolbeault cohomology. 

The non-abelian analogs are (names were given by Simpson):
\begin{itemize}
	\item \textbf{Betti moduli space}: space of representations of the fundamental group, i.e. the character variety $\Rep(\pi_1\S,\GL_n(\C))$.
	\item \textbf{de Rham moduli space}: space of flat connections.
	\item \textbf{Dolbeault moduli space}: space of Higgs bundles.
\end{itemize}

The Betti and de Rham moduli space are equivalent by the Riemann--Hilbert correspondence. They are equivalent to the Dolbeault moduli space by the non-abelian Hodge correspondence.

We will see even more analogies: in the presence of a metric there is a preferred representative in each de Rham cohomology class, a harmonic form. The analog leads to the notion of harmonic bundles which are the key ingredient to prove the non-abelian Hodge correspondence.

\section{The proof strategy: Harmonic bundles}\label{harmonic-bundles}

In this section we will see two notions of harmonic bundles, giving representatives for flat bundles and Higgs bundles respectively. The existence of harmonic representatives are described by the theorems of Corlette--Donaldson and Hitchin--Simpson, which together give the non-abelian Hodge correspondence.

\medskip
\paragraph{Hermitian bundles.}
The basic ingredient to harmonic theory is the notion of a \emph{hermitian bundle}, which is a complex bundle with a hermitian product $(.,.)$ in each fiber, varying in a smooth manner.

In a holomorphic bundle $E$, a hermitian structure determines a preferred connection, similar to the Levi--Civita connection for a Riemannian manifold:
\begin{prop}
In a hermitian holomorphic bundle $E$, there is a unique connection $\nabla$, called the \textbf{Chern connection}, which is compatible with 
\begin{enumerate}
	\item the holomorphic structure: $\nabla^{0,1} = \delbar_E$,
	\item the hermitian structure: $d(s_1,s_2) = (\nabla s_1,s_2)+(s_1,\nabla s_2)$ for all sections $s_1,s_2$.
\end{enumerate}
\end{prop}

Let us analyze the space of all hermitian structures on a given flat bundle $E=E_\rho$, where $\rho:\pi_1\S\to G=\GL_n(\C)$ is the monodromy. At one point it is described by
$$\mathrm{Herm}^{++}=\{H\in \mathfrak{gl}_n\mid H^\dagger = H, \text{ positive definite}\}.$$
Indeed, a matrix $H\in \mathrm{Herm}^{++}$ determines a Hermitian product by $(x,y)_H = x^\dagger H y$. On $\mathrm{Herm}^{++}$, there is an action of $G=\GL_n(\C)$ by $(x,y)_{g.H}=(gx,gy)_H$ (where $g\in \GL_n(\C)$). Hence it is given by $g.H=g^\dagger H g$. The action is transitive and the stabilizer of $\id \in \mathrm{Herm}^{++}$ is $K=\mathrm{U}(n)$. Therefore
$$\mathrm{Herm}^{++} \cong \GL_n(\C) / \mathrm{U}(n)=G/K.$$

Note that $\mathrm{U}(n)$ is not a normal subgroup, so the quotient is merely a set.

Locally, in a given trivialization of $E$, a hermitian structure is given by a map $U\subset \S \to G/K$. Globally over $\S$, the map is not well-defined since taking a non-trivial loop $\gamma\in \pi_1\S$ results in a conjugated $G/g^\dagger K g$ where $g=\rho(\gamma)$. 

To get a well-defined map, we have to consider the universal cover $\widetilde{\S}$ and maps $\widetilde{\S}\to G/K$ which are equivariant with respect to $\pi_1\S$, which acts on the universal cover by deck transformations and on $G/K$ via $\rho$.

\begin{prop}\label{prop:harm-metric}
The space of hermitian structures on $E=E_\rho$ can be identified with the space of $\pi_1\S$-equivariant functions $u:\widetilde{\S}\to G/K$.
\end{prop}
For more details, we refer to \cite[Prop. 2]{toulisseexistence}.

\medskip
\paragraph{Harmonic flat bundles and Corlette--Donaldson theorem.}
For a point in the moduli space of flat connections (a flat connection modulo gauge equivalence), we wish to define a nice representative, a harmonic flat bundle. This representative should exist whenever the point in the de Rham moduli space is semistable. We have seen in Section \ref{Sec-GIT} that this is the case whenever the monodromy is completely reducible. 

Thus, to find the appropriate notion of a harmonic flat bundle, we have to translate the property of being completely reducible from the Betti moduli space to the de Rham moduli space, i.e. in terms of flat bundles.

Recall the Riemann--Hilbert correspondence:
$$\rho\in \Rep(\pi_1\S,\GL(V)) \mapsto (E_\rho,\nabla) \text{ flat bundle given by } E_\rho = (\widetilde{\S}\times V)/\pi_1\S.$$

Clearly, a subrepresentation corresponds to a $\nabla$-invariant subbundle. Thus, completely reducible representations correspond to completely reducible flat bundles.

Consider for example a hermitian structure $h$ on the vector space $V$ such that $\mathrm{Im}(\rho)\subset U(V,h)$. If $F$ is a $\nabla$-invariant subbundle, then $F^{\perp_h}$ as well and $\nabla = \nabla_F+\nabla_{F^\perp}$ since $\nabla$ is unitary.

\begin{goal}
Find a condition on $\nabla$, where $(E,\nabla)$ is a flat bundle, which ensures complete reducibility of $E$.
\end{goal}

For a fixed hermitian structure $h$, we can decompose $\nabla = d_A+\Psi$ where $d_A$ is a unitary connection and $\Psi$ is the hermitian part of $\nabla$. Locally, this is nothing but writing a matrix as a sum of a hermitian and an anti-hermitian matrix.

Consider $F\subset E$ a $\nabla$-invariant subbundle. As complex bundles, we have $E=F\oplus F^\perp$. So we can write 
$$\nabla = \begin{pmatrix}\nabla_1 & \eta \\ 0& \nabla_2\end{pmatrix} \;\text{ with }\; \eta\in \Omega^1(\S,\Hom(F^\perp,F)).$$

Being reducible means that $\eta=0$, so we look for a condition which forces $\eta$ to vanish. For that, the idea due to Corlette is to consider $\End(E)$ with induced connection $d_A$ and the section $s=-\id_F\oplus \id_{F^\perp}$. Decompose $\nabla_i=d_{A_i}+\Psi_i$ into unitary and hermitian parts, hence
$$\nabla = \begin{pmatrix}\nabla_1 & \eta \\ 0& \nabla_2\end{pmatrix} = \begin{pmatrix}d_{A_1} & \eta/2 \\ -\eta^*/2& d_{A_2}\end{pmatrix}+\begin{pmatrix}\Psi_1 & \eta/2 \\ \eta^*/2& \Psi_2\end{pmatrix}.$$

Then compute $d_As$:
$$d_As=\begin{pmatrix}d_{A_1} & 0 \\ 0& d_{A_2}\end{pmatrix}\begin{pmatrix}-\id_F & 0 \\ 0& \id_{F^\perp}\end{pmatrix}+\left[\begin{pmatrix}0 & \eta/2 \\ -\eta^*/2& 0\end{pmatrix},\begin{pmatrix}-\id_F & 0 \\ 0& \id_{F^\perp}\end{pmatrix}\right]=\begin{pmatrix}0 & \eta \\ \eta^*& 0\end{pmatrix}.$$

Hence $$\langle \Psi, d_As\rangle_{L^2} = \int_\S (\Psi,d_As)_h = \int_\S \tr \begin{pmatrix}\Psi_1 & \eta/2 \\ \eta^*/2& \Psi_2\end{pmatrix}\begin{pmatrix}0 & \eta \\ \eta^*& 0\end{pmatrix} = \langle \eta,\eta\rangle_{L^2}.$$

We also have $\langle \Psi,d_As\rangle = \langle d_A^*\Psi,s\rangle$ using the adjoint. So to get $\eta=0$, it is sufficient to require $$d_A^*\Psi = 0.$$

\begin{definition}\label{def-harmonic-bundle}
A metric $h$ on $(E,\nabla)$ is \textbf{harmonic} if $d_A^*\Psi=0$ where $\nabla = d_A+\Psi$.
\end{definition}
A flat bundle equipped with a harmonic hermitian metric is called \emph{harmonic flat bundle}. 

\begin{prop}\label{Prop:herm-flat-bundle}
A bundle $(E,\nabla=d_A+\Psi,h)$ is a harmonic flat bundle iff $F(A)+\Psi\wedge\Psi=0$ and $d_A\Psi=0=d_A^*\Psi$.
\end{prop}
\begin{proof}
The onyl thing to check is that $\nabla$ is flat. The curvature of $\nabla$ is $F(A)+\Psi\wedge\Psi+d_A\Psi$. The unitary part $F(A)+\Psi\wedge\Psi$ and the hermitian part $d_A\Psi$ have to vanish both.
\end{proof}

The main theorem which will give half of the non-abelian Hodge correspondence is the Corlette--Donaldson theorem \cite{corlette, donaldson}:

\begin{thm}[Corlette--Donaldson]\label{corlette-donaldson-thm}
A flat bundle $(E,\nabla)$ admits a harmonic metric iff it is completely reducible, i.e. iff its monodromy is completely reducible. In addition, this harmonic metric is unique up to an overall positive constant factor.
\end{thm}

In other words $$\boxed{\Rep^{c.r.}(\pi_1\S,G) \cong \{\text{flat } G\text{-bundles}\}/\text{gauge} \cong\{\text{harmonic }G\text{-bundles}\}/\text{constants}.}$$
This is the analog of the harmonic representative in each de Rham cohomology class.

Let us see how to prove one direction of the non-abelian Hodge correspondence from the Corlette--Donaldson theorem: we can associate a Higgs bundle to a given flat bundle $(E,\nabla)$. Use a harmonic metric to decompose $\nabla=d_A+\Psi$. 

Using a complex structure on $\S$, we can further decompose $\Psi=\Phi+\Phi^{*_h}$ where $\Phi$ is the $(1,0)$-part of $\Psi$. The $(0,1)$-part of the unitary connection $d_A$ gives a holomorphic structure on $E$. Since $d_A\Psi=0$, we also have $d_A^{0,1}\Phi=0$. Hence $\Phi$ is a Higgs field. So to $(E, \nabla, h)$, we can associate $(E,d_A^{0,1},\Psi^{1,0})$ which is a Higgs bundle. The fact that it is semistable comes from the second half, the Hitchin--Simpson theorem.

\medskip
\paragraph{Harmonic map theory.}
To get an idea of the proof of the Corlette--Donaldson theorem, we give a brief introduction to harmonic maps.

The general setting is as follows: consider two Riemannian manifolds $(M,g)$ and $(N,G)$. To a smooth map $f:M\to N$, we associate the so-called \textbf{Dirichlet energy}
\begin{equation}\label{Dirichlet-energy}
E(f) =\int_M \left\| df^2\right\| d\mathrm{vol}(g) = \int_M \frac{\partial f^i}{\partial x^\alpha}\frac{\partial f^j}{\partial x^\beta} G_{ij}g^{\alpha\beta}\sqrt{\lvert \mathrm{det}(g)\rvert}dx
\end{equation}
where we use the Einstein sum convention, $x^\alpha$ are coordinates on $M$ and $(g^{\alpha\beta})$ is the inverse of the matrix $(g_{\alpha\beta})$.

A map $f$ is called \textbf{harmonic} if it is a critical point of $E$. This is the case iff $\Delta_{g,G}f=0$, where $\Delta_{g,G}$ is some generalization of the Laplacian, which explains the name ``harmonic''.

Some examples:
\begin{itemize}
	\item For $\dim M =1$, a harmonic map $f:M\to N$ is the same as a geodesic in $N$ parametrized by $M$.
	\item For $\dim N =1$, being harmonic is equivalent to $\Delta_M f=0$ where $\Delta_M$ is the Laplace--Beltrami operator.
	\item For $\dim M =2$, the energy only depends on a conformal class of $g$, i.e. the metrics $g$ and $e^\varphi g$ for $\varphi$ a function on $M$ gives the same. In dimension 2, a conformal class of a metric is the same as a complex structure, so we can do harmonic map theory with $M$ being a Riemann surface.
\end{itemize}

The main result of harmonic map theory is the following theorem due to Eells and Sampson \cite{eells-sampson}:
\begin{thm}[Eells--Sampson]
If $(M,g)$ and $(N,G)$ are compact Riemannian manifolds where $N$ has non-positive sectional curvature, then there is a unique harmonic map in each homotopy class of functions $[M,N]$.
\end{thm}
The main idea for the proof is to start with any function in a given class and to apply a heat flow (some kind of steepest descent flow). One has to show that the flow exists for a short time, then for all times and that when time goes to infinity, we get a well-defined limit which is harmonic.

Now, we are ready to see the proof idea of the Corlette--Donaldson theorem. It reduces nearly to the Eells--Sampson theorem, in an equivariant setting.

Remember that we wish to show that in the gauge-orbit of a flat connection (with completely reducible monodromy) there is a harmonic representative. So we have a flat bundle $(E,\nabla)$ with fixed hermitian metric $h$ and we vary $\nabla$ in its gauge-orbit. 

The first important idea is to notice that we can fix $\nabla$ and vary $h$ instead. To determine the action of a gauge transformation $g$ on $h$, just note that $\Psi$ is hermitian with respect to $h$, i.e. $\Psi^{*_h}=\Psi$. So $g.\Psi$ has to be hermitian with respect to $g.h$. Locally, we can write $\Psi^{*_h}=h\Psi^\dagger h^{-1}$. Hence
$$g\Psi g^{-1} = g.\Psi = (g.\Psi)^{*_{g.h}}=(g.h)(g\Psi g^{-1})^\dagger (g.h)^{-1}.$$
Using $\Psi=\Psi^{*h}=h\Psi^\dagger h^{-1}$, we deduce $$g.h=ghg^\dagger,$$
which is the usual action on hermitian structures.

We have now a flat bundle $(E,\nabla)$ with varying hermitian metric $h$. We have seen in Proposition \ref{prop:harm-metric} that a hermitian metric is a $\pi_1\S$-equivariant map $u:\widetilde{\S}\to G/K$.

The second important observation is that $h$ is harmonic iff $u$ is harmonic in the sense of harmonic map theory. Note that  $\widetilde{\S}$ is the hyperbolic plane and $G/K$ is a symmetric space, so both carry a natural Riemannian structure. One can check that changing $u$ by a homotopy is equivalent to changing $h$ by a gauge transform.

Since $K$ is the maximal compact subgroup of $G$, the symmetric space $G/K$ has non-positive sectional curvature. Thus, we are almost in the setting of the Eells--Sampson theorem. The only problem is that $\widetilde{\S}$ and $G/K$ are not compact. But our map $u$ is equivariant and the fundamental domain is $\S$ which is compact. 
In his paper \cite{corlette}, Corlette imitates the proof strategy of the Eells--Sampson theorem in the given setting. I recommend his paper to find more details.

Finally, note that the Dirichlet energy for a flat connection $\nabla = d_{A_h}+\Psi_h$ decomposed using a hermitian metric $h$, is given by
$$E_\nabla(h) = \int_\S \left\|\Psi_h\right\|^2_{L^2} d\mathrm{vol}_\S.$$

\medskip
\paragraph{Harmonic Higgs bundles and the Hitchin--Simpson theorem.}
We present the notion of a harmonic Higgs bundle, giving a link between Higgs bundles and flat connections. The existence of harmonic representatives, the Hitchin--Simpson theorem, completes the proof of the non-abelian Hodge correspondence.

To a stable Higgs bundle $(E,\Phi)$, we wish to associate a flat connection. The idea is to fix a hermitian structure $h$ on $E$. Since $E$ is holomorphic, we get the Chern connection $\nabla_A$. Then we consider
$$\mc{A} = \Phi+\nabla_A+\Phi^{*_h}.$$
The appearance of $\Phi+\Phi^{*_h}$ is not surprising since we wish to get an equivalence with harmonic flat bundles $(E,\nabla=d_A+\Psi,h)$. Since $\Phi=\Psi^{1,0}$, we have $\Psi=\Phi+\Phi^{*_h}$.

The strategy is to find a point in the gauge-orbit of $(E,\Phi)$, i.e. a point in the moduli space $\mathcal{M}_H$, such that $\mc{A}$ is flat. There is one important observation to be made: a $\C^*$-action on $\mathcal{M}_H$ simplifying the flatness condition.

The flatness of $\mathcal{A}$ is \textit{a priori} one complicated equation. The trick is to split this into five much simpler equations. The fact which allows this decomposition is a $\C^*$-action on $\mc{M}_H$ given by 
$$\lambda.[(E,\Phi)] = [(E,\lambda \Phi)],$$
i.e. we simply scale the Higgs field. This is well defined since the scaling commutes with the gauge action (where $\Phi$ simply gets conjugated) and one easily checks that $(E,\lambda \Phi)$ stays stable.

Using this action, we are looking actually at a whole family of connections
\begin{equation}\label{flat-lambda-connection}
\mc{A}(\lambda) = \lambda \Phi+\nabla_A+\lambda^{-1}\Phi^{*_h}.
\end{equation}
The reason why to consider $\Phi^{*_h}$ with weight $\lambda^{-1}$ comes from twistor theory, explained in Section \ref{HK}.

The curvature of $\mathcal{A}(\lambda)$, which is a Laurent polynomial in $\lambda$, is flat for all $\lambda$ iff all its coefficients are zero. In a local chart where $\mc{A}(\lambda) = d+\lambda \Phi+A+\lambda^{-1}\Phi^{*_h}$, we get
\begin{enumerate}
	\item $\Phi\wedge \Phi = 0$ and $\Phi^*\wedge \Phi^*=0$ (coefficients of $\lambda^2$ and $\lambda^{-2}$),
	\item $\delbar\Phi+[A^{0,1},\Phi]=0$ and $\del\Phi^*+[A^{1,0},\Phi^*]=0$ (coefficients of $\lambda$ and $\lambda^{-1}$),
	\item $F(A)+[\Phi,\Phi^*]=0$ (coefficient for constant term).
\end{enumerate}

Note that the couples of equations on the same line are equivalent (by taking the hermitian conjugate). Equation 1. is automatic since $\Phi$ is of type $(1,0)$ and we are on a surface\footnote{For the notion of a Higgs bundle on a higher-dimensional manifold, one requires $\Phi\wedge\Phi=0$ in the definition. We see here why.}. 
Equation 2. is also automatic since $\Phi$ is a Higgs field, so holomorphic (see Proposition \ref{Prop:higgs-bundle-cotangent-vector}).

The only remaining equation is the so-called \textbf{Hitchin equation}:
\begin{equation}\label{Hitchin-eq}
\boxed{F(A) + [\Phi,\Phi^*]=0.}
\end{equation}

\begin{Remark}
One can obtain this equation as a dimensional reduction of the Yang--Mills equation in dimension 4.
\end{Remark}

\begin{definition}
A \textbf{harmonic Higgs bundle} is a Higgs bundle $(E,\Phi)$ equipped with a hermitian metric $h$, such that Hitchin's equation \eqref{Hitchin-eq} holds.
\end{definition}

The main result is the theorem of Hitchin \cite{hitchin1987self} and Simpson \cite{simpson}:

\begin{thm}[Hitchin--Simpson]\label{hitchin-simpson-thm}
In the complex gauge orbit of a stable Higgs bundle $(E,\Phi)$ with $\deg(E)=0$, there is a unique (up to unitary gauge) harmonic representative iff $(E,\Phi)$ is polystable.
\end{thm}

Note that for $\Phi=0$, the Hitchin--Simpson theorem reduce to the Narasimhan--Seshadri theorem, since the Hitchin equation becomes $F(A)=0$, giving a flat unitary connection and $(E,0)$ is polystable iff $E$ is.

The proof is similar in spirit to the one of the Narasimhan--Seshadri theorem. If there is a harmonic representative, one shows polystability by a direct argument, see for example \cite[Section 3.2]{garcia-prada}.

For the converse, one defines a gradient flow using the functional 
$$f(\Phi,h) = \int_S \left\|F(A)+[\Phi,\Phi^{*_h}]\right\|^2_{L^2}.$$
Note that the integrand is nothing but some $L^2$-norm of the term from the Hitchin equation. One has to show that the flow stays inside the gauge-orbit, that a minimizing sequence $(A_n,\Phi_n)$ converges (for stable $(E,\Phi)$) and that the limit solves the Hitchin equation. Again details can be found in \cite[Section 3.2]{garcia-prada} and the original papers by Hitchin \cite{hitchin1987self} and Simpson \cite{simpson}.

We will see another proof sketch, similar to our proof sketch of the Narasimhan--Seshadri theorem, by interpreting the Hitchin equation as a moment map. To do so, we will introduce the notion of hyperkähler geometry and the hyperkähler quotient in Section \ref{HK}.

\medskip
\paragraph{Non-abelian Hodge correspondence.}
Now that we have the notions of harmonic representatives for both, flat bundles and Higgs bundles, the non-abelian Hodge correspondence reduces to a simple observation, the equivalence of harmonic flat bundles and harmonic Higgs bundles.

To a harmonic flat bundle $(E,\nabla=d_A+\Psi,h)$, we associate the Higgs bundle $(E,\delbar_E=d_A^{0,1}, \Phi=\Psi^{1,0})$. Together with the hermitian metric $h$, we actually get a harmonic Higgs bundle since the flatness $$F(A)+\Psi\wedge \Psi=0$$ is equivalent to Hitchin's equation
$$F(A)+[\Phi,\Phi^{*_h}]=0$$
since $\Psi=\Phi+\Phi^{*_h}$.

In the reverse direction, to a harmonic Higgs bundle $(E,\delbar_E, \Phi,h)$, we associate $(E,d_A+\Phi+\Phi^{*_h},h)$, where $d_A$ is the Chern connection. This is a harmonic flat bundle. Both constructions are inverse to each other.

This finishes the proof sketch of the non-abelian Hodge correspondence, which we summarized in Figure \ref{Fig:non-ab-proof}. 

\begin{figure}[h!]
\begin{center}
\includegraphics[height=7cm]{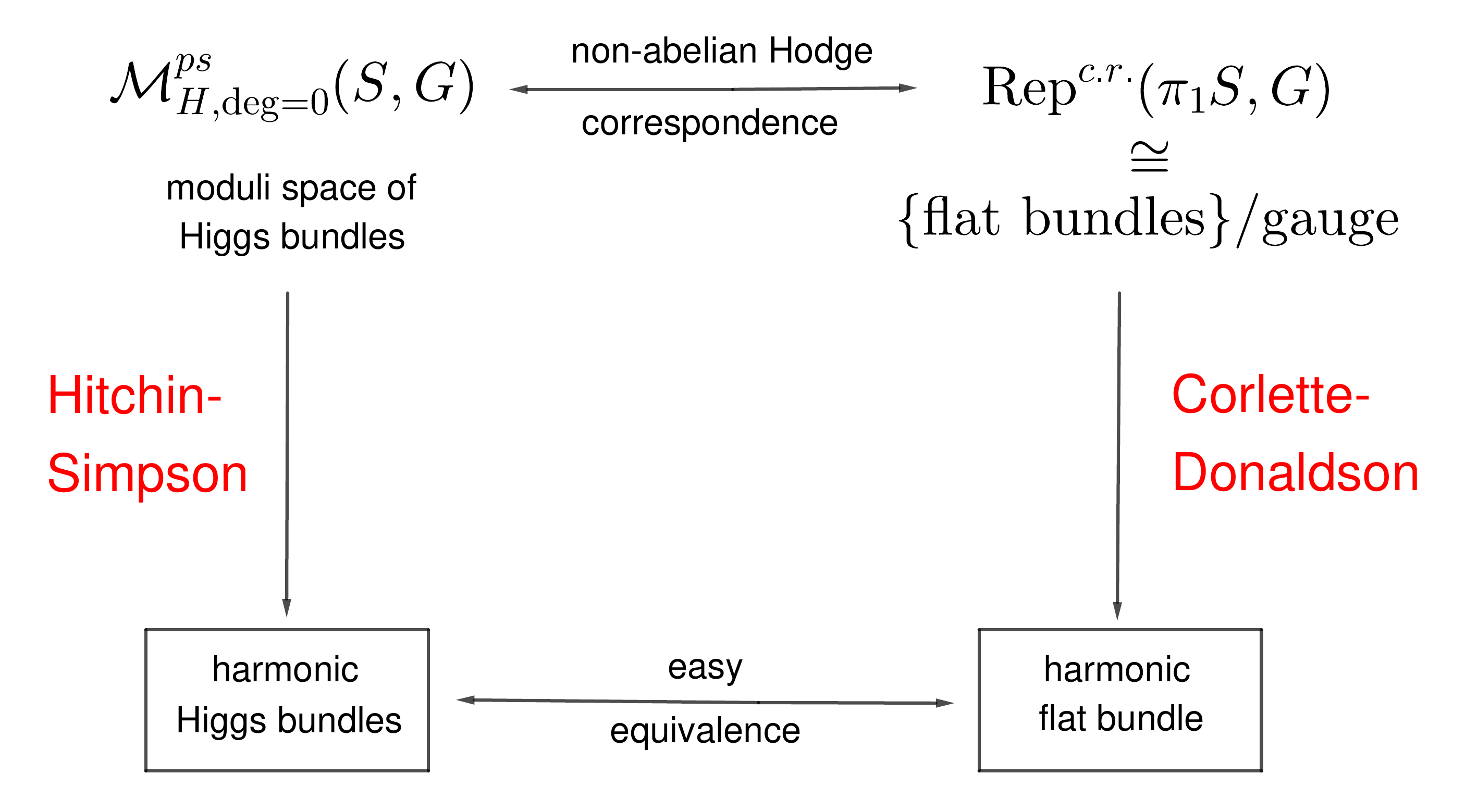}

\caption{Proof scheme of non-abelian Hodge correspondence via harmonic representatives}\label{Fig:non-ab-proof}
\end{center}
\end{figure}

Taking a step back, we can say that from a flat connection $\nabla=d_A+\Psi$ decomposed into a unitary and hermitian part, we can easily get a holomorphic structure $d_A^{0,1}$ and a Higgs field $\Psi^{1,0}$ (by forgetting half of the information). The difficulty lies in finding a preferred decomposition, which is achieved by using a harmonic hermitian structure.

The more surprising part of the non-abelian Hodge correspondence is that from half of the data, a holomorphic structure and a Higgs field, we can recover the flat connection by choosing an appropriate hermitian structure. The flatness condition reduces to solving Hitchin's equation.

To get a better and deeper understanding of the correspondence, we introduce hyperkähler geometry. The hyperkähler structure is the strongest possible geometric structure on a manifold (in a certain sense). This viewpoint will unify the character variety and the moduli space of Higgs bundles into one big picture.

\section{Hyperkähler geometry}\label{HK}

We give an overview on hyperkähler manifolds, in particular the quotient and twistor construction. This allows to understand the non-abelian Hodge correspondence as a natural diffeomorphism in the twistor space of the moduli space of Higgs bundles. A nice reference is Hitchin's paper \cite{hitchin1992hyper}.

\medskip
\paragraph{Kähler trilogy.}
A Kähler structure on a manifold is a Riemannian, symplectic and complex structure which interact nicely such that any two structures determine the third.

To start, let us see what happens at one point, i.e. we reduce to linear algebra. Consider $\R^{2n}$. We have the following correpondences between geometric structures and their symmetry groups:
\begin{center}
\begin{itemize}
	\item Riemannian structure $\longleftrightarrow$ $\mathrm{O}_{2n}(\R)$
	\item Symplectic structure $\longleftrightarrow$ $\mathrm{Sp}_{2n}(\R)$
	\item Complex structure $\longleftrightarrow$ $\mathrm{GL}_{n}(\C)$
\end{itemize}
\end{center}

\begin{prop}\label{kaehler-trilogy}
The intersection of any two of these three groups is $\mathrm{U}(n)$.
\end{prop}

\begin{figure}[h!]
\centering
\includegraphics[height=4cm]{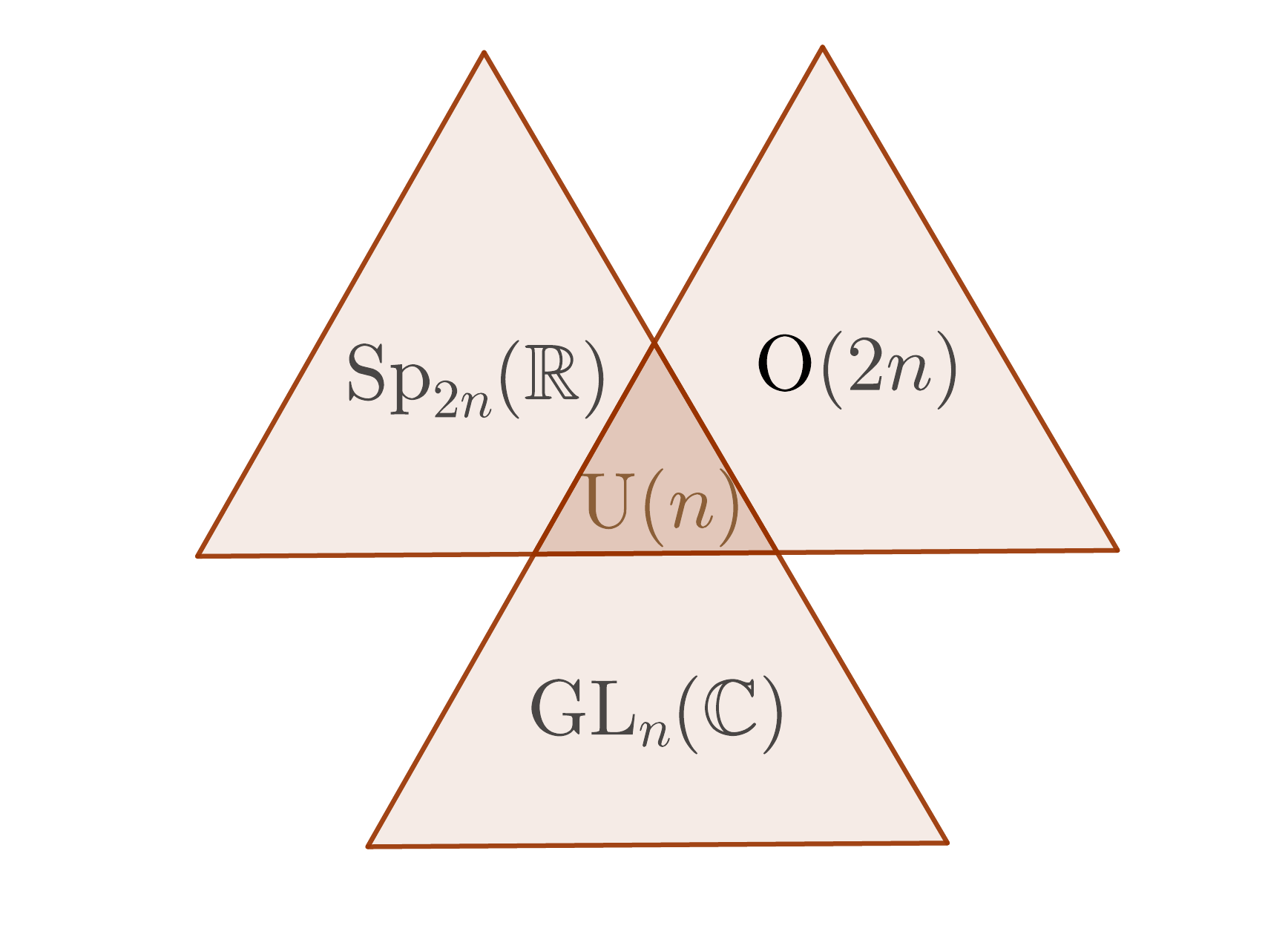}

\caption{Kähler trilogy for linear groups}
\end{figure}

The group $\mathrm{U}(n)$ corresponds to hermitian structures on $\R^{2n}\cong \C^n$. Our convention is that a hermitian product $h(X,Y)$ is $\C$-linear in $Y$ and anti-$\C$-linear in $X$.

\begin{prop}
If $h\in \mathrm{Herm}^{++}$, then $h=g+i\omega$ where $g$ is a Riemannian and $\omega$ a symplectic structure.
\end{prop}
This realizes the Kähler trilogy in one equation. In addition, we have
$$\omega(X,iY) = \mathrm{Im}(h(X,iY)) = \mathrm{Im}(ih(X,Y)) = \mathrm{Re}(h(X,Y)) = g(X,Y).$$

On a manifold, we gather these three structures together to define:
\begin{definition}
A \textbf{Kähler manifold} is $(M,g,\omega,I)$ where $g$ is a Riemannian, $\omega$ a symplectic and $I$ a complex structure such that $$g(X,Y) = \omega(X,IY) \;\;\forall\; X,Y\in \Gamma(TM).$$
\end{definition}

Note that in this definition, we consider the endomorphism $I\in \mathrm{End}(TM)$ satisfying $I^2=-\mathrm{id}$, which mimics the multiplication by $i$. Such a structure is called an \emph{almost-complex structure}. Any complex structure induces an almost-complex structure, but the converse is not true. An almost-complex structure which comes from a complex structure is called integrable.
The symplectic form $\omega$ is usually called the \emph{Kähler form}.

Using the fact that any two structures determine the third, we can equivalently define a Kähler manifold in the following ways:
\begin{itemize}
	\item \textit{Riemannian viewpoint:} Riemannian manifold with almost-complex structure $J$ which is orthogonal ($g(X,Y)=g(IX,IY)$) and with vanishing covariant derivative (using the Levi--Civita connection).
	\item \textit{Complex viewpoint:} Complex manifold with hermitian structure $h$ such that $\mathrm{Im}(h)$ is closed.
\end{itemize}

\begin{example}
Consider the simplest case $M=\C$ with hermitian metric $h=dz\otimes d\bar{z}$. Using real coordinates $z=x+iy$, we see that
$$h=(dx+idy)\otimes(dx-idy) = (dx^2+dy^2)-i(dx\wedge dy).$$
Hence the real part of $h$ is the standard Riemannian structure, and the negative imaginary part the standard symplectic structure on $\R^2$. 
\end{example}

Other examples:
\begin{itemize}
	\item $\C P^n$ is Kähler since the hamiltonian reduction of $\C^{n+1}$ by $\mathbb{S}^1$ gives not only a symplectic, but also a complex structure\footnote{There is the notion of Kähler reduction as for symplectic manifolds.}.
	\item Complex submanifolds of Kähler manifolds are again Kähler, so in particular any complex projective variety. In particular all Riemann surfaces are Kähler.
\end{itemize}

\begin{prop}
The holonomy of a Kähler manifolds (the monodromy of the Levi--Civita connection) is in $\mathrm{U}(n)$.
\end{prop}
This comes from the fact that the unitary group is the structure group of hermitian products.

\begin{Remark}
Two more important properties of Kähler manifolds from \cite[Chapter 0, Section 7]{griffiths}:
\begin{itemize}
	\item A metric is Kähler iff it is Euclidean up to order 2.
	\item The two possible Laplacians $\Delta_g$ (from the Riemannian structure) and $\Delta_\partial$ (from the complex structure) coincide up to a factor 2. This has many consequences, for example the \emph{Hodge identities} and the \emph{Lefschetz decomposition}.
\end{itemize}
\end{Remark}

\medskip
\paragraph{Hyperkähler trilogy.}
We get the same trilogy as for Kähler manifolds by shifting $\R$ to $\C$ and $\C$ to the quaternions $\H$. 

Consider the vector space $\C^{2n}$ which we can equip with the following geometric structures:
\begin{center}
\begin{itemize}
	\item Complex symplectic structure $\longleftrightarrow$ $\mathrm{Sp}_{2n}(\C)$
	\item Quaternionic structure $\longleftrightarrow$ $\mathrm{GL}_{n}(\H)$
	\item Hermitian structure $\longleftrightarrow$ $\mathrm{U}_{2n}(\C)$
\end{itemize}
\end{center}

\noindent The starting point for hyperkähler geometry is the analog of Proposition \ref{kaehler-trilogy}:
\begin{prop}\label{HK-trilogy}
The intersection of any two of these three groups is $\mathrm{U}_n(\H)$.
\end{prop}

\begin{figure}[h!]
\centering
\includegraphics[height=4cm]{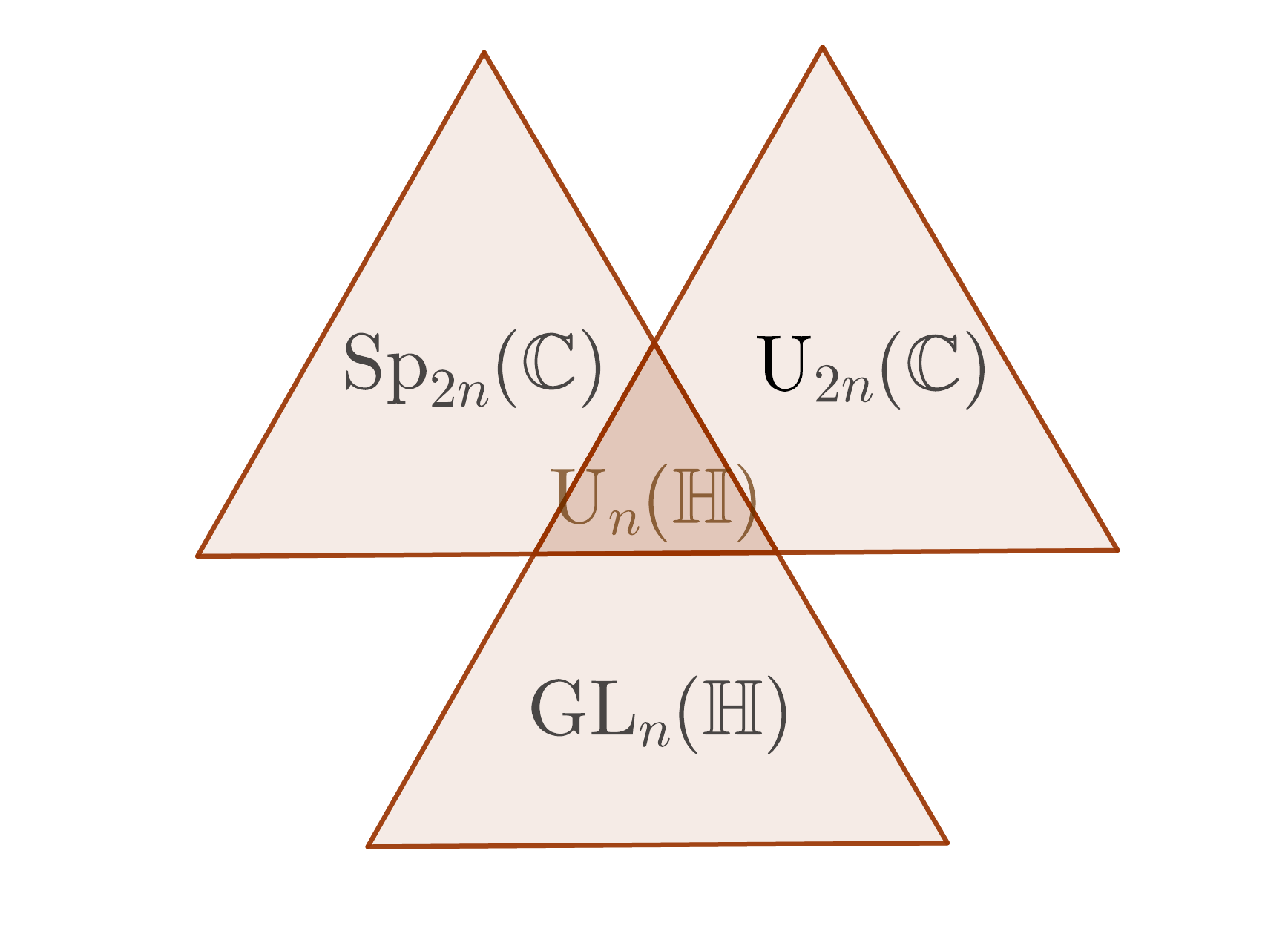}

\caption{Hyperkähler trilogy for linear groups}
\end{figure}

The group $\mathrm{U}_n(\H)$ corresponds to quaternionic scalar products. We will see below that a quaternionic scalar product $Q$ can be written
$$Q=h-\omega_\C J,$$ where $h$ is a hermitian, $\omega_\C$ a complex symplectic structure (both with respect to $I$) and $J$ a complex structure (from the quaternionic structure). This equation illustrates the hyperkähler trilogy.

As for Kähler manifolds, we could give three equivalent definitions for hyperkähler manifolds, but we restrict to the Riemannian viewpoint:
\begin{definition}
A \textbf{hyperkähler manifold}, \emph{HK manifold} for short, is a Riemannian manifold $(M,g)$ with three orthogonal covariant constant automorphisms $I,J,K\in \End(TM)$ satisfying the quaternionic relations $I^2=J^2=K^2=IJK=-\id$.
\end{definition}

We see that a hyperkähler manifold has several Kähler structures which together equip the tangent space with a quaternionic structure. From the definition we get three Kähler structures $(g,I), (g,J)$ and $(g,K)$, but there are much more:
a linear combination $\alpha I+\beta J+\gamma K$ is a complex structure iff $(\alpha I+\beta J+\gamma K)^2=-\id$ which is equivalent to $\alpha^2+\beta^2+\gamma^2=1$, which defines a sphere.

Therefore, a hyperkähler manifold has a 1-parameter family of Kähler structures, indexed by $\C P^1$.

\begin{exo}
Formulate the equivalent ways to define a hyperkähler manifold from the trilogy of groups.
\end{exo}

We mentioned earlier that hyperkähler structures are the ``strongest'' geometric structures on manifolds. This is true in the following sense: the \emph{Berger classification} gives a complete list of all possible holonomies of Riemannian manifolds. Recall that the holonomy is the monodromy of the Levi-Civita connection in the tangent bundle. In general, the holonomy group is $\mathrm{O}(n)$. Any reduction of this structure group corresponds to some geometric structure:
\begin{itemize}
	\item $\mathrm{SO}(n)$ corresponds to an orientation,
	\item $\mathrm{U}(n)$ corresponds to a Kähler structure.
\end{itemize}
The smallest of all groups in Berger's list is $\mathrm{U}_n(\H)$ which corresponds, as you can guess, to hyperkähler manifolds.

\begin{example}
The simplest example is the linear case $M=\H$. The three complex structures given by $i, j$ and $k$ can be seen by the different identifications between $\H$ and $\C^2$:
\begin{align*}
q=x_0+ix_1+jx_2+kx_3 &= (x_0+ix_1)+(x_2+ix_3)j\\
&= (x_0+jx_2)+(x_3+jx_1)k\\
&= (x_0+kx_3)+(x_1+kx_2)i.
\end{align*}
The corresponding symplectic structures $\omega_I, \omega_J$ and $\omega_K$ can be computed from these expressions. For example we get
$$\omega_I = dx_0\wedge dx_1+dx_2\wedge dx_3.$$
We can consider the quaternionic scalar product $Q=dq\otimes d\bar{q}$ where $\bar{q}=x_0-ix_1-jx_2-kx_3$ is the conjugate. One can compute that
$$Q=dq\otimes d\bar{q}=g-i\omega_I-j\omega_J-k\omega_K = h-\omega_\C j$$
where $h=g-i\omega_I$ is hermitian (with respect to $i$) and $\omega_\C=\omega_J+i\omega_K$ is a complex symplectic form.
This illustrates the hyperkähler trilogy from Proposition \ref{HK-trilogy}.
\end{example}

From this example, we can extract some general facts: the form $\omega_\C=\omega_J+i\omega_K$ is a holomorphic symplectic structure (with respect to $I$). In this contexte, we write $\omega_\R$ for $\omega_I$. On a HK manifold, we have a quaternionic scalar product $Q$ satisfying $$Q=h-\omega_\C J$$ where $h$ is the hermitian structure associated to $I$.

Since all the complex structures are on the same footing, you can change $I$ to any other and redefine $\omega_\C$ and $h$. The only structure which will not move is $Q$. A change of the basic complex structure is called a \emph{hyperkähler rotation} (having the sphere of complex structures in mind).

Other examples of HK manifolds are $\H^n$ and $T^*\C P^n$ (see Example \ref{TCPn} below). Hyperkähler manifolds are much more rigid and rare than Kähler manifolds: in particular \emph{no submanifold of $\H P^n$ (including the whole quaternionic projective space) is HK}.

There are two general methods to construct HK manifolds: the hyperkähler quotient and the twistor space construction.

\medskip
\paragraph{Hyperkähler quotient construction.}
The hyperkähler quotient is very much modeled on the symplectic quotient, i.e. the hamiltonian reduction.

Consider a group $G$ with an action on a HK manifold $(M,\omega_I,\omega_J,\omega_K)$, Hamiltonian with respect to all three symplectic forms. From Section \ref{hamiltonian-reduction}, we get three moment maps $\mu_1,\mu_2$ and $\mu_3$ which we can put together into a vector-valued moment map 
$$\mu: M\to \g^*\otimes \R^3.$$

\begin{thm}
For coadjoint orbits $\mc{O}_1,\mc{O}_2,\mc{O}_3$ which represent a regular value of $\mu$, the quotient
$$M/\!/\!/G := \mu^{-1}(\mc{O}_1,\mc{O}_2,\mc{O}_3)/G$$
is a hyperkähler manifold.
\end{thm}

Concentrating on one complex structure $I$, we can split up $\mu$ into $\mu_\R=\mu_1$ and $\mu_\C=\mu_2+i\mu_3$:
$$\mu=(\mu_\R,\mu_\C):M\to \g^*\oplus \g^*\otimes \C.$$
The map $\mu_\C$ is in fact holomorphic and corresponds to the moment map of the action of the complexified group $G^\C$ on $(M,\omega_\C)$. The HK quotient can then be computed in two steps:
\begin{enumerate}
	\item Compute $\mu_C^{-1}(\mc{O}_\C)$ which is Kähler (complex submanifold of a Kähler manifold),
	\item Perform the Hamiltonian reduction with respect to the $G$-action: 
	$$\mu_\C^{-1}(\mc{O}_\C)\sslash G = \left(\mu_\R^{-1}(\mc{O}_1)\cap \mu_\C^{-1}(\mc{O}_\C)\right)/G=M/\!/\!/G.$$
\end{enumerate}

By the Kempf--Ness theorem (which works also in the Kähler setting), we get
\begin{equation}\label{HK-quotient-equality}
\boxed{M/\!/\!/G = \mu_\C^{-1}(\mc{O}_\C)\sslash G = \mu_\C^{-1}(\mc{O}_\C)/_{GIT}\, G^\C = M\sslash G^\C}
\end{equation}
where the last quotient is the holomorphic hamiltonian reduction. All these equalities are very helpful to compute HK quotients as we shall see below.

\begin{example}\label{TCPn}
Consider the action of $G=\mathbb{S}^1$ on $\C^n\oplus (\C^n)^*=T^*(\C^n)$ given by $\lambda.(z,\xi) = (\lambda z,\lambda^{-1}\xi)$. One checks, using for example an identification $\H^n\cong T^*(\C^n)$, that
$$\omega_\R = \frac{i}{2}(dz\wedge d\bar{z}+d\xi\wedge d\bar{\xi}) \text{ and } \omega_\C = dz\wedge d\xi,$$
where we use the short-hand notation $dz\wedge d\bar{z}$ for $\sum_i dz_i\wedge d\bar{z}_i$ and similar for $\xi$.

To $\theta\in\R=\mathrm{Lie}(\mathbb{S}^1)$, the associated vector field is $X_\theta(z,\xi) = \binom{i\theta z}{-i\theta\xi}$. Hence
\begin{align*}
i_{X_\theta}\omega_1(\delta z, \delta \xi) &= \frac{i}{2}\left((i\theta z \delta\bar{z}+i\theta\bar{z}\delta z)+(-i\theta\xi\delta\bar{\xi}-i\theta\bar{\xi}\delta\xi)\right) \\
&=\delta\left(-\frac{\theta}{2}(\lvert z \rvert^2-\lvert \xi\rvert^2)\right)
\end{align*}
So $\mu_\R(z,\xi)=-\frac{1}{2}\left(\lvert z \rvert^2-\lvert \xi\rvert^2)\right)$. Further, we have
$$i_{X_\theta}\omega_\C(\delta z,\delta \xi) = i\theta z \delta \xi+i\theta \xi \delta z = \delta\left(i\theta z\xi\right),$$
so $\mu_\C(z,\xi) = iz\xi=i\xi(z)$.

Finally, the reduction over $(-1/2,0)$ gives
$$T^*\C^n/\!/\!/\mathbb{S}^1 = \{(z,\xi)\mid \xi(z)=0, \lvert z \rvert^2-\lvert \xi\rvert^2=1\}/\mathbb{S}^1 \cong T^*(\C P^{n-1}).$$
The last identification needs some reflection. For $\xi=0$ we simply get $\C P^{n-1}$. For $\xi\neq 0$, the condition $\xi(z)=0$ shows that $\xi$ is a cotangent vector to $z$ considered as a point in the projective space.

Without any precise computation, we could also compute $$T^*\C^n/\!/\!/\mathbb{S}^1 = T^*\C^n \sslash \C^* = T^*(\C^n/_{GIT}\, \C^*) = T^*(\C P^{n-1}),$$
where we used Equation \eqref{HK-quotient-equality} and the equality $T^*(X/_{GIT}\, G)=T^*X\sslash G$ (see Exercise \ref{cotang-group-quotient}).

One Kähler structure of $T^*(\C P^{n-1})$ comes from $\mu_\C^{-1}(0)\sslash\mathbb{S}^1$. To get all the others, we have to consider $\mu_\C^{-1}(\alpha)\sslash\mathbb{S}^1$ with $\alpha\in \C^*$. These hamiltonian reductions are affine bundles over $\C P^{n-1}$ with associated vector bundle $T^*(\C P^{n-1})$.
\end{example}

\medskip
\paragraph{Twistor construction.}
Another way to construct HK manifolds, and to get a pictorial approach to them, is the twistor space construction. The basic idea is to gather all Kähler structures together in one slightly larger space.

We have seen that on a HK manifold $M$, there is a 1-parameter family of Kähler structures indexed by $\C P^1$. Consider $$Z=M\times \C P^1,$$
equipped with the almost-complex structure $I(m,\lambda)=I_\lambda(m)\oplus I_0$, where $I_0$ is the (unique) standard complex structure on $\C P^1$ and $I_\lambda$ the complex structure on $M$ associated to $\lambda\in \C P^1$.

It turns out that $(Z,I)$, the so-called \textbf{twistor space}, is a complex manifold (i.e. that $I$ is integrable). Further the map to the second factor $Z\to \C P^1$ is holomorphic. Note that as complex manifold, $Z$ is not a direct product between $M$ and $\C P^1$ (only as smooth manifolds). 

We draw the following picture for the twistor space $Z$: it is a $M$-bundle over $\C P^1$ where all fibers are diffeomorphic (to $M$), but not biholomorphic (since the complex structure may change). Diametrical opposite points correspond to conjugated complex structures ($I$ becomes $-I=\bar{I}$).

\begin{center}
\begin{tikzpicture}[scale=1.2]
	\draw (0,0) circle (1cm);
\draw [domain=180:360] plot ({cos(\x)},{sin(\x)/3});
	\draw [domain=0:180, dotted] plot ({cos(\x)},{sin(\x)/3});
	\draw [fill=white] (0,1) circle (0.04);
	\draw [fill=white] (0,-1) circle (0.04);
	\draw [fill=white] (1,0) circle (0.04);
	\draw (0,-1.25) node {$(M,I)$};  
	\draw (0,1.2) node {$(M,-I)$};
	\draw (1.55,0) node {$(M,J)$};
\end{tikzpicture}
\end{center}

On $Z$, there is a real structure, i.e. an involution given by $$r(m,\lambda)=(m,-1/\bar{\lambda}).$$ Note that $-1/\bar{\lambda}$ is the diametral opposite point to $\lambda\in \C P^1$.

We wish to recover the HK manifold $M$ from $Z$. Any fiber of $p:Z\to \C P^1$ is diffeomorphic to $M$, but this does not recover the hyperkähler structure. For this, we need the information of all fibers. The idea is simply that $M$ is embedded into the space of sections via the constant section $(m,\lambda)$ with $m\in M$ fixed and $\lambda \in \C P^1$ varying.

The important observation is the following:
\begin{prop}
The holomorphic sections of $Z\to \C P^1$ which are invariant under the real structure $r$ are the constant ones.
\end{prop}

This allows us to recover $M$ from the twistor space. The real and holomorphic sections are called \textbf{twistor lines}. 

We can revert the whole procedure to construct HK manifolds from twistor spaces. This goes as follows (see \cite[Theorem 3.3]{HKLR}):
\begin{thm}
Let $Z$ be a complex $(2n+1)$-dimensional manifold with a holomorphic map $\pi: Z\to \C P^1$ such that
\begin{enumerate}
	\item $\pi$ allows a family of holomorphic sections with normal bundle $\mc{O}(1)\otimes \C^{2n}$,
	\item there is a holomorphic symplectic structure in each fiber, depending quadratically\footnote{Abstractly, it is a section of $\Lambda^2T^*F\otimes \mc{O}(2)$ where $F$ denotes the fiber.} on $\lambda\in \C P^1$,
	\item there is a compatible real structure on $Z$ inducing the antipodal map on $\C P^1$.
\end{enumerate}
Then the space of real and holomorphic sections is a hyperkähler manifold with twistor space $Z$.
\end{thm}

Some remarks on the conditions:
\begin{itemize}
	\item The requirement to have normal bundle $\mc{O}(1)\otimes \C^{2n}$ roughly means that any 2 points (not in the same fiber) define a unique holomorphic section, since sections of $\mc{O}(1)$ are just affine functions. Hence, any point $P\in Z$ defines a unique real holomorphic section, passing through $P$ and $r(P)$.
	\item The holomorphic symplectic form can be written $\omega(\lambda) = \lambda\omega_\C+ i\omega_\R+\lambda^{-1}\bar{\omega}_\C$, which satisfies $\overline{\omega(-1/\bar\lambda)}=\omega(\lambda)$, the compatibility condition between the real structure and $\omega$.
\end{itemize}

\begin{example}
Let us analyze the twistor space of $M=\H$. 

We claim that $Z$ is the total space of the bundle $\mc{O}(1)\oplus\mc{O}(1)$ over $\C P^1$. The fiber has complex dimension 2 and condition (1) above is trivially satisfied. 

A section is of the form $s(\lambda)=(a\lambda+b,c\lambda+d)$. Using these coordinates, the real structure is given by $r(\lambda,a,b,c,d)=(-1/\bar\lambda,\bar{d},-\bar{c},-\bar{b},\bar{a})$. You can check that $r^2=\id$.
Another way to understand $r$ is to compose the antipodal map $\lambda\mapsto -1/\bar\lambda$ with the anti-involution $(z,w)\mapsto (\bar{w},-\bar{z})$ where $z, w$ are the coordinates of the fibers.
Indeed, we get 
$$r(s) = r(a\lambda+b,c\lambda+d)=(-\bar{c}\lambda^{-1}+\bar{d},\bar{a}\lambda^{-1}-\bar{b})=(\bar{d}\lambda-\bar{c},-\bar{b}\lambda+\bar{a})$$
where we used the transition function $\lambda\mapsto \lambda^{-1}$ between the two charts of $\C P^1$.

Hence, a section is $r$-invariant iff $a=\bar{d}$ and $b=-\bar{c}$. So the twistor lines are given by
$$s(\lambda) = (a\lambda+b,-\bar{b}\lambda+\bar{a}).$$
This corresponds to the matrix representation of quaternions, where a quaternion $x_0+ix_1+jx_2+kx_3$ with $a=x_0+ix_1$ and $b=x_2+ix_3$ is represented by 
$$q\mapsto \begin{pmatrix}a & -\bar{b}\\ b& \bar{a}\end{pmatrix}.$$
\end{example}

Two other examples without details, presented only through their twistor space picture:
\begin{itemize}
	\item The twistor space of $T^*(\C P^n)$: all fibers, apart from $\pm i$, are biholomorphic to an affine bundle over $\C P^n$ with associated vector bundle $T^*(\C P^n)$.
	\item Coadjoint orbits of complex simple Lie groups are hyperkähler. The maximal coadjoint orbit can be identified with $G/H$ where $H$ is the Cartan group of $G$. Again, all fibers are biholomorphic, apart from the fibers over $\pm i$ where we see the cotangent bundle to the flag variety $T^*(G^c/B)$ (where $G^c$ denotes the compact form of $G$ and $B$ a Borel subgroup of $G^c$).
\end{itemize}

A source for HK structures with a twistor space where all fibers, apart over two opposite points, are biholomorphic is the following (see \cite{feix} and \cite{kaledin}):
\begin{thm}[Feix--Kaledin]\label{feix-kaledin-thm}
If $M$ is a Kähler manifold, then there is a neighborhood of the zero-section $M\subset T^*M$ which has a $\C^*$-invariant hyperkähler structure, where the $\C^*$-action on $T^*M$ is given by $\lambda(m,v)=(m,\lambda v)$.
\end{thm}

\medskip
\paragraph{Character varieties and Higgs bundles again.}
We will show that the moduli space of Higgs bundles $\mc{M}_H$ is hyperkähler by a quotient construction and describe its twistor space. 

Consider $\mc{A}$, the space of all $G$-connections ($G=\GL_n(\C)$) on a trivial complex bundle $E$ over a Riemann surface $S$ with fixed hermitian structure. The space $\mc{A}$ has a holomorphic symplectic form given by
$$\omega_\C = \int_S \tr \delta A \wedge \delta A.$$
It has also a Riemannian metric given by $$\left\|A\right\| = \int_S \tr (A^{1,0}\wedge A^{1,0 \,*})-\tr (A^{0,1}\wedge A^{0,1 \,*})$$
where we used the complex structure on $S$ to decompose $A$.
Hence $\mc{A}$ is an infinite-dimensional (flat) hyperkähler manifold.

Consider further the space of unitary gauge transformations $\mc{G}$ which acts on $\mc{A}$, preserving the HK structure. From \cite[Section 6.3]{hitchin1992hyper}, we have:
\begin{prop}
The moment maps are given by $\mu_\C(A)=F(A)$ the curvature (Atiyah--Bott) and $\mu_\R(A)=F'-F''$ where $F'$ and $F''$ are the curvatures of the unique unitary connections $\nabla'$ and $\nabla''$ with $(\nabla')^{1,0}=A^{1,0}$, respectively $(\nabla'')^{0,1}=A^{0,1}$.

Further, we have $F'=F''$ iff the metric $h$ is harmonic.
\end{prop}
In the last statement, instead of varying $A$ in its gauge orbit, we vary the hermitian structure $h$, a trick we have already seen for proving the Corlette--Donaldson theorem.

Therefore, the HK quotient over the zero coadjoint orbit is\footnote{We neglect all difficulties due to the fact that $\mc{A}$ is infinite-dimensional.}
$$\mc{A}/\!/\!/\mc{G}=\{\text{flat harmonic bundles}\}/\mc{G} = \Rep(\pi_1\S,G)$$
where we used the Corlette--Donaldson theorem \ref{corlette-donaldson-thm}, stating that in each $\mc{G}^\C$-orbit of a flat connection, there is a unique (up to $\mc{G}$-gauge) harmonic representative.
We can also use the multiple ways to compute the HK quotient (see Equation \eqref{HK-quotient-equality}): 
$$\mc{A}/\!/\!/\mc{G} = \mu_\C^{-1}(0)/\mc{G}^\C = \Rep(\pi_1\S,G)$$
where we used the Atiyah--Bott theorem \ref{atiyah-bott-thm}.

Anyway, we see that \emph{the character variety $\Rep(\pi_1\S,\GL_n(\C))$ is a hyperkähler manifold!}

Still another way to do the quotient, which is not completely rigorous, is the following. Start from $$\mc{A}/\!/\!/\mc{G}=\mc{A}\sslash\mc{G}^\C\;\;\; \text{(Equation \eqref{HK-quotient-equality})}.$$ Then notice that $\mc{A}=\Omega^1(S,\g) = T^*\Omega^{0,1}(S,\g)$ (see Equation \eqref{cotangent01}). Modulo some stability conditions we then get
$$\mc{A}/\!/\!/\mc{G}=T^*\Omega^{0,1}(S,\g)\sslash\mc{G}^\C \approx T^*\left(\Omega^{0,1}(S,\g)/\mc{G}\right) = T^*\mathrm{Hol}(S,G).$$
Note that the \emph{moment map of the $\mc{G}^\C$-action on $T^*\Omega^{0,1}(S,\g)$ is nothing but the Hitchin equation!}

Finally, we know that the moduli space of Higgs bundles $\mc{M}_H$ is an open dense subset of $T^*\mathrm{Hol}(S,G)$ (see Equation \eqref{higgs-moduli-cotangent-bundle}). It turns out that the stability conditions we have neglected precisely describe $\mc{M}_H(S,G)$:
$$\mc{A}/\!/\!/\mc{G}=\mc{M}_H(S,G).$$
In particular, \emph{the moduli space of Higgs bundles is a hyperkähler manifold, the same as the character variety!}

The twistor picture of this hyperkähler manifold looks like this:
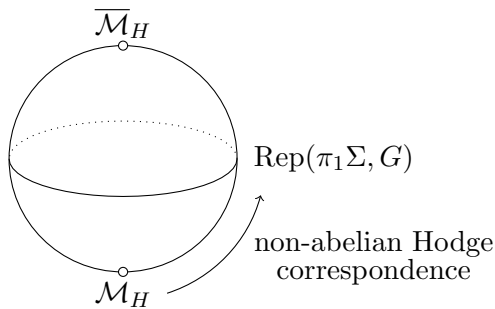
\begin{figure}
\begin{center}
\begin{tikzpicture}[scale=1.5]
	\draw (0,0) circle (1cm);
\draw [domain=180:360] plot ({cos(\x)},{sin(\x)/3});
	\draw [domain=0:180, dotted] plot ({cos(\x)},{sin(\x)/3});
	\draw [fill=white] (0,1) circle (0.04);
	\draw [fill=white] (0,-1) circle (0.04);
	\draw (0,-1.2) node {$\mc{M}_H$};  
	\draw (0,1.2) node {$\overline{\mc{M}}_H$};
	\draw (1.85,0) node {$\Rep(\pi_1\S, G)$};
	\draw[->] ({1.25*cos(288)},{1.25*sin(288)}) arc (288:345:1.25);
	\draw (2.2,-0.75) node {non-abelian Hodge};
	\draw (2.2,-1) node {correspondence};
\end{tikzpicture}

\caption{Twistor space of moduli space of Higgs bundles}\label{twistor-space-higgs}
\end{center}
\end{figure}

All the fibers are diffeomorphic, but only the fibers for $\l\in \C^*$ are also biholomorphic (to the character variety). Over the points $\lambda=0$ and $\lambda=\infty$, we see the moduli space of Higgs bundles $\mc{M}_H(S,G)$ and its conjugate. A twistor line is given by a quadratic expression $$\mc{A}(\lambda)=\lambda\Phi+\nabla_A+\lambda^{-1}\Phi^{*_h}$$ which is the form we considered in Equation \eqref{flat-lambda-connection} to prove Hitchin--Simpson theorem. The reality constraint is given by $$-\mathcal{A}(-1/\bar{\lambda})^{*_h} = \mc{A}(\lambda),$$
which explains the quadratic depends in $\lambda$ of $\mc{A}$ and the appearance of the term $\lambda^{-1}\Phi^{*_h}$.

\begin{thm}
The moduli space of Higgs bundles and the character variety are two incarnation of the same hyperkähler manifold. The twistor lines are connections of the form $\mc{A}(\lambda)=\lambda\Phi+\nabla_A+\lambda^{-1}\Phi^{*_h}$.
\end{thm}

Note that the complex structure on $\mc{M}_H(S,G)$ comes from the complex structure of the surface, while the complex structure on $\Rep(\pi_1\S,G)$ comes from the complex Lie group $G$.

The twistor space approach explains several phenomena in a concise way:
\begin{itemize}
	\item The need for a real structure leads to the consideration of a hermitian structure $h$ on the bundle.
	\item The fact that a twistor line is determined by a point $P$ (here $\Phi$) and its conjugate $r(P)$ (here $\Phi^{*_h}$) partially explains the form of $\mc{A}(\lambda)$.
	\item All fibers are canonically diffeomorphic via the twistor lines, which gives the diffeomorphism between $\lambda=0$ and $\lambda=1$: the non-abelian Hodge correspondence!
	\item The $\C^*$-action on $\mc{M}_H(S,G)$ explains why all fibers over $\l\in \C^*$ are the same Kähler manifold. Indeed applying the non-abelian Hodge correspondence to $\ell \Phi$ instead of $\Phi$ is equivalent to choosing $\lambda=\ell$ instead of $\lambda=1$.
\end{itemize}

\section{Application: Hitchin components}\label{hit-comp}

We have analyzed character varieties $\Rep(\pi_1\S,G)$ for unitary groups and complex groups. \emph{What about other real forms of $\SL_n(\C)$, in particular the split real form $\SL_n(\R)$?}

The question turns out to be quite difficult. For a complex simple Lie group, the character variety is connected, but this is not true any longer for real groups, where many components can appear. Already the count of these component is highly non-trivial.

The main tool we have to analyze any kind of character variety is the non-abelian Hodge correspondence, since $G\subset G^\C$ allows to go to the complex group. The question above becomes: \emph{What kind of Higgs bundle correspond to $\SL_n(\R)$?}
We have seen that the unitary group corresponds simply to $\Phi=0$ (vanishing Higgs field), since the non-abelian Hodge correspondence reduces to the Narasimhan--Seshadri theorem \ref{ns-thm} in the unitary case.

\medskip
\paragraph{Motivation.}
The main motivation for studying character varieties for split real forms is the link to geometric structures. For $G=\SL_2(\R)$ (or more precisely for $\PSL_2(\R)=\SL_2(\R)/\pm \id$) and $\S$ a surface of genus at least 2, there is a connected component of the character variety, called the \textbf{Teichmüller space}, describing several geometric structures of the surface:
\begin{align*}
\mathrm{Teich}(\S) &= \{\text{complex structures}\}/\mathrm{Diff}_0(\S) \\
&= \{\text{hyperbolic structures}\}/\mathrm{Diff}_0(\S) \\
&= \text{connected component of } \Rep(\pi_1\S,\PSL_2(\R)),
\end{align*}
where $\mathrm{Diff}_0(\S)$ denotes the identity component of the diffeomorphism group of $\S$. A \emph{hyperbolic structure} is a Riemannian metric with constant curvature equal to -1.

The link between the character variety, hyperbolic and complex structures goes as follows: equip a surface $\S$ with a complex structure, so it becomes a Riemann surface $S$. Then its universal cover $\widetilde{\S}$ also gets a complex structures (since a complex structure is a local property). The famous \emph{Poincaré uniformisation theorem} asserts that any simply connected Riemann surface is either $\C P^1, \C$ or the hyperbolic plane $\H^2$. For $\S$ of genus at least 2, the universal cover has the topological type of $\H^2$, so by the uniformisation theorem it is biholomorphic to $\H^2$.

The fundamental group $\pi_1\S$ acts on $\widetilde{\S}$ by deck transformations which are isometries. Since the isometry group of $\H^2$ is $\PSL_2(\R)$, we get a representation $\rho: \pi_1\S\to \PSL_2(\R)$. It is this representation which allows to recover $S$ from $\H^2$ since $S=\H^2/\rho(\pi_1\S)$. It turns out that two complex structures obtained by this quotient are equivalent under $\mathrm{Diff}_0(\S)$ iff the representations are conjugated. Therefore we get an inclusion of $\mathrm{Teich}(\S)$ into the character variety. It is then easy to check that it is a connected component. Even better: it is \emph{the} connected component of discrete and faithful representations (since the quotient of $\H^2$ by $\rho$ is a manifold).

\medskip
\paragraph{Hitchin components.}
In the seminal paper \cite{hitchin}, Nigel Hitchin constructs a connected component in $\Rep(\pi_1\S,\PSL_n(\R))$ (in fact more generally for split real groups $G$) with similar properties to Teichmüller space. In particular, all representations of the so-called Hitchin component are discrete and faithful.

Consider a Riemann surface $S$ and fix a spin structure $K^{1/2}$ (a line bundle with square the canonical bundle $K=T^*S$). We have seen in Example \ref{higgs-basic-ex} that $$\left(V=K^{1/2}\oplus K^{-1/2}, \Phi=\begin{pmatrix}0 & 0 \\ 1 & 0\end{pmatrix}\right)$$ is a stable Higgs bundle.

To get a bundle of rank $n$, take the symmetric product $E=\mathrm{Sym}^n(V)=K^{(n-1)/2}\oplus K^{(n-3)/2}\oplus \cdots \oplus K^{-(n-1)/2}$. For the Higgs field, we choose a matrix with identical entries along the parallel lines to the main diagonal:
\begin{equation}\label{Hitchin-section}
\Phi = \begin{pmatrix}0 & t_2&t_3& \cdots& t_n\\ 1& 0&t_2&\cdots & t_{n-1}\\ 0&1&0&\ddots &\vdots\\ \vdots &\ddots &\ddots &\ddots & t_2\\ 0&\cdots& 0 &1&0\end{pmatrix}.
\end{equation}

The entry $t_i$ is a holomorphic section of $\Hom(K^{n-(2i-1)/2},K^{n-(2i-3)/2}\otimes K)\cong K^i$. To show stability, notice that for $t_i=0$ for all $i$, we get a stable Higgs bundle by the same argument as for the case $n=2$. Then, there is a diagonal gauge transformation which transforms $t_k$ to $\lambda^{k-1}t_k$ for $\lambda\in \C^*$. Hence we can get arbitrarily close to $t_k=0 \;\forall\, k$. Since stability is an open condition, we see that our $(E,\Phi)$ is stable.

\begin{thm}[Hitchin]
The flat connections associated to $(E,\Phi)$ in Equation \eqref{Hitchin-section} via the non-abelian Hodge correspondence describe a connected component of $\Rep(\pi_1\S,\PSL_n(\R))$.
\end{thm}

Since $t_i\in H^0(S,K^i)$, we get a parametrization of the \textbf{Hitchin component} $\mathrm{Hit}(n,\S)$ by $\bigoplus_{i=2}^n H^0(S,K^i)$. For $n=2$, the Hitchin component coincides with Teichmüller space. Notice that all these components are contractible.

\begin{Remark}
The Hitchin component contains all the representations of the form 
\begin{equation}\label{rep-n-hitchin}
\pi_1\S \xrightarrow[\text{fuchsian}]{} \PSL_2(\R)\xrightarrow[\text{principal}]{}\PSL_n(\R)
\end{equation}
where the first map is a discrete and faithful representation (a point in Teichmüller space) and the second is a canonical map which corresponds to the unique irreducible representation of dimension $n$ of $\PSL_2(\R)$.
\end{Remark}

To strategy of the proof is to construct an involution on the Lie algebra $\g$ with fixed points the split real form. Using the uniqueness of the non-abelian Hodge correspondence, we can conclude that the monodromy is fixed by the involution.

\begin{proof}
Let $h$ be the harmonic metric on $E$ associated to $(E,\Phi)$. Denote by $\rho$ the compact real form associated to $h$, i.e. $\rho(M)=-M^{*_h}$. Let $\sigma$ be minus the ``skew-transpose'', i.e. which sends $e_{i,j}$ to $-e_{n+1-j,n+1-i}$ (where $e_{i,j}$ denote the standard matrix entries).
\begin{lemma}
The map $\tau=\rho \sigma=\sigma \rho$ is an involution corresponding to the split real form $\mathfrak{sl}_n(\R)$.
\end{lemma}
In a chart where $M^{*_h}=M^\dagger$, the map $\tau$ turns a matrix by 180 degree and conjugates all the entries. It is interesting to directly check that this is an Lie algebra anti-homomorphism (i.e. that $\tau([A,B]=-[\tau(A),\tau(B)]$). We refer to \cite[Prop.6.1]{hitchin} for a proof of the lemma.

By definition of $\Phi$, we have $\sigma(\Phi) = -\Phi$. Hence 
$$\tau(\Phi+\Phi^{*_h})=\tau(\Phi-\rho(\Phi)) = -\rho(\Phi)+\Phi = \Phi+\Phi^{*_h}.$$
Let $(\Phi,A)$ be the flat connection which solves Hitchin's equation. Then $(-\Phi,A)$ is also a solution. Further $(\sigma(\Phi),\sigma(A))=(-\Phi,\sigma(A))$ is again a solution. By uniqueness of the non-abelian Hodge correspondence, we get $\sigma(A)=A$. Thus,
$$\tau(A) = \rho(\sigma(A)) = \rho(A) = A.$$
Therefore the flat connection $\Phi+\nabla_A+\Phi^{*_h}$ is invariant under $\tau$, so its monodromy is in the split real form $\PSL_n(\R)$. A topological argument (dimension count and open-closed property) shows that we get a component of $\Rep(\pi_1\S,\PSL_n(\R))$.
\end{proof}

We can describe Hitchin's construction in the twistor picture: there is a map, the so-called \textbf{Hitchin fibration}, $\mc{M}_H\to \bigoplus_{i=2}^n H^0(S,K^i)$ given by $[(E,\Phi)]\mapsto \mathrm{det}(\Phi-t\id)$, i.e. the characteristic polynomial of the Higgs field. The coefficients are holomorphic differentials. The $(E,\Phi)$ considered above in Equation \eqref{Hitchin-section} is a section of the Hitchin fibration.

\begin{figure}
\begin{center}
\begin{tikzpicture}[scale=1.5]
	\draw (0,0) circle (1cm);
\draw [domain=180:360] plot ({cos(\x)},{sin(\x)/3});
	\draw [domain=0:180, dotted] plot ({cos(\x)},{sin(\x)/3});
	\draw [fill=white] (0,1) circle (0.04);
	\draw [fill=white] (0,-1) circle (0.04);
	\draw (0,-1.2) node {$\mc{M}_H$};  
	\draw (0,1.2) node {$\overline{\mc{M}}_H$};
	\draw (2.2,0.5) node {$\Rep(\pi_1\S,\PSL_n(\C))$};
	\draw (2.2,0.25) node {$\cup$};
	\draw (2.2,0) node {$\mathrm{Hit}(n,\S)$};
	\draw[->] ({1.25*cos(288)},{1.25*sin(288)}) arc (288:345:1.25);
	\draw (2.2,-0.75) node {non-abelian Hodge};
	\draw (2.2,-1) node {correspondence};
	\draw (0,-1.5) node {$\downarrow$};
	\draw (0,-1.9) node {$\bigoplus_{i=2}^n H^0(K^i)$};
	\draw (-1.25,-1.4) node {Hitchin};
	\draw (-1.25,-1.65) node {fibration};
	\draw [domain=-30:32, dashed, ->] plot ({0.33*cos(\x)}, {0.33*sin(\x)-1.5});
	\draw (1,-1.4) node {Hitchin};
	\draw (1,-1.62) node {section};
\end{tikzpicture}

\caption{Hitchin component via non-abelian Hodge correspondence}\label{hitchin-comp-twistor-picture}
\end{center}
\end{figure}
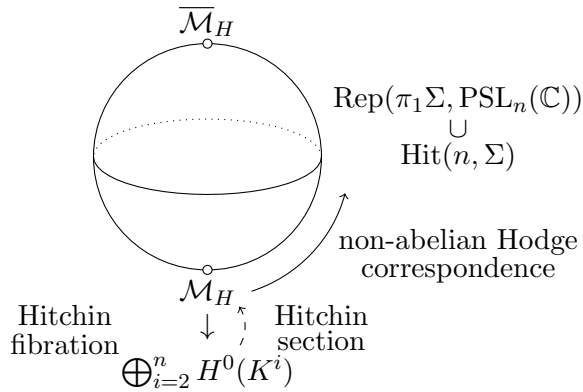

The image of the Hitchin section in the character variety under the non-abelian Hodge correspondence is inside the $\PSL_n(\R)$-character variety and forms the Hitchin component.

\section{Generalizations and research directions}\label{sec:generalizations}

There are various generalizations, some very active research projects, of the non-abelian Hodge correspondence. We present a selection here.

\medskip
\paragraph{Reductive groups $G$.}
We restricted mostly attention to $G=\GL_n(\C)$ in the previous chapters. Nearly all can be generalized to reductive groups (direct sum of simple and abelian Lie groups), typically subgroups of $\GL_n(\C)$. For the bundles, this means that its structure group is $G$, so the fibers carry some extra geometric structures, invariant under parallel transport. Some examples:
\begin{itemize}
	\item $G=\SL_n(\C)$: the fiber carries a fixed volume form,
	\item $G=\mathrm{SO}_n(\R)$: the fiber carries a scalar product,
	\item $G=\mathrm{Sp}_{2n}(\R)$: the fiber carries a symplectic structure.
\end{itemize}
For general $G$, the notion of a \emph{principal $G$-bundle} has to be used.

On the Higgs bundle side, there is the notion of a $G$-Higgs bundle for a complex Lie group $G$. The non-abelian Hodge correspondence reads
$$\mc{M}_H(S,G) \cong \Rep(\pi_1\S,G).$$

For real forms $G^\R$, there is an appropriate notion of a $G^\R$-Higgs bundle (which is still a holomorphic object!), developed in \cite{garcia2009hitchin}, giving $$\mc{M}_H(S,G^\R) \cong \Rep(\pi_1\S,G^\R).$$ This generalizes Hitchin's construction in the split real case.

With these techniques, the number of components of character varieties can be counted. The study of components consisting entirely of discrete and faithful representations (called higher Teichmüller components) is called \textbf{higher Teichmüller theory}. I recommend the nice introduction of Anna Wienhard \cite{wienhard2018invitation}.

Some active research tasks are:
\begin{itemize}
	\item Determine the topology (Betti numbers, Hilbert--Poincaré polynomial) of the components of the character variety for real groups $G^\R$.
	\item Find geometric structures whose moduli spaces are these components. In particular, find a geometric interpretation of the Hitchin components.
	\item Characterize and classify all higher Teichmüller components. A recent breakthrough is the notion of $\Theta$-positivity in \cite{guichard2018positivity}. See also \cite{bradlow2021general} for the Higgs bundle analog.
\end{itemize}

\medskip
\paragraph{Non-trivial complex bundles.}
In the case where the underlying complex vector bundle $E$ is not trivial, we cannot hope for the existence of flat connections. But we can get as close as possible: we can get \textbf{projectively flat} connections.

A flat connection allows to restrict the transition functions of the bundle to be constant. A projectively flat connection allows to restrict to homotheties (so the transition functions of the projectivized bundle are constant). The curvature of a projectively flat connection is a central element of $\Omega^2(S,\g)$, i.e. a constant multiple of the identity. To be more precise, on a holomorphic bundle $E$ over a Riemann surface with Kähler form $\omega$ (simply an area form in this case), the curvature of a projectively flat connection $\nabla$ is 
\begin{equation}\label{proj-flat-connection-curvature}
F(\nabla) = -2i\pi\mu(E) \omega \id.
\end{equation}

On the side of character varieties, a projectively flat connection on $E$ corresponds to a representation of a central extension of $\pi_1\S$, or equivalently to a twisted representation of $\pi_1\S$. The central extension $\widehat{\pi_1\S}$ is defined by
$$1\to \Z \to \widehat{\pi_1\S} \to \pi_1\S \to 1,$$
or in terms of a presentation by 
$$\widehat{\pi_1\S} = \langle a_i, b_i, c \mid 1\leq i\leq g, c\text{ central}, \textstyle\prod_i [a_i,b_i]=c\rangle.$$
A \emph{twisted representation} is a map $\rho:\pi_1\S \to G$, such that $\rho(\prod_i[a_i,b_i])=\zeta_n^d\id$ where $\zeta_n$ is a $n$-th root of unity and $n=\mathrm{rk}(E)$ and $d=\deg(E)$. We then have
\begin{align*}
\mc{M}_{H, \deg(E)=d}(S,G) \cong \Rep(\widehat{\pi_1\S},G) &= \{(a_i,b_i)\in G^{2g}\mid \textstyle\prod_i [a_i,b_i]=\zeta_n^d\id\}/G \\
&= \{\text{projectively flat connections}\}/\mc{G}.
\end{align*}

\medskip
\paragraph{Compact Kähler manifolds.}
The non-abelian Hodge correspondence can be generalized from compact Riemann surfaces to compact Kähler manifolds $X$. For this, Simpson \cite{simpson} generalized the notion of a Higgs bundle. Roughly speaking a Higgs field $\Phi$ is still a holomorphic section of $\End(E)\otimes K$ (where $K$ is the canonical bundle of $X$, i.e. the determinant bundle of $T^*X$), such that $\Phi\wedge\Phi=0$.

A first step is the generalized Narasimhan--Seshadri theorem, giving the equivalence between polystable holomorphic bundles and certain connections called \emph{hermitian Einstein connections}. These are connections $\nabla$ satisfying $$F(\nabla)\wedge \omega^{n-1}=\lambda(E)\omega^n\id$$where $\omega$ is the Kähler form on $X$ and $\lambda(E)$ is some constant depending on $E$. Note the similarity to Equation \eqref{proj-flat-connection-curvature}.

The \emph{Hitchin--Kobayashi correspondence} asserts that a holomorphic bundle over a compact Kähler manifold allows a hermitian Einstein connection iff it is polystable. In terms of moduli spaces: $$\Hol^{ps}(X) \cong \{\text{hermitian Einstein connections}\}/\mc{G}.$$

The second step is to add the Higgs field. The connections now have to satisfy $$\left(F(\nabla_A)+[\Phi,\Phi^*]\right)\wedge \omega^{n-1}=\lambda(E)\omega^n\id.$$ These are called \emph{hermitian Yang--Mills connections}. In the case the underlying complex vector bundle $E$ is trivial, we get the non-abelian Hodge correspondence for Kähler manifolds:
$$\mc{M}_H(X,G) \cong \Rep^{c.r.}(\pi_1X,G).$$
The moduli space of stable Higgs bundles is still hyperkähler under mild assumptions, see for example \cite{biswas2006geometry}.

\medskip
\paragraph{Parabolic Higgs bundles.}
Let $S_{g,n}$ be a Riemann surface of genus $g$ with $n$ marked points with underlying surface $\S_{g,n}$. We can consider meromorphic connections with simple poles at the marked points. When these marked points are considered as boundary components, we have seen that it is natural to fix the monodromy of a flat connection around each boundary component (to get a symplectic moduli space). 

The corresponding notion for holomorphic bundles is called a \emph{parabolic structure}, which roughly speaking is a bundle with fixed (partial) flags in the fibers over the marked points. The stable parabolic bundles correspond to unitary representations of $\pi_1\S_{g,n}$ (see \cite{mehta1980moduli}). Going further to the non-compact group $\GL_n(\C)$, there is the notion of a parabolic Higgs bundle and harmonic metrics with singularities at the marked points, such that
\begin{equation}\label{nahc-parabolic}
\mc{M}_{H}^{\text{para}}(S_{g,n},\GL_n(\C)) \cong \Rep(\pi_1\S_{g,n},\GL_n(\C)).
\end{equation}

Current research directions:
\begin{itemize}
	\item Find the correct notion of a parabolic principal $G$-Higgs bundle generalizing the diffeomorphism \eqref{nahc-parabolic}. To get an overview on different approaches, see for example the introduction of \cite{kydonakis2021tame}.
\end{itemize}

\medskip
\paragraph{Wild character varieties.}
In the same setting, one might consider meromorphic connections on $S_{g,n}$ with higher order poles at the marked points. For a pole of order at least 2, the monodromy around the marked point is not sufficient anymore to characterize the gauge class of the meromorphic connection. 

The extra data you need is called the \emph{Stokes data}. The corresponding character varieties of meromorphic connections and fixed Stokes data are called \emph{wild character varieties}. There are studied by Philip Boalch \cite{boalch2001symplectic} and give surprising links to quantum groups and integrable systems. In particular, there is a generalized Atiyah--Bott reduction, showing that the space of generalized monodromies (including the Stokes data) is a symplectic space.

\medskip
\paragraph{Non-holomorphic setting.}
Many deep results on the character variety, which depends only on the topology of the surface, are proven using holomorphic techniques by fixing a complex structure on $\S$, turning it into a Riemann surface $S$. Often, it turns out that the final results are independent of the complex structure chosen, but it is highly non-trivial to prove it.

One instance is the ``quantization'' of character varieties. We have seen that the character variety is symplectic. Fixing a complex structure on $\S$ gives a compatible complex structure on $\Rep(\pi_1\S,G)$, so it becomes a Kähler manifold. For Kähler manifolds, there is a procedure to quantize them, called \emph{geometric quantization}. This quantization does not depend on the complex structure, but this is not at all obvious.

Another situation where the rigidity of holomorphic structures is an obstacle is the following: there is a natural action of the mapping class group $\mathrm{MCG}(\S)=\mathrm{Diff}(\S)/\mathrm{Diff}_0(\S)$ on Teichmüller space, and also on Hitchin components (considered as deformations of representation of the form $\pi_1S\to \PSL_2(\R)\to\PSL_n(\R)$). This action is impossible to see in Hitchin's parametrization since the mapping class group changes the complex structure.

\begin{oq}
Is there a non-holomorphic approach to character varieties, linking them to geometric structures on smooth bundles or the surface itself? 
\end{oq}
In particular, this includes the question formulated earlier about the existence of a geometric structure on the surface whose moduli space is the Hitchin component.

Here is where my own research comes into play. In \cite{fock-thomas}, together with Vladimir Fock, we introduced the notion of a \emph{higher complex structure}, a geometric structure on a surface generalizing the complex structure. The main conjecture is that the moduli space of higher complex structures $\T^n$ is canonically diffeomorphic to the Hitchin component. Moreover, the cotangent bundle $\cotang$ should be diffeomorphic to an open dense subset of $\Rep(\pi_1\S,\SL_n(\C))$ which looks very similar to the non-abelian Hodge correspondence.

Let me give a flavor of this conjectural landscape: a higher complex structure (of order $n$) induces a bundle $V$ of rank $n$ over $\S$ together with a matrix-valued 1-form $\Phi=\Phi_1dx+\Phi_2dy$ where $\Phi_1$ and $\Phi_2$ are two commuting nilpotent matrices. This looks a bit similar to a Higgs bundle, but $\Phi$ is not holomorphic here, but nilpotent. A point in the cotangent bundle $\cotang$ corresponds to a deformation of $\Phi$ away from the nilpotent locus.

\begin{conj}
The space $\cotang$ has a hyperkähler structure near the zero-section. All fibers in the twistor space, apart from two, are diffeomorphic to an open subset $U$ of $\Rep(\pi_1\S,\SL_n(\C))$ and the zero-section corresponds to the Hitchin component.
\end{conj}

The twistor picture can be drawn as follows (see \cite[Chapter 10]{thomas}):

\begin{center}
\begin{tikzpicture}[scale=1.5]
	\draw (0,0) circle (1cm);
\draw [domain=180:360] plot ({cos(\x)},{sin(\x)/3});
	\draw [domain=0:180, dotted] plot ({cos(\x)},{sin(\x)/3});
	\draw [fill=white] (0,1) circle (0.04);
	\draw [fill=white] (0,-1) circle (0.04);
	\draw (0,-1.2) node {$\cotang$};  
	\draw (0,1.2) node {$\overline{\cotang}$};
	\draw (2.3,0.5) node {$U\subset \Rep(\pi_1\S,\PSL_n(\C))$};
	\draw (2.2,0.15) node {$\cup$};
	\draw (2.2,-0.1) node {$\mathrm{Hit}(n,\S)$};
	\draw[->] ({1.25*cos(288)},{1.25*sin(288)}) arc (288:345:1.25);
	\draw (0,-1.5) node {$\downarrow$};
	\draw (0,-1.9) node {$\T^n$};
	\draw (-1.25,-1.4) node {canonical};
	\draw (-1.25,-1.65) node {projection};
	\draw [domain=-30:32, dashed, ->] plot ({0.33*cos(\x)}, {0.33*sin(\x)-1.5});
	\draw (1,-1.4) node {zero-};
	\draw (1,-1.62) node {section};
\end{tikzpicture}
\end{center}

Notice the similarities with Figure \ref{hitchin-comp-twistor-picture}, the twistor picture which constructs the Hitchin component. In particular, the analog of the Hitchin fibration is much simpler: it is just the projection map, and the Hitchin section is the natural inclusion $\T^n\subset \cotang$. The conjecture that the zero-section corresponds to the Hitchin component can be seen as an analog of the Narasimhan--Seshadri theorem: for vanishing cotangent vector, the monodromy of the flat connection reduces to a real form. In our case, we get the split real form, while for Higgs bundles we get the compact real form.

\begin{Remark}
A good reason to believe in the conjecture is the following: Hitchin's component has Goldman's symplectic structure and the moduli space of higher complex structure carries a natural complex structure. If both combine to a Kähler structure, then there is a HK structure near the zero-section of $\cotang$ by the Feix--Kaledin theorem \ref{feix-kaledin-thm}.
\end{Remark}

For $n=2$, the situation is well-understood: higher complex structures of order 2 are nothing but usual complex structures, so $\T^2=\mathrm{Teich}(\S)$. Further $T^*\T^2$ is the space of complex projective structures and the map to $\Rep(\pi_1\S,\SL_2(\C))$ is given by the monodromy of the $\C P^1$-structure. Its image is an open dense subset proven in \cite{gallo2000monodromy}.

The question about what kind of geometric structures represent points in $\cotang$, a generalization of complex projective structures, is open.

\medskip
\paragraph{Other directions.}
A non-exhaustive list:
\begin{itemize}
	\item Quantization of character varieties: the celebrated \emph{Verlinde formula} gives the dimension of the Hilbert spaces associated to the geometric quantization. The link to quantum gravity is lurking.
	\item $p$-adic non-abelian Hodge correspondence: finding an analog in the $p$-adic world is very challenging and linked to the \emph{geometric Langlands program}. To give one example, Ngô's proof of the Fundamental Lemma uses the Hitchin fibration.
\end{itemize}


\bibliographystyle{alpha}
\bibliography{ref}

\end{document}